\documentclass[11pt,a4paper]{amsart}
\usepackage{upref}
\usepackage{enumitem,dsfont,amssymb,bm}
\usepackage{graphicx,color}
\usepackage{mathrsfs}
\usepackage[colorlinks=true, pdftitle={Fractional powers of monotone
  operators}, pdfauthor={Daniel Hauer, Yuahn He, Dehui Liu}]{hyperref}

\title[Fractional powers of monotone operators]{Fractional powers of monotone operators\\ in Hilbert spaces}

\author{Daniel Hauer}
\address[Daniel Hauer, Yuhan He, and Dehui Liu]{The University of
  Sydney, School of Mathematics and Statistics, NSW 2006, Australia}
\email{\href{mailto:daniel.hauer@sydney.edu.au}{\nolinkurl{daniel.hauer@sydney.edu.au}}}
\author{Yuhan He}
\email{\href{mailto:yuhe0889@uni.sydney.edu.au}{\nolinkurl{yuhe0889@uni.sydney.edu.au}}}
\author{Dehui Liu}
\email{\href{mailto:dliu5892@uni.sydney.edu.au}{\nolinkurl{dliu5892@uni.sydney.edu.au}}}

\thanks{The results presented in this paper are the outgrows of a
  summer vacation research project at the University of Sydney under
  the supervision of Mr Hauer. Ms He and Mr Liu are undergraduate students at the University
of Sydney who participated at this research project and actively contributed.}



\numberwithin{equation}{section}

\theoremstyle{theorem}
\newtheorem{theorem}{Theorem}[section]
\newtheorem{proposition}[theorem]{Proposition}
\newtheorem{lemma}[theorem]{Lemma}
\newtheorem{corollary}[theorem]{Corollary}

\theoremstyle{definition}
\newtheorem{definition}[theorem]{Definition}

\theoremstyle{remark}
\newtheorem{remark}[theorem]{Remark}

\newcommand\R{{\mathbb{R}}}

\newcommand\C{{\mathbb{C}}}
\newcommand\E{\mathcal{E}}

\newcommand\A{\mathcal{A}}
\newcommand\WE{\mathcal{W}}

\newcommand\td{\mathrm{d}}
\newcommand\dx{\mathrm{d}x }

\newcommand\dr{\mathrm{d}r }
\newcommand\dmu{\mathrm{d}\mu }
\newcommand\dz{\mathrm{d}z }
\newcommand\ds{\mathrm{d}s }
\newcommand\dt{\mathrm{d}t }

\DeclareMathOperator*{\divergence}{div}

\def\1{\raisebox{2pt}{\rm{$\chi$}}}

\newcommand\abs[1]{\lvert#1\rvert}
\newcommand\labs[1]{\left\lvert#1\right\rvert}
\newcommand\norm[1]{\lVert#1\rVert}
\newcommand\lnorm[1]{\left\lVert#1\right\rVert}

\definecolor{darkred}{rgb}{0.7,0.1,0.1}

\begin{document}
\date{\today}
\maketitle

\tableofcontents
%

\begin{abstract}
  In this article, we show that if $A$ is a maximal monotone
    operator on a Hilbert space $H$ with $0$ in the range $\textrm{Rg}(A)$ of
  $A$, then for every $0<s<1$, the Dirichlet problem associated
  with the Bessel-type equation
  \begin{displaymath}
    A_{1-2s}u:=
    -\frac{1-2s}{t}u_{t}-u_{tt}+Au\ni 0
  \end{displaymath}
  is well-posed for boundary values
  $\varphi\in \overline{D(A)}^{\mbox{}_{H}}$. This allows us to define
  the Dirichlet-to-Neumann (DtN) operator $\Lambda_{s}$ associated with $A_{1-2s}$ as
  \begin{displaymath}
    \varphi\mapsto \Lambda_{s}\varphi:=-\lim_{t\to 0+}t^{1-2s}u_{t}(t)\qquad\text{in H.}
  \end{displaymath} 
  The existence of the DtN operator $\Lambda_{s}$ associated with
  $A_{1-2s}$ is the first step to define fractional powers
  $A^{\alpha}$ of monotone (possibly, nonlinear and multivalued)
  operators $A$ on $H$. We prove that $\Lambda_{s}$ is monotone on $H$
  and if $\overline{\Lambda}_{s}$ is the closure of $\Lambda_{s}$ in
  $H\times H_{w}$ then we provide sufficient conditions implying that
  $-\overline{\Lambda}_{s}$ generates a strongly continuous semigroup
  on $\overline{D(A)}^{\mbox{}_{H}}$. In addition, we show that if $A$
  is completely accretive on $L^{2}(\Sigma,\mu)$ for a $\sigma$-finite
  measure space $(\Sigma,\mu)$, then $\Lambda_{s}$ inherits this
  property from $A$. 
%
\end{abstract}

\section{Introduction and main results}

In the pioneering work~\cite{MR2354493}, Caffarelli and
Silvestre constructed three analytical proofs to show that the fractional Laplacian
$(-\Delta)^{s}$, ($0<s<1$), on $\R^{d}$, ($d\ge 1$),
coincides up to a multiple constant with the Dirichlet-to-Neumann (DtN) operator 
\begin{displaymath}
\varphi\mapsto \Lambda_{s}\varphi:=\lim_{t\to
  0+}-t^{1-2s}u'(\cdot,t)\qquad \text{on $\R^{d}$}
\end{displaymath}
associated with the 
Bessel type equation
\begin{equation}
  \label{eq:79}
  A_{1-2s}u:=
  -\frac{1-2s}{t}u_{t}-u_{tt}+Au\ni 0
\end{equation}
on the half-space $\R^{d+1}_{+}:=\R^{d}\times (0,+\infty)$ with
$A=-\Delta$. One crucial \emph{Ansatz} in~\cite{MR2354493} to obtain
this identification is to employ the change of variable
\begin{equation}
  \label{eq:15}
  z=\left(\frac{t}{2s}\right)^{2s},
\end{equation}
which transforms~\eqref{eq:79} into the equation 
\begin{equation}
  \label{eq:80}
  \tilde{A}_{1-2s}v:=-z^{-\frac{1-2s}{s}}v''+Av\ni 0.
\end{equation}

The identification of the fractional Laplacian $(-\Delta)^{s}$ with
the DtN operator $\Lambda_{s}$ has far reaching applications; for example,
in the study of nonlocal partial differential equations, it provides
new boundary regularity results including boundary Harnack
inequalities, a monotonicity formula, and several others (for
example, cf~\cite{MR3056307} or~\cite{MR3772192} and the references theirin). The
ideas in~\cite{MR2354493} were followed up quickly by many authors and
extended to several more abstract \emph{linear} settings. For
instance, if for $\alpha\ge 0$, $\mathcal{I}_{\alpha}$ denotes the
family of linear operators $A$ on a Banach space $X$ generating a
\emph{tempered $\alpha$-times integrated semigroup
  $\{T(t)\}_{t\ge 0}$ in $\mathcal{L}(X)$} (for details
cf~\cite{MR3056307}), then it was shown by Gal\'e, Miana and
Stinga~\cite{MR3056307} that for $A\in \mathcal{I}_{\alpha}$ and
$0<s<1$, the \emph{Dirichlet problem associated with the Bessel-type operator $A_{1-2s}$},
\begin{equation}
    \label{eq:1bis}
    \begin{cases}
      A_{1-2s}u(t) \ni 0& \text{for
        almost every }t>0,\\
     \mbox{}\hspace{0.95cm} u(0)=\varphi &
    \end{cases}
  \end{equation}
  admits a unique solution $u$ and
  the fractional power operator $A^{s}$ (in the Balakrishnan
  sense~\cite{MR0115096}) coincides with the DtN operator $\Lambda_{s}$
  associated with the Bessel-type operator $A_{1-2s}$ given
  by~\eqref{eq:79}. 

  An important subclass of $\mathcal{I}_{\alpha}$ is given by the
  family of \emph{$m$-sectorial operators} $A$ defined on a Hilbert
  space $H$ which are the restriction on $H$ of a continuous,
  coercive, sesquilinear form $\mathfrak{a} : V\times V\to \C$. Here,
  $V$ is another Hilbert space such that $V$ is embedded into $H$ by a
  linear bounded injection $j : V\hookrightarrow H$
  with a dense image $j(V)$ in $H$. For this subclass of operators $A$
  on $H$, the identification of the fractional power $A^{s}$ with the
  DtN operator $\Lambda_{s}$ was recently revisited by
  Arendt, Ter Elst and Warma~\cite{MR3772192}. By
  using techniques from linear interpolation theory, they could
  characterise the domain $D(\Lambda_{s})$ of $\Lambda_{s}$ and
  establish an imporant \emph{integration by parts} rule for solutions
  to the Dirichlet problem associated with $A_{1-2s}$. 

  Note, it is for the class of $m$-sectorial
  operators $A$ associated with a real symmetric form
  $\mathfrak{a} : V\times V\to \R$ that the \emph{linear} semigroup theory coincides with the
  \emph{nonlinear} semigroup theory 
  (cf~\cite{MR0348562} or \cite{MR3465809}). 
 We briefly recall from~\cite{MR0348562} (see
  also~\cite{MR2582280}), a (possibly multivalued) operator
  $A : D(A)\to 2^{H}$ is called \emph{maximal monotone} in $H$, if $A$
  satisfies the \emph{monotonicity property}: 
  $(\hat{v}-v,\hat{u}-u)_{H}\ge 0$ for all
    $(u,v)$, $(\hat{u},\hat{v})\in A$,
and the so-called \emph{range condition}: $Rg(I+A)=H$. Here, we follow
the standard notation 
and usually identify an operator $A$ with its \emph{graph}
\begin{displaymath}
  A=\Big\{(u,v)\in H\times H\,\Big\vert\, v\in Au\Big\}\quad
  \text{in $H\times H$.}
\end{displaymath}
The set $D(A):=\{u\in H\,\vert\,Au\neq \emptyset\}$ is called the
\emph{domain} of $A$ and
$\textrm{Rg}(A):=\bigcup_{u\in D(A)}Au\;\subseteq H$ the \emph{range} of $A$.

In the framework of maximal monotone operators $A$ in
Hilbert spaces $H$, first results toward the fractional power
$A^{1/2}$ were obtained 
by Barbu \cite{MR0331133}. More precisely, if
$0\in \textrm{Rg}(A)$, then Barbu proved that the Dirichlet
problem~\eqref{eq:1bis} associated with the Bessel-type operator
$A_{0}$ is well-posedness, and then introduced the semigroup
$\{T_{1/2}(t)\}_{t\ge 0}$ as the contractive extension on the
closure $\overline{D(A)}^{\mbox{}_{H}}$ of
$T_{1/2}(t)\varphi:=u(t)$, ($t\ge 0$), for the unique solution
$u$ of the Dirichlet problem~\eqref{eq:1bis} with boundary value
$\varphi\in D(A)$. 
The paper~\cite{MR0331133} was quickly followed up by several authors.
For example, Brezis~\cite{MR0317123} provided a different proof of the
well-posedness of Dirichlet problem~\eqref{eq:1bis}, and
V\'eron~\cite{MR0493540} established well-posedness of the following
more general (but \emph{non-singular}) Dirichlet problem
\begin{displaymath}
  -p(t) u_{tt}-q(t) u_{t}+A u\ni 0,\qquad u(0)=\varphi 
\end{displaymath}
for $\varphi\in D(A)$ under the hypothesis that $p\in
W^{2,\infty}(0,+\infty)$, $q\in W^{1,\infty}(0,+\infty)$, and there is
an $\alpha>0$ satisfying $p(t)\ge \alpha>0$ for all $t\ge 0$.
We emphasises that neither Barbu nor Brezis identified the (negative)
infinitessimal generator of the semigroup
$\{T_{1/2}(t)\}_{t\ge 0}$ by the DtN operator $\Lambda_{1/2}$
associated with $A_{0}$, but in stead denoted $\Lambda_{1/2}$ by
$A_{1/2}$ in analogy to the fractional power $A^{1/2}$ of $A$ as it is
known from the linear theory (cf~\cite{MR0115096} or
\cite[Chapter~6.3]{MR1850825}). 

More than twenty years later, Alraabiou \& B\'enilan~\cite{MR1380576}
introduce the definition of the \emph{square power}
\begin{displaymath}
  A^2:=\liminf_{\lambda\to 0+}\frac{A-A_{\lambda}}{\lambda},
\end{displaymath}
via the Yosida approximation
$A_{\lambda}:=(I-(I+\lambda A)^{-1})/\lambda$ and showed that if
$A=\partial_{H}\E$ the subdifferential operator in $H$ of a convex,
proper, lower semicontinuous functional $\E : H\to [0,+\infty]$ with
$0\in \partial_{H} \E(0)$, then $A\subseteq (A_{1/2})^{2}$. This
justifies (at least in the case $s=1/2$) that the DtN operator
$\Lambda_{s}$ is a reasonable candidate for defining fractional powers
$A^{s}$ of maximal monotone operators $A$ on Hilbert spaces.\medskip



The first aim of this paper is to extend the results obtained by
Barbu~\cite{MR0331133} and Brezis~\cite{MR0317123} (in the case
$s=1/2$) to the complete range $0<s<1$, the well-posedness of Dirchlet
problem~\eqref{eq:1bis} associated with Bessel-type operator
$A_{1-2s}$ for (possibly nonlinear) \emph{maximal monotone operators}
$A$ on a Hilbert space $H$. 
Before stating our first main result, let us introduce the following
notion of solutions of Dirchlet problem~\eqref{eq:1bis}.



  \begin{definition}
    We call a function $u : [0,+\infty)\to H$ a \emph{strong solution}
    of Bessel-type equation~\eqref{eq:79} if
    $u\in W^{2,2}_{loc}((0,+\infty);H)$ and for almost every $t>0$,
    $u(t)\in D(A)$ and $\{t^{1-2s}u'(t)\}'\in
    t^{1-2s}A(u(t))$. Moreover, for given boundary value
    $\varphi\in H$, a function $u : [0,+\infty)\to H$ is called a
    \emph{solution of Dirichlet problem}~\eqref{eq:1bis} if
    $u\in C([0,+\infty);H)$, $u(0)=\varphi$, and $u$ is a strong
    solution of~\eqref{eq:79}.
  \end{definition}
  
  In the next theorem, we write $L^{\infty}(H)$ to denote
  $L^{\infty}(0,+\infty;H)$, $L^{2}_{\star}(H)$ for the $L^{2}$-space
  of all measurable functions $u : \R_{+}\to H$ where
  $\R_{+}:=(0,+\infty)$ is equipped with the Haar-measure
  $\tfrac{\dt}{t}$, and $W^{1,2}_{s}(H)$ to denote the set of all $u\in W^{1,1}_{loc}(H)$
  such that $t^{s}u$ and $t^{s}u'\in L^{2}_{\star}(H)$ (for more
  details, see~Section~\ref{sec:pre}).
%
%
With these preliminaries, our first main result reads as follows.

\begin{theorem}
  \label{thm:1}
  Let $A$ be a maximal monotone operator on $H$
  with $0\in \textrm{Rg}(A)$. Then, for every $0<s<1$, 
  $\varphi\in \overline{D(A)}^{\mbox{}_{H}}$ and $y\in A^{-1}(\{0\})$, there is a unique solution
  $u\in L^{\infty}(H)$ of Dirichlet problem~\eqref{eq:1bis}
  satisfying \allowdisplaybreaks
  \begin{align}
    \label{eq:87A}
     \norm{u(t)-y}_{H} & \le \norm{u(\hat{t})-y}_{H}\qquad\text{for
       all $t\ge \hat{t}\ge 0$,}\\
    \label{eq:81}
    \norm{t u'}_{L^{2}_{\star}(H)} &\le
                                     \sqrt{s}\;\norm{\varphi-y}_{H}\\
    \norm{u'(t)}_{H} &\le
                       2 s\frac{\norm{\varphi-y}_{H}}{t}\quad
    \text{for every $t>0$,}\\[7pt]
    \label{eq:82}
    \lnorm{t^{1+2s}\{t^{1-2s}u'\}'}_{L^{2}_{\star}(H)} & \le
    \begin{cases}
          \sqrt{s}\;\norm{\varphi-y}_{H}
          & \text{if $s\ge  \frac{1}{2}$,}\\[7pt]
            \frac{\sqrt{s}\left(\frac{s}{1-2s}
              \frac{1}{2}+3\right)^{\frac{1}{2}}}{\sqrt{2}}\,\norm{\varphi-y}_{H}
         & \text{if $0<s<\frac{1}{2}$.}
    \end{cases}
  \end{align}
  Moreover, if $\varphi\in D(A)$, then 
  $t^{1-2s}u'\in W^{1,2}_{s}(H)$ and 
  \begin{align}
      \label{eq:30bis2A}
        \lnorm{\displaystyle\lim_{t\to 0+}t^{1-2s}u'(t)}_{H} &\le (2s)^{1-2s}
        \left(\norm{A^{0}\varphi}_{H}^{\frac{1}{2}}+\norm{\varphi-y}_H^{\frac{1}{2}}\right)^{2},\\ 
     \label{eq:14bis2A}
      \norm{t^{s}\{t^{1-2s}u'\}}_{L^{2}_{\star}(H)} &\le
                                                      (2s)^{1-s}\left(\norm{A^{0}\varphi}_{H}^{\frac{1}{2}} 
                                                      +\norm{\varphi-y}_{H}\right),\\
     \label{eq:31bis2A}
     \norm{t^{s}\{t^{1-2s}u'\}'}_{L^{2}_{\star}(H)}
   &\le (2s)^{1-s}\left(\norm{A^{0}\varphi}_{H}
     +\norm{\varphi-y}_H^{\frac{1}{2}}\,
     \norm{A^{0}\varphi}_{H}^{\frac{1}{2}}\right).
    \end{align}  
  In particular, for every two solutions
  $u$ and $\hat{u}$ of~\eqref{eq:1bis} with boundary value $\varphi$
  and $\hat{\varphi}\in \overline{D(A)}$, one has that
  \begin{equation}
    \label{eq:83}
    \norm{u(t)-\hat{u}(t)}_{H}\le \norm{u(\hat{t})-\hat{u}(\hat{t})}_{H}
    \qquad\text{for all $t\ge \hat{t}\ge 0$.}
  \end{equation}
\end{theorem}


According to Theorem~\ref{thm:1}, for every $\varphi\in D(A)$, the
\emph{outer unit normal derivative}
\begin{equation}
  \label{eq:94}
 \Lambda_{s}\varphi:= -\lim_{t\to 0+}t^{1-2s}u'(t)\qquad\text{exists in $H$.}
\end{equation}
Thus, thanks to this well-posedness result, the DtN operator
$\Lambda_{s}$ assigning each Dirichlet boundary condition
$\varphi\in D(A)$ to the Neumann derivative~\eqref{eq:94} of the
unique solution $u$ of~\eqref{eq:79} is a well-defined mapping
$\Lambda_{s} : D(A)\to H$. More precisely, we can state the following
theorem. Here, $H_{w}$ denotes the space $H$ equipped with the weak
toplogy $\sigma(H,H')$.


\begin{theorem}
  \label{thm:DtN}
  Let $A$ be a maximal monotone operator on $H$ with 
  $0\in \textrm{Rg}(A)$. Then, for every $0<s<1$, the
  Dirichlet-to-Neumann operator
  \begin{displaymath}
    \Lambda_{s}:=\Bigg\{(\varphi,w)\in H\times H\,\Bigg\vert
    \begin{array}[c]{l}
      \exists \text{ a strong solution $u$ of~\eqref{eq:79} satisfying}~\eqref{eq:87A},\\
      u(0)=\varphi\text{ in $H$, \& }w=-\displaystyle\lim_{t \to 0+} t^{1-2s}u'(t)\text{ in $H$}
    \end{array}
    \Bigg\}
  \end{displaymath}
  is a monotone, well-defined mapping $\Lambda_{s} : D(\Lambda_{s})
  \to H$ satisfying
  \begin{displaymath}
    D(A)\subseteq D(\Lambda_{s}) \subseteq\overline{D(A)}^{\mbox{}_{H}}\quad\text{ and }\quad
    D(A)\subseteq \textrm{Rg}(I_{H}+\lambda\, \Lambda_{s})\qquad\text{for all $\lambda>0$.}
  \end{displaymath}
  If $D(A)$ is dense in $H$, then closure $\overline{\Lambda}_{s}$ of
  $\Lambda_s$ in $H\times H_{w}$ is maximal monotone and
  $-\overline{\Lambda}_{s}$ generates a strongly continuous semigroup
  $\{T_{s}(t)\}_{t\ge 0}$ of contractions $T_{s}(t) : H\to
  H$. Moreover, for every $\varphi\in H$, there is a function $U(r,t)$
  satisfying
  \begin{displaymath}
      \begin{cases}
      -\frac{1-2s}{r}U_{r}(r,t)-U_{rr}(r,t)+AU(r,t) \ni 0& \text{for
        a.e. $r>0$, all $t>0$,}\\
     \mbox{}\hspace{4.76cm} U(0,t)=T_{s}(t)\varphi & \text{for all
     }t\ge 0,\\
     \mbox{}\hspace{2.86cm} \displaystyle\lim_{r \to 0+}
     r^{1-2s}U_{r}(r,t)\in \frac{\td}{\dt}T_{s}(t)\varphi &\text{for all
     }t>0.
    \end{cases}
  \end{displaymath}
\end{theorem}

Due to Theorem~\ref{thm:DtN} and the
results~\cite{MR0331133,MR0317123,MR1380576} (for $s=1/2$), it makes
sense to define the fractional power $A^s$ of $A$ for $0<s<1$ via the DtN operator
$\Lambda_{s}$.

\begin{definition}
  For a maximal monotone operator $A$ on $H$ with
  $0\in \textrm{Rg}(A)$, the \emph{fractional
    power $A^{s}$ of $A$} for $0<s<1$, is defined by
  \begin{displaymath}
    A^{s}=\overline{\Lambda}_{s},
  \end{displaymath}
  where $\overline{\Lambda}_{s}$ denotes the closure in $H\times H_{w}$ of
  the DtN operator $\Lambda_{s}$ associated with the Bessel-type
  operator $A_{1-2s}$ defined in Theorem~\ref{thm:DtN}.
\end{definition}

  It is the task of our forthcoming work in this direction
  to extend the definition of $A^2$ provided by Alraabiou \&
  B\'enilan~\cite{MR1380576} to the case $A^{k}$ for every integer
  $k\ge 2$ and to show that $A\subseteq (A^{\frac{1}{k}})^{k}$.\medskip

In the case the Hilbert space $H=L^{2}(\Sigma,\mu)$ of a
$\sigma$-finite measure space $(\Sigma,\mu)$, and the operator $A$ is
\emph{completely accretive} on $L^{2}(\Sigma,\mu)$ then this property
also holds for the DtN operator $\Lambda_{s}$ for $0<s<1$.  For the
notion of \emph{completely accretive} operators, \emph{order
  preservability}, and \emph{Orlicz spaces} $L^{\psi}(\Sigma,\mu)$, we
refer to Definitions~\ref{def:completely-accretive}
and~\ref{def:orlicz} in the subsequent section.

\begin{theorem}
\label{thm:DtNinterpol}
Let $A$ be an $m$-completely accretive operator on the Hilbert space
$H=L^{2}(\Sigma,\mu)$ of a $\sigma$-finite measure space
$(\Sigma,\mu)$, and $0\in \textrm{Rg}(A)$. Then for every $0<s<1$, the
DtN operator $\Lambda_{s}$ is also completely accretive on
$L^{2}(\Sigma,\mu)$. In particular, if $\{T_{s}(t)\}_{t\ge 0}$ is the
semigroup generated by $-\overline{\Lambda}_{s}$ on $L^{2}(\Sigma,\mu)$, then
$\{T_{s}(t)\}_{t\ge 0}$ is order-preserving and every map $T_{s}(t)$
is $L^{\Psi}$-contractive on $L^{2}(\Sigma,\mu)$ for every
right-continuous $N$-function $\psi$.
\end{theorem}

As a byproduct of the theory developed in this paper, we obtain
existence and uniqueness of the following abstract \emph{Robin problem
  associated with the Bessel-type operator $A_{1-2s}$},
\begin{equation}
    \label{eq:1Robin}
    \begin{cases}
      \mbox{}\hspace{2.6cm} A_{1-2s}u(t) \ni 0& \text{in $H$ for
        almost every }t>0,\\
     -\displaystyle\lim_{t \to 0+} t^{1-2s}u'(t) +\lambda u(0)=\varphi
     &\text{on $H$,}     
    \end{cases}
\end{equation}
for any $\lambda>0$ and $\varphi\in D(A)$. Here, we use the following
notion of solutions of problem~\eqref{eq:1Robin}.

\begin{definition}
  For given $\varphi\in H$, a function $u : [0,+\infty)\to H$ is
  called a \emph{solution} of Robin problem~\eqref{eq:1Robin} if
  $u \in C([0,+\infty);H)$ is a strong solution of~\eqref{eq:79} and
  \begin{displaymath}
    -\displaystyle\lim_{t \to 0+} t^{1-2s}u'(t)+\lambda u(0)=\varphi\qquad
    \text{exists in $H$.}
    \end{displaymath}
\end{definition}

Now, our well-posedness result to the abstract Robin
problem~\eqref{eq:1Robin} reads as follows.


\begin{theorem}
  \label{thm:2Robin} 
  Let $A$ be a maximal monotone operator on $H$ with
  $0\in \textrm{Rg}(A)$. Then, for every $0<s<1$, $\lambda>0$, and
  $\varphi\in D(A)$, there is a unique solution $u\in L^{\infty}(H)$
  of Robin problem~\eqref{eq:1Robin} satisfying~\eqref{eq:87A},
  \eqref{eq:30bis2A}-\eqref{eq:31bis2A} and
  \eqref{eq:81}-\eqref{eq:82}. In particular,
  $t^{1-2s}u'\in W^{1,2}_{s}(H)$, the function
  $t\mapsto \norm{u(t)}_{H}^{2}$ is convex, bounded and decreasing on
  $[0,+\infty)$, and for every two solutions $u$ and $\hat{u}$
  of~\eqref{eq:1Robin} respectively with boundary value $\varphi$ and
  $\hat{\varphi}\in D(A)$, one has that~\eqref{eq:83} holds.
\end{theorem}

\begin{remark}
  At the moment, we neither know how to obtain existence and
  uniqueness of solutions to the abstract Robin
  problem~\eqref{eq:1Robin} for boundary data
  $\varphi\in \overline{D(A)}^{\mbox{}_{H}}$ nor the continuous
  dependence of the solutions $u$ of~\eqref{eq:1Robin} on the boundary
  data $\varphi$. This result would not only complete the theory on
  studying the Robin problem, but also imply that the DtN map
  $\Lambda_{s}$ restricted on the closure
  $\overline{D(A)}^{\mbox{}_{H}}$ is a well-defined mapping.
\end{remark}

This paper is organized as follows. In the subsequent section, we
introduce the framework and notations used throughout this paper and
state some preliminary results which will be useful for establishing
existence and uniqueness of Dirichlet problem~\eqref{eq:1bis} and
Robin problem~\eqref{eq:1Robin}. In Section~\ref{sec:exuniqueness}, we
then prove the existence and uniqueness of equation~\eqref{eq:79}
equipped with a more abstract boundary-value problem for boundary data
$\varphi\in D(A)$, which includes~\eqref{eq:1bis}
and~\eqref{eq:1Robin} as special cases (cf problem~\eqref{eq:1}). The
well-posedness of Dirichlet problem~\eqref{eq:1bis} for boundary data
$\varphi\in \overline{D(A)}^{\mbox{}_{H}}$ follows from the
inqualities~\eqref{eq:87A}-\eqref{eq:82} and~\eqref{eq:83}. We provide
a short proof of Theorem~\ref{thm:1} and Theorem~\ref{thm:2Robin} at
the beginning of Section~\ref{sec:exuniqueness}. Our proof of
Theorem~\ref{thm:1bis} is based on the change of
variable~\eqref{eq:15} which transforms equation~\eqref{eq:79}
into~\eqref{eq:80}. We emphasize that our proof demonstrates very well
that the generalization of~\cite{MR0317123} (case $s=1/2$) to the full
range $0<s<1$ is not trivial and uses techniques and spaces from
interpolation theory. The statements of Theorem~\ref{thm:DtN} follow
from Corollary~\ref{cor:1}. In Section~\ref{sec:interpol}, we show
that if for a given convex, proper, lower semicontinuous functional
$\phi : H\to \R\cup\{+\infty\}$, the operator $A$ is
$\partial_{H}\phi$-monotone on $H$ (for a definition see below in the
next section), then the DtN operator $\Lambda_{s}$ has also this
property (see~Theorem~\ref{thm:4}). This results has several
applications; one can deduce invariance principles
(cf~\cite{MR0348562}), comparison principles and
$L^{\psi}$-contractivity properties of the semigroup
$\{T_{s}(t)\}_{t\ge 0}$ generated by $-\overline{\Lambda}_{s}$ on
$L^{2}(\Sigma,\mu)$, for a $\sigma$-finite measure space
$(\Sigma,\mu)$ (see Corollary~\ref{cor:5}). Here, $L^{\psi}$
abbreviates the Orlicz space $L^{\psi}(\Sigma,\mu)$ for a given
$N$-function $\psi$. Thus, the statement of
Theorem~\ref{thm:DtNinterpol} follows from Corollary~\ref{cor:5}. We
conclude this paper with an application on the Leray-Lions operator
$A=-\divergence (a(x,\nabla u))$ (see Section~\ref{sec:app}).

%
%
%
%

\section{Preliminaries}
\label{sec:pre}
Throughout this article, $(H,(\cdot,\cdot)_{H})$ denotes a real
Hilbert space with inner product $(\cdot,\cdot)_{H}$, and we use to
write $\R_+$ to denote  $(0, +\infty)$ and
$\overline{\R}_+ := [0, +\infty]$.

For an operator $A$ on $H$, the \emph{minimal selection} $A^{0}$ of $A$ is
given by
\begin{displaymath}
  A^0 := \Big\{(u,v) \in A\,\Big\vert \,
  \norm{v}_{H}=\displaystyle\min_{w\in Au}\norm{w}_{H} \Big\}
\end{displaymath}
and for a convex functional $\E : H \to
\R\cup\{+\infty\}$, the \emph{subdifferential operator}
\begin{displaymath}
  \partial_{H}\E:=
  \Big\{(u,h)\in H\times H\,\Big\vert\, \E(u+v)-\E(u)\ge
  (h,v)_{H}\text{ for all }v\in H\Big\}.
\end{displaymath}
The property that $A$ being monotone is equivalent to that
for every $\lambda>0$, the \emph{resolvent operator}
$J_{\lambda}^{A}:=(I+\lambda A)^{-1}$ of $A$ is \emph{contractive in $H$}:
\begin{displaymath}
 \norm{J_{\lambda}^{A}u-J_{\lambda}^{A}\hat{u}}_{H}\le \norm{u-\hat{u}}_{H}
\qquad\text{for every $u$, $\hat{u}\in \textrm{Rg}(I+\lambda A)$.}
\end{displaymath}

Further, we call a functional $j : H \to \overline{\R}_+$ 
\emph{strongly coercive} if
\begin{equation}
  \label{eq:47}
  \lim_{|v| \rightarrow +\infty} \frac{j(v)}{|v|} = +\infty.
\end{equation}
If $j : H \to \R\cup \{+\infty\}$ is convex, proper, and lower semicontinuous
on $H$, then an operator $A$ on $H$ is called \emph{$\partial_{H} j$-monotone},
if for every $\lambda>0$, the resolvent $J_{\lambda}^{A}$ of $A$ satisfies
\begin{displaymath}
  j(J_{\lambda}^{A}u - J^{A}_{\lambda}\hat{u}) \le j(u-\hat{u})
\qquad\text{for every $u$, $\hat{u}\in \textrm{Rg}(I+\lambda A)$.}
\end{displaymath}
In addition, if the subdifferential operator $\partial_{H}j$ of $j$ is
a mapping $\partial_{H}j : D(\partial_{H}j)\to H$, then we call
$\partial_{H}j$ \emph{weakly continuous} if $\partial_{H}j$ maps
weakly convergent sequences to weakly convergent sequences.
%
%

The notion of \emph{completely accretive operators} was introduced
in~\cite{MR1164641} by Crandall and B\'enilan and further developed
in~\cite{CoulHau2017}. Following the same notation as in these two
references, $\mathcal{J}_{0}$ denotes the set of all convex, lower
semicontinuous functions $j : \R\to \overline{\R}_+$ satisfying
$j(0)=0$. Let $(\Sigma,\mu)$ be a $\sigma$-finite measure space and
$M(\Sigma,\mu)$ the space (of all classes) of measurable real-valued
functions on $\Sigma$.

\begin{definition}
  A mapping $S : D(S)\to M(\Sigma,\mu)$ with domain $D(S)\subseteq
  M(\Sigma,\mu)$ is called a \emph{complete contraction} if 
  \begin{displaymath}
    \int_{\Sigma}j(Su-S\hat{u})\,\textrm{d}\mu\le 
    \int_{\Sigma}j(u-\hat{u})\,\textrm{d}\mu
  \end{displaymath}
  for all $j\in \mathcal{J}_{0}$ and every $u$, $\hat{u}\in D(S)$.
\end{definition}

Choosing $j(\cdot)=\abs{[\cdot]^{+}}^{q}\in \mathcal{J}_{0}$ if
$1\le q<\infty$ and $j(\cdot)=[[\cdot]^{+}-k]^{+}\in \mathcal{J}_{0}$
for $k\ge 0$ large enough if $q=\infty$ shows that each complete
contraction $S$ is $T$-contractive in $L^{q}(\Sigma,\mu)$ for every
$1\le q\le \infty$. And by choosing
$j(\cdot)=\abs{[\cdot]^{+}}^{q}\in \mathcal{J}_{0}$ for any
$1\le q<\infty$, a complete contraction $S$ is
\emph{order preserving}, that is, for every $u$, $\hat{u}\in D(S)$
satisfying $u\le \hat{u}$ a.e. on $\Sigma$, one has that
$Su\le S\hat{u}$. In fact, the following characterization holds.

\begin{proposition}[{\cite{MR1164641}}]\label{prop:charact-cc}
  Suppose that the mapping $S : D(S)\to M(\Sigma,\mu)$ with domain $D(S)\subseteq
  M(\Sigma,\mu)$ satisfies the following: for every $u$, $\hat{u}\in
  D(S)$, $k\ge 0$, one has either
  \begin{displaymath}
    \min\{u,(\hat{u}+k)\}\in D(S)\text{ or }\max\{(u-k),\hat{u}\}\in D(S).
  \end{displaymath}
  Then, $S$ is a complete contraction if and only if $S$ is order
  preserving, and a $L^{1}$- and $L^{\infty}$-contraction.
\end{proposition}

\begin{definition}\label{def:completely-accretive}
  An operator $A$ on $M(\Sigma,\mu)$ is called \emph{completely
    accretive} if for every $\lambda>0$, the resolvent operator
  $J_{\lambda}$ of $A$ is a complete contraction. An operator $A$ on $M(\Sigma,\mu)$ is said
  to be \emph{$m$-completely accretive} if $A$ completely accretive
  and the range condition $\textrm{Rg}(I+A)=L^{2}(\Sigma,\mu)$ holds. A semigroup
  $\{T_{t}\}_{t\ge 0}$ on a subset a closed subset $C$ of
  $M(\Sigma,\mu)$ is called \emph{order preserving} if each map
  $T_{t}$ is order preserving.
\end{definition}

Next, we first  briefly recall the notion of \emph{Orlicz spaces}. 
Following \cite[Chapter 3]{MR1113700}, a continuous function $\psi :
[0,+\infty) \to [0,+\infty)$ is an {\em $N$-function} if it is convex, $\psi(s)=0$ if
and only if $s=0$, $\lim_{s\to 0+} \psi (s) / s = 0$, and
$\lim_{s\to\infty} \psi (s) / s = \infty$. 

\begin{definition}
  \label{def:orlicz}
  Given an $N$-function $\psi$, the {\em Orlicz
    space} 
  \begin{displaymath}
    L^\psi (\Sigma,\mu) := \Bigg\{ u : \Sigma \to\R\, \Bigg\vert\;
    \text{$u$ measurable }\&\,\int_\Sigma
    \psi (\frac{|u|}{\alpha} )\; \mathrm{d}\mu <\infty \text{ for some }
    \alpha >0\Bigg\}
  \end{displaymath}
  and equipped with the Orlicz-Minkowski norm
  \begin{displaymath}
    \| u\|_{L^\psi} := \inf \Bigg\{ \alpha >0 \;\Bigg\vert\; \int_\Sigma \psi
    (\frac{|u|}{\alpha} )\; \mathrm{d}\mu\leq 1 \Bigg\} .
  \end{displaymath}
\end{definition}

For $1\le q\le \infty$, we write $L^{q}_{loc}(H)$ and
$L^{q}(H)$ to denote the  vector-valued Lebesgue spaces $L^{q}_{loc}(\R_{+};H)$,
$L^{q}(\R_{+};H)$. The derivative $u'$ of a function $u\in L^{1}_{loc}(H)$
is usually understood in the \emph{distributional sense}. More
precisely, a function $w\in L^{1}_{loc}(H)$ is called the \emph{weak
  derivative} of $u\in L^{1}_{loc}(H)$ if
\begin{displaymath}
  \int_{0}^{+\infty}u(t)\,\xi'(t)\,\dt=-\int_{0}^{+\infty}w(t)\, \xi(t)\,\dt
\end{displaymath}
for all test functions $\xi\in C_{c}^{\infty}(\R_{+})$. A
function $w\in L^{1}_{loc}(H)$ satisfying the latter equation for all
$\xi\in C_{c}^{\infty}(\R_{+})$, is unique and so, one writes
$u'=w$. We denote by $W^{1,1}_{loc}(H)$ the space of all $u\in
L^{1}_{loc}(H)$ with a weak derivative $u'\in L^{1}_{loc}(H)$.

Next, let $0<s<1$. Then $L^{2}_{s}(H)$ denotes the space of all $u\in L^{1}_{loc}(H)$
satisfying $t^{s}u\in L^{2}_{\star}(H)$. We equip the \emph{first order weighted Sobolev space}
\begin{displaymath}
  W_{s}^{1,2}(H)=\Big\{u\in W^{1,1}_{loc}(H)\,\Big\vert\,
  u,\; u'\in L^{2}_{s}(H)\Big\}.
\end{displaymath}
with the inner product
\begin{displaymath}
  (u,\hat{u})_{W^{1,2}_{s}(H)}:=\int_{0}^{+\infty}
\big(u(t)\,\hat{u}(t) +u'(t)\,\hat{u}'(t)\big)\,  t^{2s}\,\frac{\dt}{t}.
\end{displaymath}
Then, $W_{s}^{1,2}(H)$ is a Hilbert space and we denote by
$\norm{\cdot}_{W^{1,2}_{s}(H)}$ the induced norm of $W_{s}^{1,2}(H)$.
Further, throughout this paper
\begin{displaymath}
  \underline{s}:=\frac{1-s}{2s}\qquad\text{and}\qquad
  \overline{s}:=\frac{3s-1}{2s}\qquad\text{for every $0<s<1$.}
\end{displaymath}
Then, $L^{2}_{\underline{s}}(H)$ denotes
the space of all $v\in
L^{1}_{loc}(H)$ satisfying $z^{\frac{1-s}{2s}}v\in
L^{2}_{\star}(H)$ equipped with the inner product
\begin{equation}
  \label{eq:39}
 (v,w)_{L^{2}_{\underline{s}}(H)}:=\int_{0}^{\infty}(z^{\frac{1-s}{2s}}
 	v(z),z^{\frac{1-s}{2s}}w(z)) \,\frac{\dz}{z}=\int_{0}^{\infty}(
 	v(z),w(z)) z^{\frac{1-2s}{s}}\,\dz
\end{equation}
for every $v$, $w\in L^{2}_{\underline{s}}(H)$. Similarly, we write
$L^{2}_{\overline{s}}(H)$ to denote the space of all
$v\in L^{1}_{loc}(0,+\infty;H)$ satisfying
$z^{\frac{3s-1}{2s}}v\in L^{2}_{\star}(H)$. 

The spaces
$W^{1,2}_{\frac{1-s}{2s},\frac{1}{2}}(H)$ and
$W^{1,2}_{\frac{1}{2},\frac{3s-1}{2s}}(H)$ are
\emph{first order Sobolev spaces with mixed weights} defined by
\begin{displaymath}
    W^{1,2}_{\frac{1-s}{2s},\frac{1}{2}}(H)=\Big\{v\in L^{1}_{loc}(H)\,\Big\vert\,
    z^{\frac{1-s}{2s}}v\in L^{2}_{\star}(H),\; z^{\frac{1}{2}}v'\in L^{2}_{\star}(H)\Big\}
\end{displaymath}
and 
\begin{displaymath}
    W^{1,2}_{\frac{1}{2},\frac{3s-1}{2s}}(H)=\Big\{v\in L^{1}_{loc}(H)\,\Big\vert\,
    z^{\frac{1}{2}}v\in L^{2}_{\star}(H),\; z^{\frac{3s-1}{2s}}v'\in L^{2}_{\star}(H)\Big\}.
\end{displaymath}
In addition, we will employ the \emph{second order Sobolev space with mixed weights}
\begin{displaymath}
    W^{2,2}_{\frac{1-s}{2s},\frac{1}{2},\frac{3s-1}{2s}}(H)
    :=\Big\{v\in W^{1,2}_{\frac{1-s}{2s},\frac{1}{2}}(H)\,\Big\vert\,
    z^{\frac{3s-1}{2s}}v''\in L^{2}_{\star}(H)\Big\}.
\end{displaymath}
Each of these spaces
equipped with its natural inner product and the induced norms
\begin{align*}
  \norm{v}_{W^{1,2}_{\frac{1-s}{2s},\frac{1}{2}}(H)}&:=\left(\norm{z^{\frac{1-s}{2s}}v}^{2}_{L^{2}_{\star}(H)}
      +\norm{z^{\frac{1}{2}}v'}^{2}_{L^{2}_{\star}(H)}
    \right)^{1/2}\\
\norm{v}_{W^{1,2}_{\frac{1}{2},\frac{3s-1}{2s}}(H)}&:=\left(\norm{z^{\frac{1}{2}}v}^{2}_{L^{2}_{\star}(H)}
      +\norm{z^{\frac{3s-1}{2s}}v'}^{2}_{L^{2}_{\star}(H)}
    \right)^{1/2}\\
    \norm{v}_{W^{2,2}_{\frac{1-s}{2s},\frac{1}{2},\frac{3s-1}{2s}}(H)}
   &:=\left(\norm{v}^{2}_{W^{1,2}_{\frac{1-s}{2s},\frac{1}{2}}(H)} 
     +\norm{z^{\frac{3s-1}{2s}}v''}^{2}_{L^{2}_{\star}(H)}\right)^{1/2}
\end{align*}   
 is a Hilbert space.\medskip

 The next proposition shows that for functions $u\in W^{1,2}_{s}(H)$,
 the initial value $u(0)\in H$. Note, this proposition is a special
 case of~\cite[Proposition~3.2.1]{MR3772192} since for $X=Y=H$, the
 interpolation space $[X,Y]_{\theta}=H$ for every $\theta\in (0,1)$.

\begin{proposition}[{\cite[Proposition~1.2.10]{MR1329547}}]\label{propo:1}
  Let $0<s<1$. Then for every $u\in W^{1,2}_{s}(H)$, the limit
  $u(0):=\lim_{t\to0+}u(t)$ exists in $H$. Moreover, the trace map
  \begin{displaymath}
    \textrm{Tr} : W^{1,2}_{s}(H)\to H,\; u\mapsto u(0)\text{ is
      continuous, surjective}
  \end{displaymath}
  and there are $c_{1}$, $c_{2}>0$ such that
  \begin{displaymath}
    c_{1}\,\norm{x}_{H}\le \inf_{u\in W^{1,2}_{s}(H) : u(0)=x}
    \norm{u}_{W^{1,2}_{s}(H)}\le c_{2}\norm{x}_{H}
  \end{displaymath}
   for every $x\in H$.
\end{proposition}

To conclude this preliminary section, we state the following
integration by parts rule from~\cite[Proposition~3.9]{MR3772192}.

\begin{proposition}
   Let $0<s<1$. For $u\in W^{1,2}_{s}(H)$ and $\xi \in
   W^{1,2}_{1-s}(H)$, the functions $t\mapsto (u'(t),\xi(t))_{H}$ and
   $t\mapsto (u(t),\xi'(t))_{H}$ belong to
   $L^{1}(0,+\infty)$. Moreover, the following integration by parts
   rule holds:
   \begin{displaymath}
     -\int_{0}^{+\infty}(u'(t),\xi(t))_{H}\,\dt=\int_{0}^{+\infty}(u(t),\xi'(t))_{H}\,\dt+(u(0),\xi(0))_{H}.
   \end{displaymath}
\end{proposition}

%
%
%
%

\section{Well-posedness of second order boundary value problems}
\label{sec:exuniqueness}

In this section, we are concerned with establishing the well-posedness
of the following more general abstract boundary-value problem
 \begin{equation}
    \label{eq:1}
    \begin{cases}
      u''(t) + \frac{1-2s}{t}u'(t) \in Au(t) & \text{for almost
        every }t> 0,\\[7pt]
      \mbox{}\hspace{0,32cm}\displaystyle\lim_{t \to 0_+} 
      t^{1-2s}u'(t) \in \partial j(u(0) - \varphi), &
    \end{cases}
  \end{equation}
  where $j : H \to \overline{\R}_+$ be a convex, strongly coercive,
  lower semicontinuous functional satisfying $j(0) = 0$,
  $\varphi\in \overline{D(A)}^{\mbox{}_{H}}$, and $1<s<1$.

\begin{definition}\label{def:thm1bis}
  For $\varphi\in H$, a function $u : [0,+\infty)\to H$ is called a
  \emph{solutions} of problem~\eqref{eq:1} if $u\in C([0,+\infty);H)$,
  $u$ is a strong solution of~\eqref{eq:79},
  $\lim_{t\to 0+}t^{1-2s}u'(t)$ exists in $H$ and
    \begin{displaymath}
      \displaystyle\lim_{t \to 0_+} t^{1-2s}u'(t) \in \partial j(u(0) - \varphi).
    \end{displaymath}
 \end{definition}

 For our next theorem, we recall (cf~\cite{MR0348562} or
 \cite{MR2582280}) that the \emph{indicator function} $j : H \to
 \overline{\R}_+$ is defined by $j(v):=0$ if $v=0$ and
 $j(v):=+\infty$ if otherwise. The following theorem is the first main
 result of this section.

\begin{theorem}
  \label{thm:1bis}
  Let $A$ be a maximal monotone operator on $H$ with
  $0\in \textrm{Rg}(A)$, and $j : H \to \overline{\R}_+$ be a convex,
  strongly coercive, lower semicontinuous functional satisfying
  $j(0) = 0$ and either $\partial_{H}j : D(\partial_{H}j) \to H$ is a
  weakly continuous mapping or $j$ is the ``indicator
  function''. Assume further that $A$ is $\partial j$-monotone. Then,
  for every $0<s<1$, $\varphi\in D(A)$ and $y\in A^{-1}(\{0\})$, there
  is a unique solution $u\in L^{\infty}(H)$ of problem~\eqref{eq:1}
  satisfying~\eqref{eq:87A}, \eqref{eq:30bis2A}-\eqref{eq:31bis2A} and
  \eqref{eq:81}-\eqref{eq:82}. In particular,
  $t^{1-2s}u'\in W^{1,2}_{s}(H)$, the function
  $t\mapsto \norm{u(t)}_{H}^{2}$ is convex, bounded and decreasing on
  $[0,+\infty)$, and for every two strong solutions $u$ and $\hat{u}\in L^{\infty}(H)$
  of equation~\eqref{eq:79}, one has that~\eqref{eq:83} holds.
  %
  %
  %
  %
\end{theorem}

\begin{remark}[{\em The case $A=\partial_{H}\phi$}]\label{rem:1}
  Suppose $A=\partial_{H}\phi$ the subdifferential operator of a
  proper, convex, lower semicontinuous functional
  $\phi : H\to \R\cup\{+\infty\}$ attaining a global minimum
  $\min_{v\in H} \phi(v)=\phi(v_{0})$ at some $v_{0}\in H$. After possibly
  replacing $\phi$ by $\tilde{\phi}(v-v_{0})-\phi(v_{0})$, ($v\in H$),
  we may assume without loss of generality that $\phi$ attains it
  minimum at $0\in H$ and $\phi : H\to \overline{\R}_{+}$. Hence, for
  $\varphi\in D(A)$, the
  second-order boundary problem~\eqref{eq:1} can be rewritten as
  \begin{equation}
    \label{eq:84}
    0\in \partial_{L^{2}_{1-s}(H)}\E(u),
  \end{equation}
  or equivalently, as the minimization problem
  \begin{displaymath}
    \min_{u\in L^{2}_{1-s}(H)} \E(u)
  \end{displaymath}
  for the functional $\E :  L^{2}_{1-s}(H)\to \R\cup\{+\infty\}$ defined by
  \begin{displaymath}
    \E(u):=
    \begin{cases}
      \displaystyle\int_{0}^{+\infty}t^{2(1-s)}
      \Big\{\tfrac{1}{2}\norm{u'(t)}_{H}^{2}+\phi(u(t))\Big\}\,\frac{\dt}{t}
      + j(u(0)-\varphi)
      & \text{if $u\in D(\E)$,}\\[7pt]
      +\infty & \text{if otherwise,}
    \end{cases}
  \end{displaymath}
  where
  $D(\E):=\{u\in W^{1,2}_{1-s}(H)\,\vert\; \phi(u)\in
  L^{1}_{1-s}(H),\;u(0)-\varphi\in D(j)\}$. It is not difficult to
  see that $\E$ is convex and by Proposition~\ref{propo:1}, $\E$ is
  proper and lower semicontinuous on $L^{2}_{1-s}(H)$
  (cf~\cite[Exemple~2.8.3]{MR0348562}). But to see that $\E$ is
  coercive, stronger assumptions on $\phi$ are needed; for example,
  suppose there is an $\eta>0$ such that
  \begin{displaymath}
    \phi(u)\ge \eta\, \norm{u}_{H}^{2},\qquad\text{for all $u\in D(\phi)$.}
  \end{displaymath}

  This is consistent with the linear
  theory of sectorial operators (cf~\cite[p 10
  formula~(4.4)]{MR3772192}) and the case $s=1/2$ in the
  nonlinear theory (cf~\cite[Chapter V.2]{MR0390843}). 
\end{remark}


Under the assumption that the statement of Theorem~\ref{thm:1bis}
holds, we outline how to deduce the statements of Theorem~\ref{thm:1}
and Theorem~\ref{thm:2Robin}.

\begin{proof}[Proof of Theorem~\ref{thm:1} and
 Theorem~\ref{thm:2Robin}]
 By choosing $j$ to be the indicator function, one sees that
 problem~\eqref{eq:1} reduces to Dirichlet problem~\eqref{eq:1bis}.
 Therefore for $\varphi\in D(A)$, the statements of
 Theorem~\ref{thm:1} follows from Theorem~\ref{thm:1bis}. Now, let
 $\varphi\in \overline{D(A)}^{\mbox{}_{H}}$. Then there are sequences
 $(\varphi_{n})_{n\ge 1}$ of $\varphi_{n}\in D(A)$ and
 $(v_{n})_{n\ge 1}$ of corresponding solutions $u_{n}$
 of~\eqref{eq:1bis}. By~\eqref{eq:83}, 
\begin{displaymath}
  \sup_{t\ge 0}\norm{u_{n}(t)-u_{m}(t)}_{H}\le \norm{\varphi_{n}-\varphi_{m}}_{H}
\end{displaymath}
for every $n$, $m\ge 0$. Moreover, by~\eqref{eq:87A}, every $u_{n}\in
L^{\infty}(H)$. Therefore, $(u_{n})_{n\ge 1}$ is a Cauchy sequence in
$C^{b}([0,+\infty);H)$ and so, there is a function $u\in
C^{b}([0,+\infty);H)$ such that
\begin{displaymath}
  \lim_{n\to +\infty}u_{n}=u\qquad\text{in $C^{b}([0,+\infty);H)$.}
\end{displaymath}
From this limit, we can conclude that $u$ satisfies~\eqref{eq:87A}
and~\eqref{eq:83}. In particular, since $u_{n}(0)=\varphi_{n}\to
\varphi$ in $H$, one has that $u(0)=\varphi$ in $H$. Further,
by~\eqref{eq:81} 
and~\eqref{eq:82} 
applied to $u_{n}$, we can conclude that
$u\in W^{2,2}_{loc}((0,+\infty);H)$ and after possibly passing to a
subsequence of $(u_{n})_{n\ge 1}$ and taking
$\liminf_{n\to +\infty}$ in those inequalities, we obtain that $u$ satisfies~\eqref{eq:81} 
and~\eqref{eq:82}. 
To see that $u$ is a strong solution of~\eqref{eq:79}, 
one proceeds similarly to \emph{step~5} in the proof of
Theorem~\ref{thm:2} (below). This proves the statement of
Theorem~\ref{thm:1}.


 Next, to see that also Theorem~\ref{thm:2Robin} holds, for given
 $\lambda>0$ and $\varphi\in D(A)$, one chooses
 \begin{displaymath}
   j(\cdot)=\frac{\lambda}{2}\lnorm{\cdot-\left(\frac{1}{\lambda}-1\right)\varphi}_{H}^{2}
 \end{displaymath}
 and applies Theorem~\ref{thm:1bis}. 
\end{proof}

The key to prove Theorem~\ref{thm:1bis} (respectively,
Theorem~\ref{thm:1}) is via the change of variable~\eqref{eq:15},
which transforms the abstract boundary-value problem~\eqref{eq:1}
associated with the Bessel-type operator $A_{1-2s}$ to the following
abstract boun\-dary-value problem
  \begin{equation}
    \label{eq:2}
    \begin{cases}
      z^{-\frac{1-2s}{s}} v''(z) \in Av(z) & \text{for
        almost every  $z>0$,}\\
      \mbox{}\hspace{1,2cm} v'(0) \in \partial \tilde{j}(v(0) -
      \varphi), &
    \end{cases}
  \end{equation}
  for given $\varphi\in \overline{D(A)}^{\mbox{}_{H}}$ and where
  $\tilde{j}=(2s)^{-(1-2s)} j$.

 \begin{definition}\label{def:Avz}
   We call a function $v : [0,+\infty)\to H$ a \emph{strong
      solution} of equation~\eqref{eq:80} if
    $v\in W^{2,2}_{loc}((0,+\infty);H)$ and for almost every $z>0$,
    $v(z)\in D(A)$ and
    $z^{-\frac{1-2s}{s}}v''(z)\in A(v(z))$. For
    given $\varphi\in H$ and $\tilde{j}= (2s)^{-(1-2s)} j$, a function
    $v$ is called a \emph{solution} of
    problem~\eqref{eq:2} if $v\in C^{1}([0,+\infty);H)$ is a strong
    solution of~\eqref{eq:80} satisfying
    \begin{math}
     v'(0) \in \partial \tilde{j}(v(0)-\varphi).
    \end{math}
    
    Further, for given $\varphi\in H$,  a function
    $v\in C([0,+\infty);H)$ is called a \emph{solution of Dirichlet 
    problem}
    \begin{equation}
      \label{eq:2bis}
       \begin{cases}
      z^{-\frac{1-2s}{s}} v''(z) \in Av(z) & \text{for
        almost every  $z>0$,}\\
      \mbox{}\hspace{1,3cm} v(0) =\varphi, &
    \end{cases}
    \end{equation}
    if $v\in C([0,+\infty);H)$, $v$ is a strong solution
    of~\eqref{eq:80}, and $v(0)=\varphi$.
 \end{definition}

Our first lemma outlines the equivalence between
Definition~\eqref{def:thm1bis} and Definition~\eqref{def:Avz}.

 \begin{lemma}\label{lem:2}
   Let $A$ be an operator on $H$, $j : H \mapsto \R\cup\{+\infty\}$ a
   proper functional, $\varphi\in H$, and $0<s<1$. For
   $u\in W^{2,2}_{loc}((0,+\infty);H)$, let $u(t)=v(z)$ for $z$ given
   by the change of variable~\eqref{eq:15}. Then, the following
   statements hold.
   \begin{enumerate}
    \item \label{lem:2claim1}  $u$ is a solution of \eqref{eq:1}
     if and only if $v$ is a solution of~\eqref{eq:2}.
 
   \item \label{lem:2claim2} A function $u\in W^{1,2}_{s}(H)$ if and only if $v\in
     W^{1,2}_{\frac{1}{2},\frac{3s-1}{2s}}(H)$,
     and a function $u\in
     W^{1,2}_{1-s}(H)$ if and only if $v\in
     W^{1,2}_{\frac{1-s}{2s},\frac{1}{2}}(H)$ with
     \begin{equation}
       \label{eq:36}
       u'(t)=v'(z) z^{-\frac{1-2s}{2s}}.
     \end{equation}
  \end{enumerate}
\end{lemma}

\begin{proof}
  If $z =z(t)= (\frac{t}{2s})^{2s}$, then $t=t(z)= 2s z^{\frac{1}{2s}}$ and so, $v(z)
  = u(t) = u(2s z^{\frac{1}{2s}})$. Then, for $z$ and $t>0$,
  \begin{displaymath}
    v'(z) = \frac{\td}{\dz} v(z) = \frac{\td u}{\dt}
                                               \cdot \frac{\dt}{\dz} 
                                             = u'(t) z^{\frac{1-2s}{2s}},
  \end{displaymath}
  proving~\eqref{eq:36}. By using~\eqref{eq:36}, one
  easily verifies that $u\in W^{1,2}_{1-s}(H)$ if and only if $v\in
  W^{1,2}_{\frac{1-s}{2s},\frac{1}{2}}(H)$, and $u\in W^{1,2}_{s}(H)$ if and only if $v\in
     W^{1,2}_{\frac{1}{2},\frac{3s-1}{2s}}(H)$. Thus
     claim~\eqref{lem:2claim2} holds. Further,
  \begin{displaymath}
    v''(z) = \frac{\td}{\dz} v'(z) 
           = u''(t)  z^{\frac{1-2s}{s}} + u'(t) \,\frac{1-2s}{2s} z^{\frac{1-2s}{2s}-1}
  \end{displaymath}
  for $z$, $t>0$. Multiplying this equation by
  $z^{-\frac{1-2s}{s}}$, 
  one sees that
  \begin{displaymath}
    z^{-\frac{1-2s}{s}} v''(z) = u''(t) + \frac{1-2s}{2s} z^{-\frac{1}{2s}} u'(t) 
                               = u''(t) + \frac{1-2 s}{t}u'(t) 
  \end{displaymath}
  for almost every $z$, $t>0$. Therefore, $ z^{-\frac{1-2s}{s}} v''(z) \in
  Av(z)$ for a.e. $z>0$ if and only if $u''(t) + \frac{1-2
    s}{t}u'(t)\in Au(t)$ for a.e. $t>0$, showing that $u$ is a
  strong solution of~\eqref{eq:79} if and only if $v$ is a strong
  solution of~\eqref{eq:80}. Moreover,
  $t^{1-2s}u'\in W^{1,2}_{s}(H)$ if and only if
  $z^{\tfrac{1}{2}}v'$ and $z^{-\frac{1-3s}{2s}}v''\in
  L^{2}_{\star}(H)$. 
  Multiplying $v'(z)$ by $(2s)^{1-2s}$ shows that
  \begin{displaymath}
    (2s)^{1-2s} v'(z) = (2s)^{1-2s} u'(t) z^{\frac{1-2s}{2s}} 
                  = t^{1-2s} u'(t).
  \end{displaymath}
  Note that $\lim_{z\to 0+}t(z)=0$ and $\lim_{t\to 0+}z(t)=0$. Thus
  \begin{displaymath}
    \lim_{t \to 0+}t^{1-2s}u'(t)\text{ exists in $H$ if and only if }\lim_{z\to 0+}
  v'(z)\text{ exists in $H$.}
  \end{displaymath}
 By Proposition~\ref{propo:1}, $t^{1-2s}u'\in
 W^{1,2}_{s}(H)$ implies that $\lim_{t \to
    0+}t^{1-2s}u'(t)$ exists in $H$. 
Similarly, one has that $u(0):=\lim_{t\to 0+}u(t)$ exists in $H$ if and
    only if $v(0)=\lim_{z\to 0+}v(z)$ exists in $H$,
and by Proposition~\ref{propo:1}, $u(0):=\lim_{t\to 0+}u(t)=\varphi$
exists in $H$. Therefore and by the definition of the subdifferential
$\partial_{H} j$, one has that
$\lim_{t \to 0_+} t^{1-2s}u'(t) \in \partial_{H} j(u(0) -
\varphi)$ if and only if
$(1-\alpha)^{\alpha} v'(0) \in \partial_{H} j(v(0) - \varphi)$, which
completes the proof of showing that $u$ is a solution of~\eqref{eq:1}
if and only if $v$ is a solution of~\eqref{eq:2}.
\end{proof}

Our next theorem provides the existence and uniqueness of
problem~\eqref{eq:2}.


 \begin{theorem}
   \label{thm:2}
   Under the hypothesis of Theorem~\ref{thm:1bis}, let
   $\tilde{j}:= (2s)^{-(1-2s)} j$. Then, for every
   $\varphi\in D(A)$ and $y\in A^{-1}\{0\}$, there is a unique
   solution
   \begin{displaymath}
     v\in L^{\infty}(H)\cap C^{1}([0,+\infty);H)\quad\text{ with
     }\quad v'\in W^{1,2}_{\frac{1}{2},\frac{3s-1}{2s}}(H)
   \end{displaymath}
   of problem~\eqref{eq:2} satisfying
   \allowdisplaybreaks
   \begin{align}
     \label{eq:87}
     \norm{v(z)-y}_{H} & \le \norm{v(\hat{z})-y}_{H}\qquad\text{for
       all $z\ge \hat{z}\ge 0$,}\\
     \label{eq:66}
     \norm{v'(z)}_{H}&\le \norm{v'(\hat{z})}_{H}\qquad\text{for
       all $z\ge \hat{z}\ge 0$,}\\ \label{eq:65}
     \norm{z\,v'(z)}_{L^{2}_{\star}(H)} &\le
       \frac{\norm{\varphi-y}_{H}}{\sqrt{2}}\\ \label{eq:89}
     \norm{v'(z)}_{H}&\le \frac{\norm{\varphi-y}_{H}}{z}\qquad\text{for
       every $z>0$,}\\
\label{eq:77}
     \norm{z^{2}v''}_{L^{2}_{\star}(H)} &\le
     \begin{cases}
     \displaystyle\frac{\norm{\varphi-y}_{H}}{\sqrt{2}} & \text{if $s\ge \frac{1}{2}$,}\\[7pt]
     \displaystyle\frac{1}{2}\left(\frac{s}{1-2s}\frac{1}{2}+3\right)^{\frac{1}{2}}\,
     \norm{\varphi-y}_{H},
     & \text{if $0<s <\frac{1}{2}$.}
     \end{cases}\\
      \label{eq:30bis2}
     \norm{v'(0)}_{H}^{\frac{1}{2}}&\le \norm{A^{0}\varphi}_{H}^{\frac{1}{2}}+\norm{\varphi-y}_H^{\frac{1}{2}},\\ 
     \label{eq:14bis2}
     \norm{z^{\frac{1}{2}}v'}_{L^{2}_{\star}(H)} &\le \norm{A^{0}\varphi}_{H}^{\frac{1}{2}} +\norm{\varphi-y}_{H},\\
     \label{eq:31bis2}
     \norm{v''}_{L^{2}_{\overline{s}}}
   &\le \norm{A^{0}\varphi}_{H} +\norm{\varphi-y}_H^{\frac{1}{2}}\,\norm{A^{0}\varphi}_{H}^{\frac{1}{2}}.
   \end{align}
   For every two strong solutions $v$ and $\hat{v}\in L^{\infty}(H)$
   of~\eqref{eq:80}, one has that
   \begin{equation}
     \label{eq:88}
     \norm{v(z)-\hat{v}(z)}_{H}\le
     \norm{v(\hat{z})-\hat{v}(\hat{z})}_{H}\qquad
     \text{for every $z\ge \hat{z}\ge 0$.}
   \end{equation}

   Further, for every $\varphi\in \overline{D(A)}^{\mbox{}_{H}}$ and
   $y\in A^{-1}\{0\}$, there is a unique solution $v \in
   L^{\infty}(H)$ of Dirichlet problem~\eqref{eq:2bis}
   satisfying~\eqref{eq:87}-\eqref{eq:77}, and~\eqref{eq:88}.
\end{theorem}

Thanks to Lemma~\ref{lem:2}, Theorem~\ref{thm:2} implies that the
statement of Theorem~\ref{thm:1bis} holds. In particular, by
Theorem~\ref{thm:2}, if the boundary value $\varphi\in D(A)$, then the
unique solution $v$ of Dirichlet problem~\eqref{eq:2bis}
satisfies $v\in W^{1,2}_{\frac{1}{2},\frac{3s-1}{2s}}(H)$,
implying that the Neumann derivative
\begin{displaymath}
  -(2s)^{1-2s}
  v'(0)=:\Theta_{s}\varphi\qquad\text{exists in $H$.}
\end{displaymath}
This allows us to define the DtN operator $\Theta_{s}$
associated with $\tilde{A}_{1-2s}$ as given in the following Corollary.

\begin{corollary}
  \label{cor:1}
  Let $A$ be a maximal monotone operator on $H$ with
  $0\in \textrm{Rg}(A)$. 
  Then, for every $0<s<1$, the Dirichlet-to-Neumann operator
  $\Theta_{s}$ associated with $\tilde{A}_{1-2s}$ defined by
  \begin{displaymath}
    \Theta_{s}=\Bigg\{(\varphi,w)\in H\times H\,\Bigg\vert
    \begin{array}[c]{l}
      \exists\text{ a solution $v$ of Dirichlet problem~\eqref{eq:80}},\\
      \text{with $v(0)=\varphi$ and $w=-(2s)^{1-2s} v'(0)$ in $H$.}
    \end{array}
    \Bigg\}
  \end{displaymath}
  is a monotone, well-defined mapping $\Theta_{s} :
  D(\Theta_{s})\to H$ satisfying
  \begin{displaymath}
    D(A)\subseteq D(\Theta_{s})\subseteq \overline{D(A)}^{\mbox{}_{H}},\qquad\text{ and }\qquad
     D(A)\subseteq \textrm{Rg}(I_{H}+\lambda\,
      \Theta_{s})\qquad\text{for all $\lambda>0$.}
  \end{displaymath} 
  The closure $\overline{\Theta}_{s}$ of $\Theta_s$ in $H\times H$
  is characterized by
  \begin{displaymath}
    \overline{\Theta}_{s}:=\Bigg\{(\varphi,w)\in H\times H\,\Bigg\vert
    \begin{array}[c]{l}
      \exists \text{ $(\varphi_{n},w_{n})\in \Theta_{s}$ s.t. }
      \displaystyle\lim_{n\to +\infty}(\varphi_{n},w_{n})=(\varphi,w)\\
      \text{ in $H\times H$ \&} \text{ a strong solution $v$ of~\eqref{eq:80}}\\
      \text{satisfying } v(0)=\varphi\text{ in $H$.}
    \end{array}
    \Bigg\}
  \end{displaymath}
  with domain $D(\overline{\Theta}_{s})=\overline{D(A)}^{\mbox{}_{H}}$ and 
  \begin{displaymath}
    \overline{D(A)}^{\mbox{}_{H}} \subseteq \textrm{Rg}(I_{H}+\lambda\,
    \overline{\Theta_{s}})
    \qquad\text{for all $\lambda>0$.}
  \end{displaymath}
  If $D(A)$ is dense in $H$, then $\overline{\Theta}_{s}$ is
  maximal monotone.
\end{corollary}

By Lemma~\ref{lem:2}, Corollary~\ref{cor:1} implies that
Theorem~\ref{thm:DtN} holds. The rest of this section is dedicated to
the proof of Theorem~\ref{thm:2}\medskip

By using the change of variable~\eqref{eq:15},
Proposition~\ref{propo:1} can be rewritten for functions $v\in
W^{1,2}_{\frac{1-s}{2s},\frac{1}{2}}(H)$. Moreover, by~\cite{MR3772192}, we have
that the following \emph{integration by parts} holds.

\begin{lemma}\label{lem:ibp}
  Let $0<s<1$. Then, the following statements hold.
  \begin{enumerate}
  \item \label{lem:ibp1} ({\bfseries Trace theorem on
      $W^{1,2}_{\frac{1-s}{2s},\frac{1}{2}}(H)$}) For
    $v\in W^{1,2}_{\frac{1-s}{2s},\frac{1}{2}}(H)$, the
    limit $v(0):=\lim_{z\mapsto 0+}v(z)$ exists in $H$. In particular,
    the trace operator
    \begin{displaymath}
      \textrm{Tr} : W^{1,2}_{\frac{1-s}{2s},\frac{1}{2}}(H)\to
      H,\quad v\mapsto v(0)
    \end{displaymath}
    is continuous and surjective.

\item \label{lem:ibp1bis} ({\bfseries Trace theorem on
      $W^{1,2}_{\frac{1}{2},\frac{3s-1}{2s}}(H)$}) For
  $v\in W^{1,2}_{\frac{1}{2},\frac{3s-1}{2s}}(H)$, the limit
  $v(0):=\lim_{z\mapsto 0+}v(z)$ exists in $H$. In particular, the
  trace operator
  \begin{displaymath}
    \textrm{Tr} : W^{1,2}_{\frac{1}{2},\frac{3s-1}{2s}}(H)\to
  H,\quad v\mapsto v(0)
 \end{displaymath}
 is continuous and surjective, and there are $c_{1}$, $c_{2}>0$ such that
  \begin{equation}
    \label{eq:42}
    c_{1}\,\norm{x}_{H}\le \inf_{\xi\in W^{1,2}_{\frac{1}{2},\frac{3s-1}{2s}}(H) : \xi(0)=x}
    \norm{\xi}_{W^{1,2}_{\frac{1}{2},\frac{3s-1}{2s}}(H)}\le c_{2}\norm{x}_{H}
  \end{equation}
   for every $x\in H$.

   \item \label{lem:ibp2} For $w\in W^{1,2}_{\frac{1}{2},\frac{3s-1}{2s}}(H)$ and $\xi\in
     W^{1,2}_{\frac{1-s}{2s},\frac{1}{2}}(H)$, the functions
     \begin{displaymath}
       z\mapsto (w'(z),\xi(z))_{H}\text{ and }
       z\mapsto (w(z),\xi'(z))_{H}
     \end{displaymath}
     belong to $L^{1}(0,+\infty)$ and the following integration by parts
   rule holds:
   \begin{equation}
     \label{eq:32}
     \int_{0}^{+\infty}(w'(z),\xi(z))_{H}\,\dz=-(w(0),\xi(0))_{H}-\int_{0}^{+\infty}(w(z),\xi'(z))_{H}\,\dz.
   \end{equation}
  \end{enumerate}
\end{lemma}

\begin{proof}
  As mentioned in the proof
  of Lemma~\ref{lem:2}, $u\in W^{1,2}_{1-s}(H)$ if and only if $v\in
  W^{1,2}_{\frac{1-s}{2s},\frac{1}{2}}(H)$, and since
  $\lim_{z\to 0+}t(z)=0$ and $\lim_{t\to 0+}z(t)=0$, one has that the
  fact that $u(0):=\lim_{t\to
    0+}u(t)$ exists in $H$ is equivalent to $v(0):=\lim_{z\to 0+}v(z)$ exists in $H$.
 Therefore by Proposition~\ref{propo:1}, for every $v\in W^{1,2}_{\frac{1-s}{2s},\frac{1}{2}}(H)$,
    the limit $v(0):=\lim_{z\mapsto 0+}v(z)$ exists in $H$ and the trace
    operator $\textrm{Tr} :
    W^{1,2}_{\frac{1-s}{2s},\frac{1}{2}}(H)\to H$ is surjective. To see that the operator
    $\textrm{Tr}$ is continuous, it suffices to note that
  \begin{displaymath}
    \norm{u}_{W^{1,2}_{1-s}(H)}
    =(2s)^{\frac{1-2s}{2}}\norm{v}_{W^{1,2}_{\frac{1-s}{2s},\frac{1}{2}}(H)}.
  \end{displaymath}
  Thus, claim~\eqref{lem:ibp1} of this proposition holds and since by
  Lemma~\ref{lem:2}, $u\in W^{1,2}_{s}(H)$ if and only if
  $v\in W^{1,2}_{\frac{1}{2}, \frac{3s-1}{2s}}(H)$, one shows similarly that
  claim~\eqref{lem:ibp1bis} as well holds true. Next, let
  $w\in W^{1,2}_{\frac{1}{2}, \frac{3s-1}{2s}}(H)$ and
  $\xi\in W^{1,2}_{\frac{1-s}{2s},\frac{1}{2}}(H)$. Since
  $z=z^{\frac{3s-1}{2s}}\, z^{\frac{1-s}{2s}}$,
  H\"older's inequality yields that
  \begin{align*}
    \int_{0}^{\infty}\abs{(w'(z),\xi(z))_{H}}\,\dz 
    & =\int_{0}^{\infty}\abs{(z^{\frac{3s-1}{2s}}w'(z),
      z^{\frac{1-s}{2s}}\xi(z))_{H}}\,\frac{\dz}{z}\\
    & \le \norm{z^{\frac{3s-1}{2s}}w'}_{L^{2}_{\ast}(H)}\,\norm{z^{\frac{1-s}{2s}}\xi}_{L^{2}_{\ast}(H)}
  \end{align*}
  and
  \begin{align*}
    \int_{0}^{\infty}\abs{(w(z),\xi'(z))_{H}}\,\dz 
    & =\int_{0}^{\infty}\abs{(z^{\frac{1}{2}}w(z),
      z^{\frac{1}{2}}\xi'(z))_{H}}\,\frac{\dz}{z}\\
    & \le \norm{z^{\frac{1}{2}}w}_{L^{2}_{\ast}(H)}\,\norm{z^{\frac{1}{2}}\xi'}_{L^{2}_{\ast}(H)},
  \end{align*}
  proving that $(w',\xi)_{H}$ and $(w,\xi')_{H}\in L^{1}(0,+\infty)$. Finally, to see that integration by
  parts~\eqref{eq:32} holds, one applies
  \cite[Proposition~3.9]{MR3772192} to $w(z)=u(t)$ and
  $\xi(z)=v(t)$ with the change of variable~\eqref{eq:15}.
\end{proof}

With these preliminaries, we can now prove the uniqueness of solutions
to problem~\eqref{eq:2}. Here,
our proof adapts an idea by Brezis~\cite{MR0317123} to the more general case $0<s<1$.

\begin{proof}[Proof of Theorem~\ref{thm:2} (uniqueness and proof of inquality~\eqref{eq:88})]
  Suppose $v_1$ and $v_2\in L^{\infty}(H)$ are two strong solutions
  of~\eqref{eq:80} and set $w= v_1-v_2$. Then, by the monotonicity of
  $A$ and by~\eqref{eq:80},
  \begin{displaymath}
    z^{-\frac{1-2s}{s}}(w''(z), w(z))_{H} 
    = (z^{-\frac{1-2s}{s}}v_1''(z)- z^{-\frac{1-2s}{s}}v_2'(z),v_1(z) -v_2(z))_{H}\ge 0
  \end{displaymath}
  for almost every $z>0$.Thus, $(w''(z), w(z))_{H} \ge 0$ for almost
  every $z > 0$ and so,
  \begin{equation}
   \label{eq:52}  
   \begin{split}
     \frac{1}{2} \frac{\td^2}{\td z^2}\norm{w(z)}^2_{H}&=\frac{\td}{\td
       z}(w'(z),w(z))_{H}\\
     &= (w''(z),w(z))_{H}+\norm{w'(z)}_{H}^{2}
     \ge\norm{w'(z)}_{H}^{2} \ge 0
   \end{split}
  \end{equation}
  for a.e. $z>0$. Therefore, the function
  $z\rightarrow \norm{w(z)}^2_{H}$ is convex and since by hypothesis,
  $w$ is bounded on $\R_{+}$ with values in $H$, the function
  $z\rightarrow \norm{w(z)}^2_{H}$ is necessarily monotonically
  decreasing on $\mathbb{R}_+$. In particular, this argument shows
  that for every two strong solutions $v_{1}$,
  $v_{2}\in L^{\infty}(H)$ of~\eqref{eq:80}, one as that~\eqref{eq:88}
  holds.  Further, from this, we can deduce that
  \begin{displaymath}
    (w'(z),w(z))_{H}=\frac{\td}{\dz} \frac{1}{2}\norm{w(z)}_{H}^{2}\le 0
  \end{displaymath}
  for every $z\in \R_{+}$. 
  Now, note that for $0<s<1$, $\alpha=1-2s>0$. Thus, for
  $\tilde{j} := (2s)^{-(1-2s)}\, j$, $\partial_{H}\tilde{j}$ is
  monotone. Hence, if $v_1$ and $v_2$ are solutions of
  problem~\eqref{eq:2} for the same $\varphi\in H$, then by the
  condition $v_i'(0)\in \partial_{H} \tilde{j}(v(0)-\varphi)$ for
  $i=1,2$, one has that
\begin{displaymath}
  0\ge (w'(0),w(0))_{H}=(v_1'(0)-v_2'(0),(v_1(0)-\varphi)-(v_2(0)-\varphi))_{H}\ge 0.
\end{displaymath}
Combining this with~\eqref{eq:52}, one finds
\begin{displaymath}
  0\ge (w'(z), w(z))_{H} 
  = \int_{0}^{z} \frac{\td}{\td r}(w'(r),w(r))_{H}\ds \ge \int_{0}^{z} \norm{w'(r)}^2_{H}\dr,
\end{displaymath}
implying that $w'(z)=0$ in $H$ for all $z\ge 0$. Thus,
$v_1'(0)=v_2'(0)$ and since 
\begin{displaymath}
\tilde{j}(v_1(0) - \varphi) - \tilde{j}(v_2(0)- \varphi) \ge
(v_2'(0),v_1(0)-v_2(0))
\end{displaymath}
and
\begin{displaymath} 
\tilde{j}(v_2(0) - \varphi) - \tilde{j}(v_1(0) - \varphi) \ge (v_1'(0),v_2(0)-v_1(0)),
\end{displaymath}
it follows that
\begin{equation}
\label{eq:7}
\tilde{j}(v_2(0) - \varphi) - \tilde{j}(v_1(0)- \varphi) = (v_1'(0) , v_2(0) - v_1(0)).
\end{equation}
Now, if $v_1(0) \neq v_2(0)$, then by the strict convexity
of $\tilde{j}$ and~\eqref{eq:7},
\begin{align*}
  \frac{1}{2} \tilde{j}(v_1(0) - \varphi) + \frac{1}{2} \tilde{j}(v_2(0) -\varphi) 
  &>\tilde{j}\Big( \frac{v_1(0) + v_2(0)}{2} - \varphi \Big) \\ 
  &\ge \tilde{j}(v_1(0) - \varphi) + \Big(v_1'(0),\frac{v_2(0) -
    v_1(0)}{2} \Big)\\ 
  &= \frac{1}{2}\tilde{j}(v_1(0) - \varphi) + \frac{1}{2}\tilde{j}(v_2(0) - \varphi),
\end{align*}
    which is a contradiction. Therefore, $v_1(0)=v_2(0)$, implying that
    $v_{1}=v_{2}$. This completes the proof of uniqueness.
\end{proof}


For proving existence of strong solutions of problem~\eqref{eq:2}, we
need the following proposition.

\begin{proposition}\label{prop:1}
 Suppose $j : H \to \overline{\R}_+$ is a convex, strongly coercive, lower
  semicontinuous functional satisfying $j(0) = 0$, let $\varphi\in H$
  and for $0<s<1$, $\tilde{j} := (2s)^{-(1-2s)}\, j$. Further,
  let $\E_{1}$ and $\E_{2}  : L^{2}_{\underline{s}}(H)\to
\R\cup\{+\infty\}$ be given by
\begin{displaymath}
\E_{1}(v):=
\begin{cases}
        \frac{1}{2} \displaystyle\int_{0}^{+\infty} \norm{v'(z)}^2_H\,
        \dz &\text{if  $v \in \WE_{\frac{1-s}{2s},\frac{1}{2}}^{1,2}(H)$,} \\
        +\infty &\text{if otherwise.}
    \end{cases}
\end{displaymath}
and
\begin{displaymath}
\E_{2}(v):=
\begin{cases}
        \tilde{j}(v(0)-\varphi)&\text{if  $v \in \WE_{\frac{1-s}{2s},\frac{1}{2}}^{1,2}(H)$,} \\
        +\infty &\text{ if otherwise.}
    \end{cases}
\end{displaymath}
Then, the functional $\E : L^{2}_{\underline{s}}(H)\to \R\cup\{+\infty\}$ defined
by $\E=\E_{1}+\E_{2}$ is proper, convex and lower semicontinuous on
$L^{2}_{\underline{s}}(H)$. In particular, the subdifferential
$\partial_{L^{2}_{\underline{s}}}\E$ of $\E$ is a mapping
$\partial_{L^{2}_{\underline{s}}}\E  : D(\partial\E)\to L^{2}_{\underline{s}}(H)$ given by
  \begin{displaymath}
    \partial_{L^{2}_{\underline{s}}}\E
    =\Big\{(v,-z^{-\frac{1-2s}{s}}v'')\in 
    L^{2}_{\underline{s}}(H)\times L^{2}_{\underline{s}}(H)\ \Big\vert\ v'(0)
    \in \partial_{H}\tilde{j}(v(0)-\varphi) \Big\}.
\end{displaymath}
In particular, for every $v\in D(\partial_{L^{2}_{\underline{s}}}\E)$, one has that
\begin{displaymath}
  v\in W^{2,2}_{\frac{1-s}{2s},\frac{1}{2},\frac{3s-1}{2s}}(H)
 \cap 
  C^{1}([0,+\infty);H).
\end{displaymath}
\end{proposition}


\begin{proof}[Proof of Proposition~\ref{prop:1}]
  It is clear that $\E$ is convex, and $\E$ is proper since by
  Lemma~\ref{lem:ibp}, for every $\varphi\in H$, there is a $v\in
  W^{1,2}_{\frac{1-s}{2s},\frac{1}{2}}(H)$ satisfying
  $v(0)=\varphi$ and $\E(v)=\E_{1}(v)$ is finite. To see that $\E$ is
  lower semicontinuous on $L^{2}_{\underline{s}}(H)$, let
  $c\in \R$ and $(v_{n})_{n\ge 1}$ be a sequence in
  $L^{2}_{\underline{s}}(H)$ such that $v_{n}\to v$ in
  $L^{2}_{\underline{s}}(H)$ for some
  $v\in L^{2}_{\underline{s}}(H)$ and satisfying
  $\E(v_{n})\le c$ for all $n\ge 1$. Since $\tilde{j}\ge 0$, this
  implies that $(v_{n})_{n\ge 1}$ is bounded in
  $W^{1,2}_{\frac{1-s}{2s},\frac{1}{2}}(H)$. Since
  $W^{1,2}_{\frac{1-s}{2s},\frac{1}{2}}(H)$ is reflexive, one
  can conclude that
  $v\in W^{1,2}_{\frac{1-s}{2s},\frac{1}{2}}(H)$ and there is
  a subsequence of $(v_{n})_{n\ge 1}$, which we denote, for
  simplicity, again by $(v_{n})_{n\ge 1}$ such that $v'_{n}$
  converges weakly to $v'$ in $L^{2}(H)$. By
  Lemma~\ref{lem:ibp}, the trace map
  $\textrm{Tr} : W^{1,2}_{\frac{1-s}{2s},\frac{1}{2}}(H)\to H$
  is linearly bounded and so, $v_{n}(0)$
  converges weakly to $v(0)$ in $H$. Therefore, and since $\E$ is
  convex, we can conclude that $v\in D(\E)$ and $\E(v)\le c$. 
  It remains to characterize the subdifferential
  \begin{displaymath}
    \partial_{L^{2}_{\underline{s}}}\E:=\Bigg\{(v,w)\in
    L^{2}_{\underline{s}}(H)\times L^{2}_{\underline{s}}(H)\ \Bigg\vert\
    \begin{array}[c]{c}
\E(\hat{v})-\E(v)
    \ge (w, \hat{v}-v)_{L^{2}_{\underline{s}}(H)}\\
      \hspace{2.2cm}\text{ for all } \hat{v} \in D(\E)
    \end{array}
\Bigg\}.
\end{displaymath}
For every $v\in D(\E)$, the weak derivative
$v'\in L^{2}(H)$. Hence $v\in C([0,+\infty);H)$. Now, let
$(v,w)\in \partial_{L^{2}_{\underline{s}}}\E$ and take
$\hat{v}=v+\varepsilon\xi$ for $\varepsilon \in \R$ and
$\xi \in D(\E)$. Then,
\begin{displaymath}
      \E(v+\varepsilon\xi)-\E(v) \ge \varepsilon\,(w,\xi)_{L^{2}_{\underline{s}}(H)}.
\end{displaymath}
Suppose first that $\varepsilon>0$. Then, dividing the above
inequality by $\varepsilon$ gives
\begin{equation}\label{eq:16}
  \begin{split}
     &\int_{0}^{+\infty}
        \frac{\frac{1}{2}\norm{v'(z)+\varepsilon\xi'(z)}^2_H -
     \frac{1}{2}\norm{v'(z)}^2_H}{\varepsilon}\,\dz\\
     &\hspace{0.4cm} +
       \frac{\tilde{j}(v(0)+\varepsilon\xi(0)-\varphi)-\tilde{j}(v(0)-\varphi)}{\varepsilon}
       \ge \int_{0}^{+\infty}(w(z),\xi(z))_H z^{\frac{1-2s}{s}} \dz.
  \end{split}
\end{equation}
Since
\begin{displaymath}
     \lim_{\varepsilon\to 0+} \int_{0}^{+\infty}
        \frac{\frac{1}{2}\norm{v'(z)+\varepsilon\xi'(z)}^2_H -
     \frac{1}{2}\norm{v'(z)}^2_H}{\varepsilon}\,\dz=
   \int_{0}^{+\infty} (v',\xi')_{H}\,\dz,
\end{displaymath}
and $\tilde{j}$ is convex, we can conclude by sending
$\varepsilon\to 0+$ in~\eqref{eq:16} that
\begin{align*}
    &\int_{0}^{+\infty} (v',\xi')_{H}\,\dz+\inf_{\varepsilon>0}
    \frac{\tilde{j}(v(0)+\varepsilon\xi(0)-\varphi)-\tilde{j}(v(0)-\varphi)}{\varepsilon}\\
    &\hspace{6cm}\ge \int_{0}^{+\infty}(w(z),\xi(z))_H z^{\frac{1-2s}{s}} \dz.
\end{align*}
In particular, we have that
\begin{equation}\label{eq:17}
    \begin{split}
    &  \int_{0}^{+\infty} (v',\xi')_{H}\,\dz+
    \tilde{j}(v(0)+\xi(0)-\varphi)-\tilde{j}(v(0)-\varphi)\\
      &\hspace{6cm}\ge \int_{0}^{+\infty}(w(z),\xi(z))_H z^{\frac{1-2s}{s}} \dz.
    \end{split}
\end{equation}
Note, for any $\xi \in C^1_c((0,+\infty);H)$, the sum
$v+\varepsilon\xi$ belongs to $D(\E )$ and 
\begin{displaymath}
     \frac{\tilde{j}(v(0)+\epsilon\xi(0)-\varphi)-\tilde{j}(v(0)-\varphi)}{\epsilon}
        = 0.
\end{displaymath}
Thus, proceeding as before with $\xi \in C^1_c((0,+\infty);H)$ and
$\varepsilon<0$, one obtains
\begin{align*}
      \int_{0}^{+\infty}
      (v',\xi')_{H}\,\dz&=\int_{0}^{+\infty}(w(z),\xi(z))_H z^{\frac{1-2s}{s}}
      \dz\\
       &=-\int_{0}^{+\infty}(-z^{\frac{1-2s}{s}}w(z),\xi(z))_H \dz.
\end{align*}
Since this equality holds for all $\xi \in C^1_c((0,+\infty);H)$, we
have thereby shown that $-z^{\frac{1-2s}{s}}w=v''$
and $v\in
W^{2,2}_{\frac{1-s}{2s},\frac{1}{2},\frac{3s-1}{2s}}(H)$. Moreover, 
$W^{2,2}_{\frac{1-s}{2s},\frac{1}{2},\frac{3s-1}{2s}}(H)$ is a linear subspace of
$W^{2,2}_{loc}((0,+\infty);H)$. Thus, one also has that
\begin{displaymath}
  v\in C^{1}((0,+\infty);H)\cap W^{2,2}_{loc}((0,+\infty);H).
\end{displaymath}
In addition, since $v'\in
W^{1,2}_{\frac{1}{2},\frac{3s-1}{2s}}(H)$,
Lemma~\ref{lem:ibp} says that $v'(0):=\lim_{z\to 0+}v'(z)$ exists in
$H$. Thus, $v\in C^{1}([0,+\infty);H)$.

Further, by Lemma~\ref{lem:ibp}, for given $\xi_0 \in H$, there is a $\xi
\in W^{1,2}_{\frac{1-s}{2s},\frac{1}{2}}(H)$ satisfying 
$\xi(0)=\xi_0$ in $H$. Suppose $v(0)+\xi_0\in D(\E_{2})$ (otherwise,
inequality~\eqref{eq:37} below always holds). Then, inserting $w=-z^{-\frac{1-2s}{s}}v''$ into~\eqref{eq:17}
and integrating by parts (Lemma~\ref{lem:ibp}) on the right hand side
of the same inequality yields that
\begin{align*}
        \int_{0}^{+\infty}(v'(z),\xi'(z))_H \,\dz &+ \tilde{j}(v(0)+\xi_{0}-\varphi)-\tilde{j}(v(0)-\varphi)\\
        &\ge -\int_{0}^{+\infty}(v''(z),\xi(z))_H \,\dz\\
        &= -(v'(0),\xi(0))_{H} + \int_{0}^{+\infty}(v'(z),\xi'(z))_{H}\,\dz\\
        &= (v'(0),\xi_0)_H + \int_{0}^{+\infty}(v'(z),\xi'(z))_{H}\,\dz,
\end{align*}
from where we obtain that
\begin{equation}
  \label{eq:37}
        \tilde{j}(v(0)+\xi_0-\varphi)-\tilde{j}(v(0)-\varphi) 
        \ge (v'(0),\xi_0)_H. 
\end{equation}
Since inequality~\eqref{eq:37} holds for arbitrary $\xi_{0}\in H$, we
have thereby shown that
$v'(0) \in \partial_{H} \tilde{j}(v(0)-a)$. This completes the
proof of this proposition.
\end{proof}

With these preliminaries in mind, we focus now on proving existence of solutions
of problem~\eqref{eq:2} ($2$nd part of Theorem~\ref{thm:2}). First, we give a briefly sketch
of the existence proof.\medskip
 
Let $\varphi\in D(A)$. Then the strategy of proving existence
of solutions to~\eqref{eq:2} is lifting equation~\eqref{eq:80} in $H$
to the following abstract equation
  \begin{equation}\label{eq:19}
    \A_{loc}v-z^{-\frac{1-2s}{s}}v''\ni 0\qquad\text{in $L^{2}_{loc}(H)$,}
  \end{equation}
  where 
  \begin{displaymath}
    \A_{loc}:=\Big\{(v,w)\in L^{2}_{loc}(H)\times 
  L^{2}_{loc}(H)\;\Big\vert\; w(z)\in
  A(v(z)) \text{ for a.e. $z\ge 0$} \Big\}.
  \end{displaymath}
Existence of a solutions $v$ of~\eqref{eq:19} satisfying
\begin{equation}
  \label{eq:53}
  v'(0)\in \partial_{H}\tilde{j}(v(0)-\varphi)
\end{equation}
is shown in two steps: first, let
\begin{equation}
  \label{eq:38}
  \A:=\Big\{(v,w)\in L^{2}_{\underline{s}}(H)\times 
  L^{2}_{\underline{s}}(H)\;\Big\vert\; w(z)\in
  A(v(z)) \text{ for a.e. $z\ge 0$} \Big\}.
\end{equation}
Then, one shows that for every $\lambda$, $\delta>0$, the following
regularized equation
 \begin{equation}
    \label{eq:6}
    \A_{\lambda}v_{\lambda}+\delta v_{\lambda}+
    \partial_{L^{2}_{\underline{s}}(H)}\E(v_{\lambda})=
    0\qquad\text{in $L^{2}_{\underline{s}}(H)$}
  \end{equation}
  admits a (unique) solution $v_{\lambda}$. Here,
  $\A_{\lambda}:=
  \tfrac{1}{\lambda}(I_{L^{2}_{\underline{s}}(H)}-J_{\lambda}^{\A})$
  denotes the Yosida approximation of $\A$ in
  $L^{2}_{\underline{s}}(H)$. After establishing \emph{a
    priori} estimates on $(v_{\lambda})_{\lambda>0}$, one can conclude
  that for every $\delta>0$, there is a subsequence of
  $(v_{\lambda})_{\lambda>0}$ converging to a (unique) solution $v_{\delta}$ of
\begin{equation}
    \label{eq:6bis}
    \A v_{\delta}+\delta v_{\delta}+\partial_{L^{2}_{\underline{s}}(H)}\E(v_{\delta})\ni
    0\qquad\text{in $L^{2}_{\underline{s}}(H)$.}
\end{equation}
After establishing \emph{a priori} estimates on
$(v_{\delta})_{\delta\in (0,1]}$, one shows that there is a
subsequence of $(v_{\delta})_{\delta\in (0,1]}$ converging to a
solution $v$ of~\eqref{eq:19} satisfying~\eqref{eq:53}. This method
generalizes an idea by Brezis~\cite{MR0317123} to the general
fractional power case $0<s<1$.

\begin{proof}[Proof of Theorem~\ref{thm:2} (existence)]
  We begin by taking $\varphi\in D(A)$. By hypothesis, $A$ is a
  maximal monotone operator on $H$ satisfying
  $0_{H}\in \textrm{Rg}(A)$. For simplicity, we may assume without
  loss of generality that $0\in A0$, otherwise we replace $A$ by
  $\tilde{A}:=A(\cdot+y)$ for some $y\in A^{-1}(\{0\})$. Then, the
  corresponding operator $\A$ on $L^{2}_{\underline{s}}(H)$
  given by~\eqref{eq:38} is maximal monotone
  (cf~\cite{MR0348562}). Moreover, the Yosida approximation
  $\A_{\lambda}$ of $\A$ is a maximal monotone and Lipschitz
  continuous mapping on $L^{2}_{\underline{s}}(H)$. Since
  $\partial_{L^{2}_{\underline{s}}}\E$ is also maximal
  monotone operators on $L^{2}_{\underline{s}}(H)$,
  \cite[Lemme~2.4]{MR0348562} implies that for every $\lambda$ and
  $\delta>0$, problem~\eqref{eq:6}
has a strong solution $v_{\lambda}\in L^{2}_{\underline{s}}(H)$. By
Proposition~\ref{prop:1},
\begin{displaymath}
  v_{\lambda}\in W^{2,2}_{\frac{1-s}{2s},\frac{1}{2},\frac{3s-1}{2s}}(H)\cap 
  C^{1}([0,+\infty);H).
\end{displaymath}
In addition, since the
Yosida approximation $A_{\lambda}$ of $A$ is Lipschitz continuous on
$H$, we can conclude from~\eqref{eq:6} that $v_{\lambda}\in C^{2}((0,+\infty);H)$.\medskip

\underline{1. \emph{A priori} estimates on $(v_{\lambda})_{\lambda>0}$.}
 The following estimates hold uniformly for all $\lambda>0$ and
 $0<\delta\le 1$:
 \begin{align}
   \label{eq:70}
   \norm{v_{\lambda}(z)}_{H}&\le
                                 \norm{v_{\lambda}(\hat{z})}_{H}\quad\text{for
                                 all $z\ge \hat{z}\ge 0$,}\\
      \label{eq:41}
                      \frac{1}{\sqrt{2}}
                      \norm{v_{\lambda}(0)}_H&\le C,\\
     \label{eq:30}
  \norm{v'_{\lambda}(0)}_{H}^{\frac{1}{2}}&\le \Big(\norm{A^{0}\varphi}_{H}+
\delta\,\norm{\varphi}_{H}\Big)^{\frac{1}{2}} +\norm{\varphi}_H^{\frac{1}{2}},\\   \label{eq:14}
   \norm{z^{\frac{1}{2}}v'_{\lambda}}_{L^{2}_{\star}(H)} &\le \left(\Big(\norm{A^{0}\varphi}_{H}+
\delta\,\norm{\varphi}_{H}\Big)^{\frac{1}{2}} +\norm{\varphi}_{H}^{2}\right)\,\norm{\varphi}_{H}^{\frac{1}{2}},\\
\label{eq:31}
   \norm{v''_{\lambda}}_{L^{2}_{\overline{s}}(H)}
   &\le \Big(\norm{A^{0}\varphi}_{H}+
\delta\,\norm{\varphi}_{H}\Big) +\norm{\varphi}_H^{\frac{1}{2}}\,\Big(\norm{A^{0}\varphi}_{H}+
\delta\,\norm{\varphi}_{H}\Big)^{\frac{1}{2}},\\ \label{eq:11}
   \norm{\A_{\lambda}v_{\lambda}}_{L^{2}_{\underline{s}}(H)}&\le \Big(\norm{A^{0}\varphi}_{H}+
\delta\,\norm{\varphi}_{H}\Big) +\norm{\varphi}_H^{\frac{1}{2}}\,\Big(\norm{A^{0}\varphi}_{H}+
\delta\,\norm{\varphi}_{H}\Big)^{\frac{1}{2}},\\ 
   \label{eq:60}
     \norm{v'_{\lambda}(z)}_{H}&\le
                                 \norm{v'_{\lambda}(\hat{z})}_{H}\quad\text{for
                                 all $z\ge \hat{z}\ge 0$,}
 \end{align}
 where $C$ is a constant independent of $\lambda$. To show that these inequalities
 hold, we first multiply~\eqref{eq:6} by $v_{\lambda}$ with respect
 to the $L^{2}_{\underline{s}}(H)$-inner product~\eqref{eq:39}. Then,
  \begin{displaymath}
        (\A_{\lambda} v_{\lambda}, v_{\lambda})_{L^{2}_{\underline{s}}(H)} +
            \delta\,\norm{v_{\lambda}}^{2}_{L^{2}_{\underline{s}}(H)} + 
            (\partial_{L^{2}_{\underline{s}}(H)}\E(v_{\lambda}),
            v_{\lambda})_{L^{2}_{\underline{s}}(H)}=0.
   \end{displaymath}
   Since $\partial_{L^{2}_{\underline{s}}}\E(v_{\lambda}) =-z^{-\frac{1-2s}{s}}
   v''_{\lambda}$, the last equation is equivalent to
   \begin{equation}\label{eq:20}
          0 = (\A_{\lambda} v_{\lambda}, v_{\lambda})_{L^{2}_{\underline{s}}(H)} +
        \delta\,\norm{v_{\lambda}}^{2}_{L^{2}_{\underline{s}}(H)}
        -( z^{-\frac{1-2s}{s}} v''_{\lambda}, v_{\lambda})_{L^{2}_{\underline{s}}(H)}
   \end{equation}
   By Cauchy-Schwarz's inequality,
   \begin{displaymath}
        \abs{(z^{-\frac{1-2s}{s}} v''_{\lambda}, v_{\lambda})_{L^{2}_{\underline{s}}(H)}}\le
        \norm{z^{-\frac{1-2s}{s}}
          v''_{\lambda}}_{L^{2}_{\underline{s}}(H)}\,\norm{v_{\lambda}}_{L^{2}_{\underline{s}}(H)}=
        \norm{v''_{\lambda}}_{L^{2}_{\overline{s}}}\,\norm{v_{\lambda}}_{L^{2}_{\underline{s}}(H)},
   \end{displaymath}
   and since
   \begin{displaymath}
        (z^{-\frac{1-2s}{s}} v''_{\lambda},
      v_{\lambda})_{L^{2}_{\underline{s}}(H)} = 
      \int_{0}^{\infty}(v''_{\lambda}(z), v_{\lambda}(z))_{H}\,\dz=(v''_{\lambda},v_{\lambda})_{L^{2}(H)},
  \end{displaymath}
  one has that
  \begin{equation}
    \label{eq:21}
    (v''_{\lambda},v_{\lambda})_{H}\in L^{1}(\R_{+}).
  \end{equation}
  On the other hand, \eqref{eq:6} is equivalent to
  \begin{equation}
    \label{eq:10}
    v''_{\lambda}(z) = z^{\frac{1-2s}{s}} A_{\lambda}v_{\lambda}(z) + z^{\frac{1-2s}{s}}\,\delta\,
    v_{\lambda}(z)\qquad\text{for every $z>0$.}
  \end{equation}
  Multiplying~\eqref{eq:10} by $v_{\lambda}$ with respect to
  the $H$-inner product applying the monotonicity of $A_{\lambda}$ and that $0\in
  A_{\lambda}0$, one sees that
\begin{equation}
  \label{eq:23}
  \begin{split}
    (v''_{\lambda}(z),v_{\lambda}(z))_{H} &=
    z^{\frac{1-2s}{s}} (A_{\lambda}v_{\lambda}(z),
    v_{\lambda}(z))_{H} +
    z^{\frac{1-2s}{s}}\,\delta\,\norm{v_{\lambda}(z)}^2_{H}\\
    &\ge
    z^{\frac{1-2s}{s}}\,\delta\,\norm{v_{\lambda}(z)}^2_{H}
  \end{split}
\end{equation}
for every $z>0$. Thus,
\begin{equation}
  \label{eq:22}
  \int_{z}^{\infty} (v''_{\lambda}(r),v_{\lambda}(r))_{H}\,\dr
  \ge \delta\,\int_{z}^{\infty}\norm{v_{\lambda}(r)}^2_{H}\,
  r^{\frac{1-2s}{s}}\,\dr\ge 0
\end{equation}
for every $z\ge 0$.
 Further, since
 \begin{displaymath}
   \frac{\td}{\dz} (v'_{\lambda}(z),v_{\lambda}(z))_{H} =
   (v''_{\lambda}(z),v_{\lambda}(z))_{H} + \norm{v'_{\lambda}(z)}^2_{H}
 \end{displaymath}
 for every $z>0$ and since $v_{\lambda}\in D(\E)$ requires that
 $\norm{v'_{\lambda}}^2_{H}\in L^{1}(\R_{+})$, it follows from
 \eqref{eq:21} that the function
 $z\mapsto (v'_{\lambda}(z),v_{\lambda}(z))_{H}$ is continuous on
 $[0,+\infty)$ and by~\eqref{eq:22} that
\begin{equation}
  \label{eq:13}
  - (v'_{\lambda}(z), v_{\lambda}(z))_{H} =
  \int_{z}^{+\infty} (v''_{\lambda}(r),v_{\lambda}(r))_{H}\, \dr +
  \int_{z}^{+\infty} \norm{v'_{\lambda}(r)}^2_{H}\, \dr\ge 0
 \end{equation}
for every $z\ge 0$. Therefore,
\begin{equation}
  \label{eq:35}
  (v'_{\lambda}(z),v_{\lambda}(z))_{H}
  \le 0\qquad\text{for every $z\ge 0$}
\end{equation}
and since
\begin{displaymath}
  \frac{\td}{\dz}\frac{1}{2}\norm{v_{\lambda}(z)}_{H}^{2}=(v'_{\lambda}(z),v_{\lambda}(z))_{H},
\end{displaymath}
the function $z\mapsto \frac{1}{2}\norm{v_{\lambda}(z)}_{H}^{2}$ is
decreasing on $[0,+\infty)$, implying that~\eqref{eq:70} holds. Next,
by~\eqref{eq:23}, one has that
\begin{displaymath}
        \frac{\td^2}{\td z^2} \frac{1}{2} \norm{v_{\lambda}(z)}^2_{H}
        = (v''_{\lambda}(z),v_{\lambda}(z))_{H} + \norm{v'_{\lambda}(z)}^2_{H} 
        \ge \norm{v'_{\lambda}(z)}^2_{H}\ge 0
\end{displaymath}
for all $z>0$.  Hence, the function $z \mapsto \norm{v_{\lambda}(z)}^2_{H}$
is convex on $[0,+\infty)$. Taking $z=0$ in~\eqref{eq:13} and
applying~\eqref{eq:22}, one finds
\begin{displaymath}
   (v'_{\lambda}(0), v_{\lambda}(0))_{H}
       +  \int_{0}^{+\infty} \norm{v'_{\lambda}(z)}^2_{H}\, \dz =
       -(v''_{\lambda}, v_{\lambda})_{L^{2}(H)} \le 0.
\end{displaymath}
Therefore, and since
$v'_{\lambda}(0) \in \partial_{H} \tilde{j}(v_{\lambda}(0)-\varphi)$ and
$\partial_{H}\tilde{j}$ is monotone with
$0\in \partial_{H}\tilde{j}(0)$, we get that
\begin{align*}
  \norm{z^{\frac{1}{2}}v'_{\lambda}}^{2}_{L^{2}_{\star}(H)}
     & =\int_{0}^{+\infty} \norm{v'_{\lambda}(z)}^2_{H}\,\dz\\ 
    &\le - (v'_{\lambda}(0), v_{\lambda}(0))_{H} \\
    &= -(v'_{\lambda}(0)-0, (v_{\lambda}(0)-\varphi)-0)_{H} - (v'_{\lambda}(0),\varphi)_{H}\\
    &\le 0 - (v'_{\lambda}(0), \varphi)_{H}
\end{align*}
and so, by Cauchy-Schwarz's inequality, one obtains
\begin{equation}
  \label{eq:43}
     \norm{z^{\frac{1}{2}}v'_{\lambda}}_{L^{2}_{\star}(H)} \le 
     \norm{v'_{\lambda}(0)}_{H}^{\frac{1}{2}}\,\norm{\varphi}_H^{\frac{1}{2}}.
\end{equation}

By the Lipschitz continuity of $A_{\lambda} : H \to H$, the
 function
 \begin{displaymath}
   w_{\lambda}(z):=
   A_{\lambda}v_{\lambda}(z) + \delta v_{\lambda}(z)
\end{displaymath}
is differentiable at almost everywhere $z\in \R_{+}$ with weak derivative
 \begin{displaymath}
   w'_{\lambda}(z) =\frac{\td}{\dz}A_{\lambda}(v_{\lambda}(z)) +
        \delta\, v'_{\lambda}(z)\quad\text{ for almost every $z\in \R_{+}$,}
 \end{displaymath}
where $\tfrac{\td}{\dz} A_{\lambda}(v_{\lambda})$ is the weak
 derivative of $z\mapsto A_{\lambda}(v_{\lambda}(z))$. On the other
 hand, by~\eqref{eq:10},
 \begin{displaymath}
   w_{\lambda}(z)=z^{-\frac{1-2s}{s}}v''_{\lambda}(z).
\end{displaymath}
Hence,
 \begin{equation}
   \label{eq:24}
   \frac{\td}{\dz}  (w_{\lambda}(z), v'_{\lambda}(z))_{H} =(w'_{\lambda}(z),
   v'_{\lambda}(z))_{H}+ \norm{z^{-\frac{3s-1}{2s}}v''_{\lambda}(z)}_{H}^{2}\frac{1}{z} 
 \end{equation}
 for almost every $z\in \R_{+}$. Since the Yosida approximation
 $A_{\lambda}$ is Lipschitz continuous with
 constant $1/\lambda$ (cf~\cite[Proposition~2.6]{MR0348562}), one has that
 \begin{displaymath}
   \lnorm{\frac{\td}{\dz} A_{\lambda}(v_{\lambda}(z))}_{H}\le \frac{1}{\lambda}\,\norm{v'_{\lambda}(z)}_{H}\qquad\text{for a.e.
 $z\in \R_{+}$.}
\end{displaymath}
Therefore and since
 $v_{\lambda}\in W^{2,2}_{\frac{1-s}{2s},\frac{1}{2},\frac{3s-1}{2s}}(H)$, 
\eqref{eq:24} means that the function
 $z\mapsto (w_{\lambda}(z), v'_{\lambda}(z))_{H}$ is absolutely
 continuous on $[0,+\infty)$ and
 \begin{equation}
\label{eq:28}
   -(w_{\lambda}(z), v'_{\lambda}(z))_{H}=\int_{z}^{+\infty} \frac{\td}{\dr}  (w_{\lambda}(r), v'_{\lambda}(r))_{H}\,\dr
 \end{equation}
 for every $z\ge 0$. Moreover, by the monotonicity of $A_{\lambda}$
 \begin{align*}
  & \left(\frac{\td}{\dz} A_{\lambda}(v_{\lambda}(z)),
   v'_{\lambda}(z)\right)_{H}\\
   &\quad=\lim_{h\to
     0}\left(\frac{A_{\lambda}(v_{\lambda}(z+h))-A_{\lambda}(v_{\lambda}(z))}{h}, 
     \frac{v_{\lambda}(z+h)-v_{\lambda}(z+h)}{h}\right)_{H} \ge 0
\end{align*}
for almost every $z\in \R_{+}$. Therefore, one has that
 \begin{equation}
   \label{eq:27}
   (w'_{\lambda}(z), v'_{\lambda}(z))_{H} = \left(\frac{\td}{\dz}A_{\lambda}(v_{\lambda}(z)),
        v'_{\lambda}(z)\right)_{H} + \delta\,\norm{v'_{\lambda}(z)}_{H}^2 \ge 0
 \end{equation}
for almost every $z\in \R_{+}$. Applying this inequality
to~\eqref{eq:24} and subsequently, inserting~\eqref{eq:24}
into~\eqref{eq:28} yields
\begin{displaymath}
  (w_{\lambda}(z), v'_{\lambda}(z))_{H}\le 0\qquad\text{for every $z\ge 0$.}
\end{displaymath}
From this, it follows that
\begin{displaymath}
  \frac{\td}{\dz}\frac{1}{2}\norm{v'_{\lambda}(z)}_{H}^{2} =
  (v''_{\lambda}(z),v'_{\lambda}(z))_{H}
  =z^{\frac{1-2s}{s}} (w_{\lambda}(z), v'_{\lambda}(z))_{H}\le 0,
\end{displaymath}
implying that $z\mapsto \norm{v'_{\lambda}(z)}_{H}^{2}$ is
non-increasing on $[0,+\infty)$ and, in particular, \eqref{eq:60}
holds.  Moreover, applying~\eqref{eq:27} to~\eqref{eq:24}, gives
\begin{displaymath}
   \frac{\td}{\dz}  (w_{\lambda}(z), v'_{\lambda}(z))_{H} \ge
   \norm{v''_{\lambda}(z)}_{H}^{2} z^{-\frac{1-2s}{s}}\qquad
   \text{for a.e. $z\in \R_{+}$}
 \end{displaymath}
and by integrating this inequality over $\R_{+}$, one finds that 
\begin{displaymath}
  \norm{v''_{\lambda}}_{L^{2}_{\overline{s}^{\star}(H)}}^{2}
  \le  - (w_{\lambda}(0), v'_{\lambda}(0))_{H}. 
\end{displaymath}
Since $w_{\lambda}=A_{\lambda}v_{\lambda} + \delta
v_{\lambda}$ and by applying~\cite[Proposition~4.7(iii)]{MR0348562}) to
$u'_{\lambda}(0)\in \partial\tilde{j}(u_{\lambda}(0)-\varphi)$, one
sees that
\begin{align*}
       \norm{v''_{\lambda}}_{L^{2}_{\overline{s}^{\star}(H)}}^{2}&\le - (w_{\lambda}(0), v'_{\lambda}(0))_{H}\\
         &= -(A_{\lambda}v_{\lambda}(0)+\delta v_{\lambda}(0),v'_{\lambda}(0))_{H}\\
         &\le -(A_{\lambda}\varphi+\delta \varphi,
           v'_{\lambda}(0))_{H},
\end{align*}
and if $A^0$ is the \emph{minimal selection} of $A$, then
\begin{equation}
  \label{eq:25}
  \norm{v''_{\lambda}}_{L^{2}_{\overline{s}}}
\le \norm{v'_{\lambda}(0)}^{1/2}_{H} \,\Big(\norm{A^{0}\varphi}_{H}+
\delta\,\norm{\varphi}_{H}\Big)^{1/2}.
\end{equation}
%
%

Next, given $x\in H$ satisfying $\norm{x}_{H}\le 1$. By
Lemma~\ref{lem:ibp}, there is a $\xi\in
W^{1,2}_{\frac{1-s}{2s},\frac{1}{2}}(H)$ such that
$\xi(0)=x$ in $H$. Then, by $v'\in
W^{1,2}_{\frac{1}{2},\frac{3s-1}{2s}}(H)$, Cauchy-Schwarz's
inequality gives
\begin{align*}
  (v'_{\lambda}(0),x)_{H}
  &=-\int_{0}^{\infty}\frac{\td}{\ds}(v'_{\lambda}(r),\xi(r))_{H}\dr\\
  &=-\int_{0}^{\infty}
    (v''_{\lambda}(r),\xi(r))_{H}\dr-\int_{0}^{\infty}
    (v'_{\lambda}(r),\xi'(r))_{H}\dr\\
  &\le \norm{v''_{\lambda}}_{L^{2}_{\overline{s}}}\,
    \norm{\xi}_{L^{2}_{\underline{s}}(H)}+
  \norm{z^{\frac{1}{2}}v'}_{L^{2}_{\star}(H)}\,\norm{z^{\frac{1}{2}}\xi'}_{L^{2}_{\star}(H)}\\
  &\le \left(\norm{v''_{\lambda}}_{L^{2}_{\overline{s}}} +
  \norm{z^{\frac{1}{2}}v'}_{L^{2}_{\star}(H)}\right)\,\norm{\xi}_{W^{1,2}_{\frac{1}{2},\frac{3s-1}{2s}}(H)}.
\end{align*}
Moreover, in the latter inequality, taking the infimum over all $\xi\in
W^{1,2}_{\frac{1-s}{2s},\frac{1}{2}}(H)$ satisfying
$\xi(0)=x$ and subsequently applying~\eqref{eq:42}, it follows that 
\begin{displaymath}
  (v'_{\lambda}(0),x)_{H}\le \left(\norm{v''_{\lambda}}_{L^{2}_{\overline{s}}} +
  \norm{z^{\frac{1}{2}}v'}_{L^{2}_{\star}(H)}\right)\,C\,\norm{x}_{H}.
\end{displaymath}
Now, taking the supremum over all $x\in H$ satisfying $\norm{x}_{H}\le
1$, yields
\begin{displaymath}
  \norm{v'_{\lambda}(0)}_{H}\le C\,\left(\norm{v''_{\lambda}}_{L^{2}_{\overline{s}}} +
  \norm{z^{\frac{1}{2}}v'}_{L^{2}_{\star}(H)}\right).
\end{displaymath}
Applying~\eqref{eq:25} and~\eqref{eq:43} to this inequality, one obtains
\begin{displaymath}
  \norm{v'_{\lambda}(0)}_{H}\le \left(\norm{v'_{\lambda}(0)}^{1/2}_{H} \,\Big(\norm{A^{0}\varphi}_{H}+
\delta\,\norm{\varphi}_{H}\Big)^{1/2} +\norm{v'_{\lambda}(0)}_{H}^{1/2}\,\norm{\varphi}_H^{1/2}\right),
\end{displaymath}
from where we can conclude that~\eqref{eq:30} holds. Now,
inserting~\eqref{eq:30} into~\eqref{eq:25}, one obtains~\eqref{eq:31},
and inserting~\eqref{eq:30} into~\eqref{eq:43}, one sees
that~\eqref{eq:14} holds. 

Next, by hypothesis, $\tilde{j}$ satisfies~\eqref{eq:47}, which is
equivalent to $(\partial_{\!  H}\tilde{j})^{-1}$ maps bounded sets
into bounded sets (cf~\cite[Proposition~2.14]{MR0348562}). Since each
$v_{\lambda}(0) \in (\partial_{\!
  H}\tilde{j})^{-1}(v'_{\lambda}(0))+\varphi$ and by~\eqref{eq:30},
the sequence $(v'_{\lambda}(0))_{\lambda>0}$ is bounded, we have that
there is a constant $C>0$ such that~\eqref{eq:41} holds. 


To see that \emph{a priori} estimate~\eqref{eq:11} holds, we
multiply~\eqref{eq:6} by $\A_{\lambda}v_{\lambda}$ with respect to the
$L^{2}_{\underline{s}^{\star}}(H)$-inner product. Then, by the
monotonicity of $\A_{\lambda}$ and since
$\partial_{L^{2}_{\underline{s}}}\E(v_{\lambda})
=-z^{-\frac{1-2s}{s}} v''_{\lambda}$, one sees that
    \begin{align*}
      \norm{\A_{\lambda}v_{\lambda}}^2_{L^{2}_{\underline{s}}(H)}
        &\le
          \norm{\A_{\lambda}v_{\lambda}}^2_{L^{2}_{\underline{s}}(H)}
          +\delta\,(v_{\lambda}, \A_{\lambda}v_{\lambda})_{L^{2}_{\underline{s}}(H)}\\ 
      &=
        (z^{-\frac{1-2s}{s}} v''_{\lambda},\A_{\lambda}v_{\lambda})_{L^{2}_{\underline{s}}(H)}\\
      & \le \norm{v''_{\lambda}}_{L^{2}_{\overline{s}}(H)} 
          \, \norm{\A_{\lambda}v_{\lambda}}_{L^{2}_{\underline{s}}(H)}.
    \end{align*}
  Therefore,
  \begin{displaymath}
    \norm{\A_{\lambda}v_{\lambda}}\le \norm{v''_{\lambda}}_{L^{2}_{\overline{s}}(H)} 
  \end{displaymath}
  for all $\lambda>0$ and so, by~\eqref{eq:31}, one
  gets~\eqref{eq:11}. 
\medskip 

\underline{2. For every $\delta>0$, there is a unique solution $v_{\delta}$
  of~\eqref{eq:6bis} and $v_{\lambda}\to v_{\delta}$.} To establish
the existence of a solution $v_{\delta}$
  of~\eqref{eq:6bis}, we begin with the following convergence result.

 \begin{lemma}\label{lem:3}
    For every $\delta>0$, the sequence $(v_{\lambda})_{\lambda>0}$
    of solutions $v_{\lambda}$ of~\eqref{eq:6}, is a Cauchy sequence
    in $L^{2}_{\underline{s}}(H)$. In particular, there is a
    $v_{\delta}\in L^{2}_{\underline{s}}(H)$ such that
    \begin{equation}
      \label{eq:40}
      \lim_{\lambda\to 0+}v_{\lambda}=v_{\delta}\qquad\text{ in
       $L^{2}_{\underline{s}}(H)$.} 
   \end{equation}
 \end{lemma}
    
  \begin{proof}[Proof of Lemma~\ref{lem:3}]
    For $\lambda$, $\hat{\lambda} > 0$, let $v_{\lambda}$ and
    $v_{\hat{\lambda}}$ be two solutions of~\eqref{eq:6}. Then,
    multiplying 
    \begin{displaymath}
        \delta
        (v_{\lambda}-v_{\hat{\lambda}})=-(\A_{\lambda}v_{\lambda}-\A_{\hat{\lambda}}v_{\hat{\lambda}})-
        (\partial_{L^{2}_{\underline{s}}}\E(v_{\lambda})-\partial_{L^{2}_{\underline{s}}}\E(v_{\hat\lambda}))
    \end{displaymath}
    by $v_{\lambda}-v_{\hat\lambda}$ with respect to the
    $L^{2}_{\underline{s}}(H)$-inner product and using that $\partial_{L^{2}_{\underline{s}}}\E$ is monotone,
    shows that
    \begin{displaymath}
      \delta\,\norm{v_{\lambda}-v_{\hat{\lambda}}}^2_{L^{2}_{\underline{s}}(H)} \le -(\A_{\lambda}v_{\lambda}
        -\A_{\hat{\lambda}}v_{\hat{\lambda}}, v_{\lambda}-v_{\hat{\lambda}})_{L^{2}_{\underline{s}}(H)}.
    \end{displaymath}
    We recall from~\cite[p~28]{MR0348562} that for the resolvent operator $J^{\A}_{\lambda}$ of
    $\A$, one has that $\A_{\lambda}u\in \A J^{\A}_{\lambda}u$. Thus,
    by the monotonicity of $\A$, one has that
    \begin{align*}
        \delta\,\norm{v_{\lambda}-v_{\hat{\lambda}}}^2_{L^{2}_{\underline{s}}(H)} &\le -(\A_{\lambda}v_{\lambda}
        -\A_{\hat{\lambda}}v_{\hat{\lambda}}, v_{\lambda}-v_{\hat{\lambda}})_{L^{2}_{\underline{s}}(H)}\\
        &= -(\A_{\lambda}v_{\lambda}-\A_{\hat{\lambda}}v_{\hat\lambda},
        (v_{\lambda}-J^{\A}_{\lambda}v_{\lambda})-(v_{\hat\lambda}-
        J^{\A}_{\hat\lambda}v_{\hat\lambda}))_{L^{2}_{\underline{s}}(H)}\\ 
       &\hspace{3cm}- (\A_{\lambda}v_{\lambda}-\A_{\hat\lambda}v_{\hat\lambda},
        J^{\A}_{\lambda}v_{\lambda}-J_{\hat\lambda}^{\A}v_{\hat\lambda})_{L^{2}_{\underline{s}}(H)}\\
        &\le -(\A_{\lambda}v_{\lambda}-\A_{\hat\lambda}v_{\hat\lambda},
        \lambda\,\A_{\lambda}v_{\lambda}-\hat{\lambda}\,
          \A_{\hat{\lambda}}v_{\hat{\lambda}})_{L^{2}_{\underline{s}}(H)}\\
        &= -\lambda\,\norm{\A_{\lambda}v_{\lambda}}^2_{L^{2}_{\underline{s}}(H)} +(\lambda+\hat{\lambda})\,
        (\A_{\lambda}u_{\lambda},\A_{\hat\lambda}v_{\hat\lambda})_{L^{2}_{\underline{s}}(H)}\\
        &\hspace{6cm}- \hat{\lambda}\,\norm{\A_{\hat\lambda}v_{\hat\lambda}}^2_{L^{2}_{\underline{s}}(H)}.
    \end{align*}
    By Cauchy-Schwarz's and Young's inequality,
    \begin{align*}
        &(\lambda+\hat{\lambda})\,(\A_{\lambda}v_{\lambda},\A_{\hat{\lambda}}v_{\hat{\lambda}})_{L^{2}_{\underline{s}}(H)}\\
        &\quad\le
          (\lambda+\hat{\lambda})\,\norm{\A_{\lambda}v_{\lambda}}_{L^{2}_{\underline{s}}(H)}\,
        \norm{\A_{\hat{\lambda}}v_{\hat{\lambda}}}_{L^{2}_{\underline{s}}(H)}\\
        &\quad\le \lambda\,
          \norm{\A_{\lambda}v_{\lambda}}_{L^{2}_{\underline{s}}(H)}^{2}+
          \frac{\lambda}{4}
          \norm{\A_{\hat{\lambda}}v_{\hat{\lambda}}}_{L^{2}_{\underline{s}}(H)}^{2}+
          \hat{\lambda}\,
          \norm{\A_{\hat{\lambda}}v_{\hat{\lambda}}}_{L^{2}_{\underline{s}}(H)}^{2}+
          \frac{\hat{\lambda}}{4} \norm{\A_{\lambda}v_{\lambda}}_{L^{2}_{\underline{s}}(H)}^{2}.
    \end{align*}
    Hence,
    \begin{displaymath}
        \delta\,\norm{v_{\lambda}-v_{\hat{\lambda}}}^2_{L^{2}_{\underline{s}}(H)}
      \le  \frac{\lambda}{4}
          \norm{\A_{\hat{\lambda}}v_{\hat{\lambda}}}_{L^{2}_{\underline{s}}(H)}^{2}
          +\frac{\hat{\lambda}}{4} \norm{\A_{\lambda}v_{\lambda}}_{L^{2}_{\underline{s}}(H)}^{2},
     \end{displaymath}    
     which by~\eqref{eq:11} shows that
     \begin{displaymath}
       \delta\,\norm{v_{\lambda}-v_{\hat{\lambda}}}^2_{L^{2}_{\underline{s}}(H)} 
       \le \frac{\lambda+\hat{\lambda}}{4}\Bigg[\Big(\norm{A^{0}\varphi}_{H}+
       \delta\,\norm{\varphi}_{H}\Big) +\norm{\varphi}_H^{\frac{1}{2}}\,\Big(\norm{A^{0}\varphi}_{H}+
       \delta\,\norm{\varphi}_{H}\Big)^{\frac{1}{2}}\Bigg]^{2}. 
     \end{displaymath}
    Therefore, for every $\delta>0$,
    $(v_{\lambda})_{\lambda>0}$ is a Cauchy sequence
    in~$L^{2}_{\underline{s}}(H)$. This proves the claim of
    this lemma.
  \end{proof}

  \noindent\emph{Continuation of the proof of Theorem~\ref{thm:1bis}. }
   By Lemma~\ref{lem:3}, there is a $v_{\delta}\in
  L^{2}_{\underline{s}}(H)$ such that~\eqref{eq:40}
  holds. Now, the \emph{a priori} estimates~\eqref{eq:14}
  and~\eqref{eq:31} imply that $v_{\delta}\in
  W^{2,2}_{\frac{1-s}{2s},\frac{1}{2},\frac{3s-1}{2s}}(H)$
  and after possibly passing to a subsequence of
  $(v_{\lambda})_{\lambda>0}$, which we denote again by
  $(v_{\lambda})_{\lambda>0}$, one has that
  \begin{align}
    \label{eq:44}
    &\lim_{\lambda\to0+}v'_{\lambda}=v'_{\delta}\qquad\text{weakly in
      $L^{2}(H)$,}\\ \label{eq:45}
    &\lim_{\lambda\to0+}v''_{\lambda}=v''_{\delta}\qquad\text{weakly in
      $L^{2}_{\overline{s}}(H)$.}
  \end{align}
  Moreover, since
  $\partial_{L^{2}_{\underline{s}}}\E(v_{\lambda})
  =-z^{-\frac{1-2s}{s}} v''_{\lambda}$ and
  by~\cite[Proposition~2.5]{MR0348562}, the limits~\eqref{eq:40}
  and~\eqref{eq:45} imply that
  $v_{\delta}\in D(\partial_{L^{2}_{\underline{s}}}\E)$ and
  $-z^{-\frac{1-2s}{s}}
  v''_{\delta}=\partial_{L^{2}_{\underline{s}}}\E(v_{\delta})$. Next,
  by~\eqref{eq:11}, there is a
  $\chi\in L^{2}_{\underline{s}}(H)$ and a subsequence of
  $(v_{\lambda})_{\lambda>0}$, which we denote again by
  $(v_{\lambda})_{\lambda>0}$ such that
  \begin{equation}
    \label{eq:46}
    \lim_{\lambda\to0+}\A_{\lambda}v_{\lambda}=\chi\qquad\text{weakly
      in $L^{2}_{\underline{s}}(H)$.}
  \end{equation}
  Moreover,
  \begin{align*}
    \lim_{\lambda\to 0} (\A_{\lambda}v_{\lambda},v_{\lambda})_{L^{2}_{\underline{s}}(H)}
        &= \lim_{\lambda\to 0} (\A_{\lambda}v_{\lambda},v_{\lambda}
        -v_{\delta})_{L^{2}_{\underline{s}}(H)} 
          + \lim_{\lambda \to 0}(\A_{\lambda}v_{\lambda},v_{\delta})_{ L^{2}_{\underline{s}}(H)}\\
        &= (\chi,v_{\delta})_{ L^{2}_{\underline{s}}(H)}
  \end{align*}
 Note, $J_{\lambda}^{\A}v_{\delta}\to v_{\delta}$ in
  $L^{2}_{\underline{s}}(H)$ as $\lambda\to 0+$. Thus, and since
  \begin{align*}
    \norm{J_{\lambda}^{\A}v_{\lambda}-v_{\delta}}_{L^{2}_{\underline{s}}(H)}
      &\le
        \norm{J_{\lambda}^{\A}v_{\lambda}-J_{\lambda}^{\A}v_{\delta}}_{L^{2}_{\underline{s}}(H)}
       +
        \norm{J_{\lambda}^{\A}v_{\delta}-v_{\delta}}_{L^{2}_{\underline{s}}(H)}\\
       &\le \norm{v_{\lambda}-v_{\delta}}_{L^{2}_{\underline{s}}(H)}+
         \norm{J_{\lambda}^{\A}v_{\delta}-v_{\delta}}_{L^{2}_{\underline{s}}(H)},
  \end{align*}
 one has that
 \begin{displaymath}
   \lim_{\lambda\to0+} J^{\A}_{\lambda}v_{\lambda}=v_{\delta}\qquad\text{
      in $L^{2}_{\underline{s}}(H)$.}
 \end{displaymath}
 Therefore and since
 $\A_{\lambda}v_{\lambda}\in \A J_{\lambda}^{\A}v_{\lambda}$,
 \cite[Proposition~2.5]{MR0348562} implies that $v_{\delta}\in D(\A)$
 and $\chi\in \A v_{\delta}$.  Now, by~\eqref{eq:40},~\eqref{eq:45},
 and~\eqref{eq:46}, taking the
 $L^{2}_{\underline{s}}(H)$-weak limit in~\eqref{eq:6}
 yields that for every $\delta>0$, $v_{\delta}$ is a solution
 of~\eqref{eq:6bis}, which by Proposition~\ref{prop:1} has the regularity
\begin{displaymath}
  v_{\delta}\in W^{2,2}_{\frac{1-s}{2s},\frac{1}{2},\frac{3s-1}{2s}}(H)\cap
  C^{1}([0,+\infty);H)\cap W^{2,2}_{loc}((0,+\infty);H)
\end{displaymath}


Now, let $x\in H$ and for $\rho\in C^{\infty}([0,+\infty))$ satisfying
$0\le \rho\le 1$ on $[0,+\infty)$, $\rho\equiv 1$ on $[0,1]$ and
$\rho\equiv 0$ on $[2,+\infty)$, set $\xi(z)=\rho(z)\,x$ for every
$z\ge 0$. By~\eqref{eq:40}, ~\eqref{eq:70} and~\eqref{eq:41}, one has
that
 \begin{equation}
   \label{eq:73}
   \lim_{\lambda\to0+}v_{\lambda}=v_{\delta}\qquad\text{in $L^{2}(0,T;H)$, for
     every $T>0$.}
\end{equation}
Thus, by~\eqref{eq:44} and since $v_{\delta}$  and $v$
belong to $C^{1}([0,+\infty);H)$, one has that
 \begin{align*}
   (v_{\lambda}(0),x)_{H} &= 
      -\int_{0}^{+\infty}\frac{\td}{\dz}(v_{\delta}(z),\xi(z))_{H}\,\dz\\
    &= - \int_{0}^{+\infty}(v'_{\delta}(z),\xi(z))_{H}\,\dz  -
      \int_{0}^{+\infty}(v_{\delta}(z),\xi'(z))_{H}\,\dz\\
   &\to - \int_{0}^{+\infty}(v'(z),\xi(z))_{H}\,\dz  -
      \int_{0}^{+\infty}(v(z),\xi'(z))_{H}\,\dz
   = (v(0),x)_{H} 
 \end{align*}
 as $\lambda\to0+$. Since $x\in H$ was arbitrary, this means that
 $v_{\lambda}(0)\rightharpoonup v_{\delta}(0)$ weakly in $H$ as
 $\lambda\to0+$ and hence,
\begin{displaymath}
  v_{\lambda}(z)=v_{\lambda}(0)+\int_{0}^{z}v'_{\lambda}(r)\,\dr
   \rightharpoonup v_{\delta}(0)+\int_{0}^{z}v'_{\delta}(r)\,\dr =v_{\delta}(z)
\end{displaymath}
weakly in $H$ as $\lambda\to0+$ for every $z>0$. In addition,
by~\eqref{eq:73} and~\eqref{eq:44}, 
one has that
\begin{align*}
  \frac{1}{2}\norm{v_{\lambda}(0)}_{H}^{2}
  &=-\int_{0}^{2}\frac{\td}{\dz}\frac{1}{2}\norm{v_{\lambda}(z)\rho(z)}_{H}^{2}\,\dz\\
  &=-\int_{0}^{2}(v'_{\lambda}(z),v_{\lambda}(z))_{H}\,
    \rho^{2}(z)\,\dz-\int_{0}^{2}\norm{v_{\lambda}(z)}_{H}^{2}\,\rho'(z)\rho(z)\,\dz\\
  &\to -\int_{0}^{2}(v'_{\delta}(z),v_{\delta}(z))_{H}\,
    \rho^{2}(z)\,\dz-\int_{0}^{2}\norm{v_{\delta}(z)}_{H}^{2}\,\rho'(z)\rho(z)\,\dz\\
  &=\frac{1}{2}\norm{v_{\delta}(0)}_{H}^{2}
\end{align*}
as $\lambda\to 0+$. Therefore, by the weak limit
$v_{\lambda}(0)\rightharpoonup v_{\delta}(0)$ in and since $H$ is a
Hilbert space, we have that
\begin{displaymath}
  \lim_{\lambda\to0+}v_{\lambda}(0)=v_{\delta}(0)\qquad\text{ in $H$.}
\end{displaymath}
By this,~\eqref{eq:73}, and~\eqref{eq:44}, one has that
\begin{align*}
  \frac{1}{2}\norm{v_{\lambda}(z)}_{H}^{2}
  &=\frac{1}{2}\norm{v_{\lambda}(0)}_{H}^{2}+\int_{0}^{z}\frac{\td}{\dr}\frac{1}{2}\norm{v_{\lambda}(r)}_{H}^{2}\,\dr\\
  &=\frac{1}{2}\norm{v_{\lambda}(0)}_{H}^{2}+\int_{0}^{z}(v'_{\lambda}(r),v_{\lambda}(r))_{H}\,\dr\\
  &\to
    \frac{1}{2}\norm{v_{\lambda}(0)}_{H}^{2}+\int_{0}^{z}(v'_{\lambda}(r),v_{\lambda}(r))_{H}\,\dr
    =\frac{1}{2}\norm{v_{\delta}(z)}_{H}^{2}.
\end{align*}
Therefore, and since $H$ is a Hilbert space, we have that
\begin{equation}
  \label{eq:72}
  \lim_{\lambda\to0+}v_{\lambda}(z)=v_{\delta}(z)\qquad\text{ in $H$
    for every $z\ge 0$.}
\end{equation}
\medskip

 \underline{3. \emph{A priori}-estimates on
   $(v_{\delta})_{\delta>0}$. } One has that the following
 inequalities hold for all $0<\delta\le 1$:
\allowdisplaybreaks
 \begin{align}
       \label{eq:41bis}
    \norm{z\,v'_{\delta}}_{L^{2}_{\star}(H)}
                     &\le  \frac{\norm{v_{\delta}(0)}^{2}_{H}}{2}\\ \label{eq:49}
   \norm{v_{\delta}(z)}_H&\le \norm{v_{\delta}(\hat{z})}_{H}\le \norm{v_{\delta}(0)}_{H}\le
                           C\quad(\text{for all $z\ge \hat{z}\ge 0$,})\\
      \label{eq:30bis}
   \norm{v'_{\delta}(0)}_{H}^{\frac{1}{2}}&\le \Big(\norm{A^{0}\varphi}_{H}+
\norm{\varphi}_{H}\Big)^{\frac{1}{2}} +\norm{\varphi}_H^{\frac{1}{2}},\\ 
 \label{eq:14bis}
   \norm{z^{\frac{1}{2}}v'_{\delta}}_{L^{2}_{\star}(H)} &\le \left(\Big(\norm{A^{0}\varphi}_{H}+
\delta\,\norm{\varphi}_{H}\Big)^{\frac{1}{2}} +\norm{\varphi}_{H}^{\frac{1}{2}}\right)\,\norm{\varphi}_{H}^{\frac{1}{2}},\\
\label{eq:31bis}
   \norm{v''_{\delta}}_{L^{2}_{\overline{s}}}
   &\le \Big(\norm{A^{0}\varphi}_{H}+
\delta\,\norm{\varphi}_{H}\Big) +\norm{\varphi}_H^{\frac{1}{2}}\,\Big(\norm{A^{0}\varphi}_{H}+
\delta\,\norm{\varphi}_{H}\Big)^{\frac{1}{2}},\\ \label{eq:57} 
   \norm{v'_{\delta}(z)}_{H}&\le
     \norm{v'_{\delta}(\hat{z})}_{H}\quad\text{ for all $z\ge
     \hat{z}>0$,}\\
\label{eq:68}
   &(v'_{\delta}(z),v_{\delta}(z))_{H}\le 0\qquad\text{for all $z\ge 0$.}
 \end{align} 
 where the constant $C$ is independent of $\delta$. 

 Due to the limits~\eqref{eq:45}, and~\eqref{eq:46}, sending
 $\lambda\to0+$ in~
 \eqref{eq:14} and~\eqref{eq:31} shows that~\eqref{eq:14bis}
 and~\eqref{eq:31bis} hold. Next, inequality~\eqref{eq:57} follows
 from~\eqref{eq:60} and by the limit
\begin{equation}
  \label{eq:69}
  \lim_{\lambda\to0+}v'_{\lambda}(z)=v'_{\delta}(z)\qquad\text{
    strongly in $H$ for every $z>0$.}
\end{equation}
The convergence~\eqref{eq:69} follows by the same reasoning as shown
to prove~\eqref{eq:64} in Lemma~\ref{lem:4} below. To avoid repetitive
arguments, we outline this method only once.  By the two limits~\eqref{eq:69}
and~\eqref{eq:72}, sending $\lambda\to 0+$ in~\eqref{eq:35}
yields~\eqref{eq:68} for all $z>0$ and by continuity of $v_{\delta}$
and $v'_{\delta}$, one has that~\eqref{eq:68} also holds for $z=0$. 
To see that~\eqref{eq:41bis} holds, we note that by~\eqref{eq:6bis}
and since $A$ is monotone,
\begin{displaymath}
  \int_{0}^{T}\,z\, (v''_{\delta}(z),v_{\delta}(z))_{H}\,\dz \ge 0
\end{displaymath}
for every $T>0$. By this estimate, one sees that
\allowdisplaybreaks
\begin{align*}
        T(v'_{\delta}(T),v_{\delta}(T))_{H} 
         &=\int_{0}^{T}\frac{\td}{\dz}\Big(
                                                z\,(v'_{\delta}(z),v_{\delta}(z))_{H}\Big)\, \dz \\
        &= \int_{0}^{T}(v'_{\delta}(z),v_{\delta}(z))_{H}\, \dz + \int_{0}^{T}
        z\,(v''_{\delta}(z),v_{\delta}(z))_{H}\,\dz\\
         &\hspace{3cm} + \int_{0}^{T} z\,\norm{v'_{\delta}(z)}^2_H \,\dz\\
        &\ge
           \int_{0}^{T}\frac{\td}{\dz}\frac{1}{2}\norm{v_{\delta}(z)}^2_{H}\,\dz
           + \int_{0}^{T} z\,\norm{v'_{\delta}(z)}^2_H \,\dz\\
         &= \frac{1}{2} \norm{v_{\delta}(T)}^2_H - \frac{1}{2} \norm{v_{\delta}(0)}^2_H +
         \int_{0}^{T}z\norm{v'_{\delta}(z)}^2_H \,\dz\\
        & \ge - \frac{1}{2} \norm{v_{\delta}(0)}^2_H +
         \int_{0}^{T}z\norm{v'_{\delta}(z)}^2_H \,\dz
\end{align*}
Rearranging this inequality and applying by~\eqref{eq:68}, one gets
\begin{displaymath}
    \int_{0}^{T}z\norm{v'_{\delta}(z)}^2_H \,\dz \le
        T(v'_{\delta}(T),v_{\delta}(T))_{H}+\frac{1}{2}
                 \norm{v_{\delta}(0)}^2_H
                    \le  \frac{1}{2}
                 \norm{v_{\delta}(0)}^2_H.
 \end{displaymath}
Sending $T\to +\infty$ in this estimate shows that~\eqref{eq:41bis} holds.

 By~\eqref{eq:30}, we can extract another subsequence of
 $(v_{\lambda})_{\lambda>0}$ such that
 $v'_{\lambda}(0)\rightharpoonup v'_{\delta}(0)$ in $H$ and so,
 sending $\lambda\to 0+$ in~\eqref{eq:30}, one finds
 that~\eqref{eq:30bis} holds. By~\eqref{eq:30bis} and since
 $v_{\delta}(0)\in (\partial_{\!
   H}\tilde{j})^{-1}(v'_{\delta}(0))+\varphi$, we can conclude that
 the sequence $(v_{\delta}(0))_{\delta>0}$ is bounded in $H$, showing
 that the right hand side inequality in~\eqref{eq:49}
holds. 
Since $-z^{-\frac{1-2s}{s}}
   v''_{\delta}=\partial_{L^{2}_{\underline{s}}}\E(v_{\delta})$
   and $A$ is monotone with $0\in A0$, 
   multiplying~\eqref{eq:6bis} by $v_{\delta}(z)$ gives
   \begin{displaymath}
     z^{-\frac{1-2s}{s}}(v''_{\delta}(z),v_{\delta}(z))_{H}=\delta\,
     \norm{v_{\delta}(z)}_{H}^{2}+(w_{\delta}(z),v_{\delta}(z))_{H}\ge 0
   \end{displaymath}
   for almost every $z>0$, where $w_{\delta}(z)\in Av_{\delta}(z)$ satisfies
   $w_{\delta}(z)+\delta v_{\delta}(z)=z^{-\frac{1-2s}{s}}
   v''_{\delta}(z)$. Therefore,
   \begin{equation}
     \label{eq:51}
     (v''_{\delta}(z),v_{\delta}(z))_{H}\ge 0\qquad\text{ for almost every
   $z>0$.}
   \end{equation}
   Since $v_{\delta}\in
   W^{2,2}_{\frac{1-s}{2s},\frac{1}{2},\frac{3s-1}{2s}}(H)$, 
the integration by parts rule in Lemma~\ref{lem:ibp} yields that 
\begin{align*}
\frac{\td}{\dz}\frac{1}{2}\norm{v_\delta(z)}_H^2
&= (v'_\delta(z),v_\delta(z))_H\\
&=-\int_{z}^{+\infty}\frac{\td}{\td r}(v'_\delta(r),v_\delta(r))_H \,\dr\\
& = -\int_{z}^{+\infty}(v''_\delta(r),v_\delta(r))_H \,\ds - \int_{z}^{+\infty}\norm{v'_\delta(r)}_H^2 \,\dr
\end{align*}
for every $z\ge 0$. Thus and by~\eqref{eq:51}, $z\mapsto
\norm{v_\delta(z)}_H^2$ is decreasing on $[0,+\infty)$ and, in
particular, the first two inequalities in~\eqref{eq:49} hold. 
\medskip

\underline{4. There is a function $v\in C^{1}([0,+\infty);H)$ such that
  $v'_{\delta}\to v'$ as $\delta\to0+$.} 
We need the following convergence results.

 \begin{lemma}\label{lem:4}
    There is a function
    $v\in C^{1}([0,+\infty);H)\cap L^{\infty}(H)$ with $v'\in
    W^{1,2}_{\frac{1}{2},\frac{3s-1}{2s}}(H)$ such that
    after passing to a subsequence, one has that
    \begin{align}
        \label{eq:5}
      &\lim_{\delta\to 0+}v_{\delta}(z)=v(z) \qquad\text{weakly $H$ for
     every $z\ge 0$,}\\
      \label{eq:61}
      &\lim_{\delta\to 0+}v_{\delta}=v\qquad\text{weakly in
       $L^{2}_{loc}(H)$,}\\
      \label{eq:54}
      &\lim_{\delta\to 0+}v'_{\delta}=v'\qquad\text{strongly in
       $L^{2}(H)$,}\\ \label{eq:4}
      & \lim_{\delta\to 0+}v'_{\delta}(0)=v'(0)\qquad\text{weakly in $H$,}\\
      \label{eq:64}
      & \lim_{\delta\to 0+}v'_{\delta}(z)=v'(z)\qquad\text{strongly in
       $H$ for every $z>0$, and}\\
      \label{eq:62}
      &\lim_{\delta\to 0+}v''_{\delta}=v''\qquad\text{weakly in
       $L^{2}_{\overline{s}^{\star}}(H)$.}
   \end{align}
 \end{lemma}

 Before proceeding with \emph{step 4}, we give the proof of this lemma.\medskip

 \begin{proof}[Proof of Lemma~\ref{lem:4}]
   For $\delta$, $\hat{\delta}\in (0,1]$, let $v_{\delta}$ and
   $v_{\hat\delta}$ be two strong solutions of~\eqref{eq:6bis} and set $
     w=v_{\delta}-v_{\hat\delta}$. Then, by~\eqref{eq:6bis},
   \begin{displaymath}
     z^{-\frac{1-2s}{s}} w''(z)\in
     Av_{\delta}(z)-Av_{\hat\delta} +\delta v_{\delta}(z)-\hat{\delta}v_{\hat\delta}(z)\qquad\text{for
       almost every $z>0$}
   \end{displaymath}
   and hence, the monotonicity of $A$ implies that
   \begin{align*}
     &z^{-\frac{1-2s}{s}}(w''(z),w(z))_H\\
      &\qquad\ge (\delta v_\delta(z) - \hat\delta v_{\hat\delta}(z), v_\delta(z) - v_{\hat\delta}(z))_H\\
      &\qquad=\delta
        \norm{v_\delta(z)}^2_H-\delta\,(v_\delta,v_{\hat\delta}(z))_{H}
        - \hat\delta\,(v_\delta(z),v_{\hat\delta}(z))_{H}
        +\hat{\delta}\,\norm{v_\delta}^2_H.
   \end{align*}
   By Young's inequality,
   \begin{displaymath}
     \delta\,(v_\delta,v_{\hat\delta})_{H} \le
     \delta\norm{v_\delta}^2_H+\frac{\delta}{4}\norm{v_{\hat\delta}}^2_H
     \quad\text{ and }\quad
     \hat{\delta}\,(v_\delta,v_{\hat\delta})_{H} \le \hat\delta \, 
     \norm{v_{\hat{\delta}}}^2_H+\frac{\hat\delta}{4}\norm{v_\delta}^2_H.
   \end{displaymath}
   Therefore and by~\eqref{eq:49},
   \begin{displaymath}
     z^{-\frac{1-2s}{s}}(w''(z),w(z))_H
     \ge - \frac{\delta}{4}\norm{v_{\hat\delta}(z)}^2_H
     -\frac{\hat\delta}{4}\norm{v_\delta(z)}^2_H\ge -\frac{\delta+\hat{\delta}}{4}\,C^{2}.
   \end{displaymath}
   Integrating this inequality over $[0,T]$, for $T>0$, gives
   \begin{align*}
     &(w'(T),w(T))_H -(w'(0),w(0))_H-\int_{0}^{T}\norm{w'(z)}^2_H \,\dz\\
     &\qquad =\int_{0}^{T}\frac{\td}{\dz}(w'(z),w(z))_H \,\dz-\int_{0}^{T}\norm{w'(z)}^2_H \,\dz\\
     &\qquad =\int_{0}^{T}\Big((w''(z),w(z))_H + \norm{w'(z)}^2_H \Big)\,\dz- \int_{0}^{T}\norm{w'(z)}^2_H \,\dz\\
     &\qquad \ge  -\frac{(\gamma+\delta)}{4}C^2 \int_{0}^{T}z^{\frac{1-2s}{s}}\,\dz.
\end{align*}
Since $v'_\delta(0) \in \partial_{H}\tilde{j}(v_\delta(0)-\varphi)$
and $v'_{\hat\delta}(0) \in \partial_{\!
  H}\tilde{j}(v_{\hat\delta}(0)-\varphi)$, the monotonicity of 
 $\partial_{H}\tilde{j}$ implies that $(w'(0),w(0))_H \ge 0$ and so,
we can conclude from the previous inequality that
\begin{displaymath}
   \int_{0}^{T}\norm{w'(z)}^2_H \,\dz\le (w'(T),w(T))_H  
   + \frac{(\gamma+\delta)}{4}C^2 \int_{0}^{T}z^{\frac{1-2s}{s}}\,\dz.
\end{displaymath}
Applying Cauchy-Schwarz's inequality and~\eqref{eq:49} to the right
hand side of this inequality, gives
\begin{equation}
  \label{eq:56}
  \begin{split}
    \int_{0}^{T}\norm{w'(z)}^2_H \,\dz
    &\le C\norm{w'(T)}_{H}+\frac{(\gamma+\delta)}{4}C^2
    \int_{0}^{T}z^{\frac{1-2s}{s}}\,\dz.
  \end{split}
\end{equation}

On the other hand, for every $\delta\in (0,1]$ and $T>0$,
\begin{displaymath}
T\int_{T}^{+\infty}\norm{v'_{\delta}(z)}^2_H \,\dz
\le \int_{T}^{+\infty}\norm{z\,v'_\delta(z)}^2_H \,\frac{\dz}{z}
\end{displaymath}
and by~\eqref{eq:57}, the function $z\mapsto \norm{v'_{\delta}(z)}_{H}^{2}$ is decreasing on
$(0,+\infty)$. Thus,
\begin{displaymath}
\int_{0}^{T}\norm{zv'_\delta(z)}^2_H \,\frac{\dz}{z}\ge 
\frac{T^{2}}{2}\norm{v'_\delta(T)}^2_H
\end{displaymath}
and by~\eqref{eq:41bis} and \eqref{eq:49},
\begin{align*}
    &\frac{T^{2}}{2}\norm{v'_\delta(T)}^2_H
    +T\int_{T}^{+\infty}\norm{v'_{\delta}(z)}^2_H \,\dz\\
    &\qquad\le \int_{0}^{T}\norm{zv'_\delta(z)}^2_H \,\dz+
    \int_{T}^{+\infty}\norm{z\,v'_\delta(z)}^2_H \,\frac{\dz}{z}
       = \norm{z\,v'_{\delta}(z)}_{L^{2}_{\star}(H)}\le \frac{C^2}{2}.
\end{align*}
Hence,
\begin{displaymath}
\int_{T}^{+\infty}\norm{v'_\delta(z)}^2_H \,\dz \le \frac{C^{2}}{2
  T}\qquad\text{for every $\delta\in (0,1]$,}
\end{displaymath}
and
\begin{displaymath}
\norm{v'_\delta(T)}_H \le \frac{C}{T}\qquad\text{for every $\delta\in (0,1]$.}
\end{displaymath}
Now, applying these estimates to~\eqref{eq:56}. Then, we obtain
\begin{align*}
    \int_{0}^{T}\norm{w'(z)}^2_H \,\dz
       &\le C\norm{w'(T)}_{H}+\frac{(\gamma+\delta)}{4}C^2
    \int_{0}^{T}z^{\frac{1-2s}{s}}\,\dz\\
  & \le 2\frac{C^2}{T}+ \frac{(\delta+\hat\delta)}{4}C^2
   \frac{s}{1-s} T^{\frac{1-s}{s}}
\end{align*}
and so,
\begin{align*}
\int_{0}^{+\infty}\norm{w'(z)}^2_H \,\dz
& \le \int_{0}^{T}\norm{w'(z)}^2_H \,\dz + 2
  \int_{0}^{T}\norm{v'_{\delta}(z)}^2_H \,\dz +2  \int_{0}^{T}\norm{v'_{\hat{\delta}}(z)}^2_H \,\dz\\ 
&\le  2\frac{C^2}{T}+ \frac{(\delta+\hat{\delta})}{4}C^2 \frac{s}{1-s} T^{\frac{1-s}{s}}
    + 4 \frac{C^{2}}{2 T}.
\end{align*}
Choosing $T := 1 / (\delta + \hat{\delta})^{s}$ and inserting $w =v_{\delta}-v_{\hat\delta}$, then we get
\begin{displaymath}
  \int_{0}^{+\infty}\norm{v'_{\delta}(z)-v'_{\hat\delta}(z)}^2_H
  \,\dz\le  
  \left(2C^{2}+ \frac{C^2}{4} \frac{s}{1-s} 
    + 4 \frac{C^{2}}{2}\right)\, (\delta + \hat{\delta})^{s}.
\end{displaymath}
Therefore, $(v'_\delta)_{\delta>0}$ is a Cauchy sequence in
$L^{2}(H)$, implying that there is a $\hat{v}\in L^{2}(H)$ such that
\begin{displaymath}
  \lim_{\delta\to 0+}v'_\delta=\hat{v}\qquad\text{in $L^{2}(H)$.}
\end{displaymath}
In addition, by~\eqref{eq:49}, $(v_{\delta})_{\delta>0}$ is bounded in
$L^{2}(0,T;H)$ for every $T>0$. Therefore, there is a function
$v\in L^{2}_{loc}(H)$ with weak derivative $v'=\hat{v}$ in $L^{2}(H)$
and after possibly passing to a subsequence of
$(v_{\delta})_{\delta>0}$, one has that~\eqref{eq:61}
and~\eqref{eq:54} hold. Moreover,
by~\eqref{eq:31bis}, the sequence $(v''_{\delta})_{\delta\in (0,1]}$ is
bounded in $L^{2}_{\overline{s}^{\star}}(H)$. Thus, $v$ has a
second weak derivative $v''\in L^{2}_{\overline{s}^{\star}}(H)$
and by possibly replacing $(v_{\delta})_{\delta\in (0,1]}$ again by a subsequence,
one obtains~\eqref{eq:62}. Further,
\begin{displaymath}
  \norm{v(z)-v(\hat{z})}_{H} \le
  \Bigg|\int_{\hat{z}}^{z}\norm{v'(r)}_{H}\,\dr\Bigg|\le \abs{z-\hat{z}}^{1/2}\,\norm{v'}_{L^{2}(H)}
\end{displaymath}
for every $z$, $\hat{z}\ge 0$, showing that $v : [0,+\infty)\to H$ is uniformly continuous. In
particular, since $v'\in
W^{1,2}_{\frac{1}{2},\frac{3s-1}{s}}(H)$,
Lemma~\ref{lem:ibp} and~\eqref{eq:49} imply that
$v\in C^{1}([0,+\infty);H)\cap L^{\infty}(H)$. By the limits~\eqref{eq:54} and~\eqref{eq:62}, the
trace theorem on $W^{1,2}_{\frac{1}{2},\frac{3s-1}{s}}(H)$
(Lemma~\ref{lem:ibp}) implies that~\eqref{eq:4} holds. For $z>0$,
Cauchy-Schwarz's inequality gives
\begin{align*}
  z^{-\frac{1-2s}{2s}}\,\frac{1}{2}\norm{v'_{\delta}(z)-v'(z)}_{H}^{2}
  & =
    \int_{z}^{+\infty}r(r)^{-\frac{1-2s}{2s}}\,(v''_{\delta}(r)-v''(r),v'_{\delta}(r)-v'(r))_{H}\,\dr\\
  &\hspace{1cm}   +\tfrac{2s-1}{2s}
    \int_{z}^{+\infty}r^{-\frac{1}{2s}}
    \frac{1}{2}\norm{v'_{\delta}(r)-v'(r)}_{H}^{2}\,\dr\\
  &\le
    \norm{z^{\frac{3s-1}{2s}}(v''_{\delta}-v'')}_{L^{2}_{\star}(H)}\,\norm{v'_{\delta}-v'}_{L^{2}(H)},\\
  &\hspace{1cm}   +\labs{\tfrac{2s-1}{2s}} \tfrac{1}{2} z^{-\frac{1}{2s}}\;\norm{v'_{\delta}-v'}_{L^{2}(H)}^{2}\\
  & =
    \norm{v''_{\delta}-v''}_{L^{2}_{\overline{s}^{\star}}(H)}\,\norm{v'_{\delta}-v'}_{L^{2}(H)}\\
  &\hspace{2cm}  + \labs{\tfrac{2s-1}{2s}} \tfrac{1}{2} z^{-\frac{1}{2s}}\;\norm{v'_{\delta}-v'}_{L^{2}(H)}^{2}.
\end{align*}
Thus, and by~\eqref{eq:54} and~\eqref{eq:31bis}, one has
that~\eqref{eq:64}. By~\eqref{eq:49}, there is a $v_{0}\in H$ and a
subsequence of $(v_{\delta})_{\delta\in (0,1]}$ such that
$v_{\delta}(0)\rightharpoonup v_{0}$ weakly in $H$ as $\delta\to
0+$. Now, let $x\in H$ and for $\rho\in C^{\infty}([0,+\infty))$
satisfying $0\le \rho\le 1$, $\rho\equiv 1$ on $[0,1]$ and
$\rho\equiv 0$ on $[2,+\infty)$, set $\xi(z)=\rho(z)\,x$ for every
$z\ge 0$. Since $v_{\delta}$ and $v$ belong to $C^{1}([0,+\infty);H)$
and by the limits~\eqref{eq:61} and~\eqref{eq:64}, one has that
 \begin{align*}
   (v_{\delta}(0),x)_{H} &= 
      -\int_{0}^{2}\frac{\td}{\dz}(v_{\delta}(z),\xi(z))_{H}\,\dz\\
    &= - \int_{0}^{2}(v'_{\delta}(z),\xi(z))_{H}\,\dz  -
      \int_{0}^{2}(v_{\delta}(z),\xi'(z))_{H}\,\dz\\
   &\to - \int_{0}^{+\infty}(v'(z),\xi(z))_{H}\,\dz  -
      \int_{0}^{2}(v(z),\xi'(z))_{H}\,\dz\\
   &= (v(0),x)_{H} 
 \end{align*}
as $\delta\to0+$. Since $x\in H$ was arbitrary, this shows that
$v_{0}=v(0)$ and
\begin{displaymath}
  \lim_{\delta\to0+}v_{\delta}(0)=v(0)\qquad\text{weakly in $H$.}
\end{displaymath}
Moreover, since for every $z>0$,
\begin{displaymath}
  v_{\delta}(z)=v_{\delta}(0)+\int_{0}^{z}v'_{\delta}(r)\,\dr\qquad\text{and}\qquad
  v(z)=v(0)+\int_{0}^{z}v'(r)\,\dr, 
\end{displaymath}
the previous limit together with~\eqref{eq:54} yields that
\eqref{eq:5} holds. 
\end{proof}

 \noindent\emph{Continuation of the proof of
   Theorem~\ref{thm:1bis}. }\medskip

\underline{5. The limit $v$ is a solution of~\eqref{eq:2}.} 
 For a compact interval $K:=[a,b]$ with
 $0<a<b<+\infty$, let
 $L^{2}_{K}(H)$ denote the set of all $L^{2}$-integrable functions $v
 : K\to H$ and $\A_{K}$ denote the operator
\begin{displaymath}
  \A_{K}:=\Big\{(v,w)\in L^{2}_{K}(H)\times
  L^{2}_{K}(H)\,\Big\vert\, w(z)\in Av(z)
  \text{ for a.e. }z\ge 0\Big\}.
\end{displaymath}
Then, since $A$ is maximal monotone on $H$ and $K$ has finite measure,
it follows that $\A_{K}$ is maximal monotone on $L^{2}_{K}(H)$. Moreover,
since the restrictions on $K$ of functions $v\in L^{2}_{\underline{s}^{\star}}(H)$, $L^{2}(H)$, and
$L^{2}_{\overline{s}^{\star}}(H)$ belong to $L^{2}_{K}(H)$, the
equation~\eqref{eq:6bis} can be rewritten as
\begin{displaymath}
 \A_{K} v_{\delta}+\delta v_{\delta}\ni z^{-\frac{1-2s}{s}}v''_{\delta}\qquad\text{in $L^{2}_{K}(H)$.}
\end{displaymath}
By~\eqref{eq:49} and~\eqref{eq:62}, $f_{\delta}:=z^{-\frac{1-2s}{s}}v''_{\delta} -\delta\,
  v_{\delta}$ satisfies $f_{\delta}\in \A_{K} v_{\delta}$ and
\begin{equation}
  \label{eq:8}
  \lim_{\delta\to
    0+}f_{\delta}=z^{-\frac{1-2s}{s}}v''\text{ weakly in $L^{2}_{K}(H)$.} 
\end{equation}
By~\eqref{eq:61}, if 
\begin{equation}
  \label{eq:63}
\lim_{\delta\to 0+} (v_\delta,f_\delta)_{L^2_{K}(H)} \le (v,z^{-\frac{1-2s}{s}}v'')_{L^{2}_{K}(H)},
\end{equation}
then by~\cite[Proposition~2.15]{MR0348562}, we have that
\begin{displaymath}
  v\in D(\A_{K})\quad\text{and}\quad
  z^{-\frac{1-2s}{s}}v''\in \A_{K}v.
\end{displaymath}
To see that~\eqref{eq:63} holds, we write
\begin{displaymath}
  (v_{\delta},f_{\delta})_{L^2_{K}(H)} - (v, z^{-\frac{1-2s}{s}}v'')_{L^2_{K}(H)}
  =(v_{\delta}-v,f_{\delta})_{L^2_{K}(H)}+(v,f_{\delta}-z^{-\frac{1-2s}{s}}v'')_{L^2_{K}(H)}
\end{displaymath}
and note that by~\eqref{eq:8}, one has that
\begin{displaymath}
  (v,f_{\delta}-z^{-\frac{1-2s}{s}}v'')_{L^2_{K}(H)}=(v,f_{\delta})
  _{L^2_{K}(H)}-(v,z^{-\frac{1-2s}{s}}v'')_{L^2_{K}(H)}\to 0
\end{displaymath}
as $\delta\to0+$. In addition, since $K=[a,b]$, and by the limits~\eqref{eq:5}, \eqref{eq:61},
\eqref{eq:54}, \eqref{eq:64} and since by~\eqref{eq:49}, $\delta
v_{\delta}\to 0$ strongly in $L^{2}_{loc}(H)$, we have that
\allowdisplaybreaks
\begin{align*}
(v_{\delta}-v,f_{\delta})_{L^2_{K}(H)}
  &=(v_{\delta}-v,z^{-\frac{1-2s}{s}}v''_{\delta}-\delta
    v_{\delta})_{L^{2}_{K}(H)}\\
  &= (v_{\delta}-v,z^{-\frac{1-2s}{s}}(v''_{\delta}-v''))_{L^{2}_{K}(H)}\\
    &\hspace{3cm} +(v_{\delta}-v, z^{-\frac{1-2s}{s}}v''-\delta v_{\delta})_{L^{2}_{K}(H)}\\
  & = \int_{a}^{b}
    (v_{\delta}(z)-v(z),v''_{\delta}(z)-v''(z))_{H}\,z^{-\frac{1-2s}{s}}\,\dz\\
  &\hspace{3cm} +(v_{\delta}-v,
    z^{-\frac{1-2s}{s}}v''-\delta
    v_{\delta})_{L^{2}_{K}(H)}\\
  & = \int_{a}^{b}
    \frac{\td}{\dz}(v_{\delta}(z)-v(z),v'_{\delta}(z)-v'(z))_{H}\;z^{-\frac{1-2s}{s}}\,\dz\\
  &\hspace{2cm} -\int_{a}^{b} \norm{v'_{\delta}(z)-v'(z)}_{H}^{2}\;z^{-\frac{1-2s}{s}}\,\dz\\
  &\hspace{3cm} +(v_{\delta}-v,
    z^{-\frac{1-2s}{s}}v''-\delta
    v_{\delta})_{L^{2}_{K}(H)}\\
&=(v_{\delta}(z)-v(z),v'_{\delta}(z)-v'(z))_{H}\;z^{-\frac{1-2s}{s}}\,\Big\vert_{a}^{b}\\
  &\hspace{1cm}+\frac{1-2s}{s} \int_{a}^{b}
    (v_{\delta}(z)-v(z),v'_{\delta}(z)-v'(z))_{H}\;z^{-\frac{1-s}{s}}\,\dz\\
  &\hspace{2cm} -\int_{a}^{b} \norm{v'_{\delta}(z)-v'(z)}_{H}^{2}\;z^{-\frac{1-2s}{s}}\,\dz\\
  &\hspace{3cm} +(v_{\delta}-v,
    z^{-\frac{1-2s}{s}}v''-\delta
    v_{\delta})_{L^{2}_{K}(H)}\\
 & \rightarrow 0\qquad\text{as $\delta\to 0+$,}
\end{align*}
This show that~\eqref{eq:63} holds and since the compact sub-interval
$K=[a,b]$ of $\R_{+}$ was arbitrary, we have thereby shown that $v$ is
a solution of~\eqref{eq:19} or equivalently, a strong solution
of~\eqref{eq:80}. It remains to show that
$v'(0)\in \partial_{H}\tilde{j}(v(0)-\varphi)$. To see this, note
first that if $j$ is the indicator function, then
condition~\eqref{eq:53} reduces to the condition $v(0)=\varphi$. Since
$v_{\delta}(0)=\varphi$ for all $\delta>0$, we have by~\eqref{eq:5}
that $v(0)=\varphi$. Now, suppose
$\partial_{H}\tilde{j} : D(\partial_{H}\tilde{j})\to H$ is a weakly
continuous mapping. Then, since $v'_{\delta}(0)=\partial_{H}\tilde{j}(v_{\delta}(0)-\varphi)$
for every $\delta\in (0,1]$, the weak continuity of
$\partial_{\!  H}\tilde{j}$ together with~ \eqref{eq:5}
and~\eqref{eq:4} imply that $v(0)-\varphi\in D(\partial_{H}\tilde{j})$
and $v'(0)=\partial_{\!  H}\tilde{j}(v(0)-\varphi)$.\medskip

\underline{6. The solution $v$ of~\eqref{eq:2} satisfies $v\in
  L^{\infty}(H)$ and~\eqref{eq:87}-\eqref{eq:31bis2}.} 
 Thanks to \eqref{eq:64}, we can
send $\delta\to 0+$ in~\eqref{eq:57} and obtain~\eqref{eq:66}. 
Due to \eqref{eq:64} and~\eqref{eq:5}, sending $\delta\to 0+$
in~\eqref{eq:68} yields that
\begin{displaymath}
  \frac{\td}{\dz}\frac{1}{2}\norm{v(z)}_{H}^{2}=(v'(z),v(z))_{H}\le
  0\qquad\text{for all $z\ge 0$.}  
\end{displaymath}
Hence, \eqref{eq:87} holds and, 
in particular, $v\in L^{\infty}(H)$.  By~\eqref{eq:64}, sending
$\delta\to~0+$ in~\eqref{eq:57} and using that
$v\in C^{1}([0,+\infty);H)$, one obtains~\eqref{eq:66}. To see
that~\eqref{eq:65} holds, we first note that by~\eqref{eq:64}
and~\eqref{eq:5}, sending $\delta\to0+$ in~\eqref{eq:68} and using
that $v\in C^{1}([0,+\infty);H)$ yields
\begin{equation}
  \label{eq:48}
  (v'(z),v(z))_{H}\le 0 \qquad\text{for every $z\ge 0$.}
\end{equation}
Moreover, since $v$ is a strong solution of~\eqref{eq:80}
and since $A$ is monotone,
\begin{displaymath}
  \int_{\varepsilon}^{T}\,z\, (v''(z),v(z))_{H}\,\dz \ge 0
\end{displaymath}
for every $T>\varepsilon>0$. Using this estimate, one sees that
\allowdisplaybreaks
\begin{align*}
       &T(v'(T),v(T))_{H} - \varepsilon(v'(\varepsilon),v(\varepsilon))_{H}\\
        &\qquad\ge
           \int_{\varepsilon}^{T}\frac{\td}{\dz}\frac{1}{2}\norm{v(z)}^2_{H}\,\dz
           + \int_{\varepsilon}^{T} z\,\norm{v'(z)}^2_H \,\dz\\
         &\qquad= \frac{1}{2} \norm{v(T)}^2_H - \frac{1}{2} \norm{v(\varepsilon)}^2_H +
         \int_{\varepsilon}^{T}z\norm{v'(z)}^2_H \,\dz\\
        & \qquad\ge - \frac{1}{2} \norm{v(\varepsilon)}^2_H +
         \int_{\varepsilon}^{T}z\norm{v'(z)}^2_H \,\dz
\end{align*}
Rearranging this inequality and applying~\eqref{eq:48} and
\eqref{eq:87}, one gets
\begin{align*}
    \int_{\varepsilon}^{T}z\norm{v'(z)}^2_H \,\dz &\le
        T(v'(T),v(T))_{H}- \varepsilon(v'(\varepsilon),v(\varepsilon))_{H}+\frac{1}{2}
                 \norm{v(\varepsilon)}^2_H\\
                    &\le  - \varepsilon\,(v'(\varepsilon),v(\varepsilon))_{H}+\frac{1}{2}
                 \norm{v(0)}^2_H.
 \end{align*}
Since $v\in C^{1}([0,+\infty);H)$, sending first $\varepsilon\to 0+$
and then $T\to +\infty$ in the resulting inequality shows
that~\eqref{eq:65} holds. 
By~\eqref{eq:66} and~\eqref{eq:65}, we see that
\begin{displaymath}
  \frac{z^{2}}{2}\norm{v'(z)}^{2}_{H}\le
  \int_{0}^{z}r\norm{v'(r)}_{H}^{2}\,\dr\le \frac{1}{2}
                 \norm{v(0)}^2_H
\end{displaymath}
for every $z>0$, which shows that~\eqref{eq:89} holds.
The estimates~\eqref{eq:30bis2}-\eqref{eq:31bis2} are obtained  
from~\eqref{eq:30bis}-\eqref{eq:31bis} by taking advantage of the
underlying weak limit. Finally,
we want to show that~\eqref{eq:77} holds. To do this, let $h>0$. 
Then, by the monotonicity of $A$, one has that
\begin{displaymath}
  \Big((z+h)^{\frac{1-2s}{s}}v''(z+h)-z^{\frac{1-2s}{s}}v''(z),v(z+h)-v(z)\Big)_{H}\ge 0
\end{displaymath}
for almost every $z>0$. Let $\xi\in C^{2}([0,+\infty))$ be such that
\begin{displaymath}
  \psi(z):=\xi(z) (z+h)^{\frac{1-2s}{s}}
\end{displaymath}
is increasing and satisfies $\psi(0)=0$. (For example, take
\begin{equation}
  \label{eq:3}
  \xi(z)= z^{\beta}(z+h)^{-\frac{1-2s}{s}}\qquad\text{for
    every $z\ge 0$,}
\end{equation}
for some $\beta>0$.) Then,
\begin{align*}
  &\int_{0}^{z}\psi(r)\,(v''(r+h)-v''(r),v(r+h)-v(r))_{H}\,\dr\\
  &\hspace{1cm}
    +\int_{0}^{z}\xi(r)\,\Big((r+h)^{\frac{1-2s}{s}}-r^{\frac{1-2s}{s}}\Big)(v''(r),v(r+h)-v(r))_{H}\,\dr
    \ge 0
\end{align*}
and hence
\begin{align*}
 & \psi(z) \,(v'(z+h)-v'(z),v(z+h)-v(z))_{H}\\
  &\hspace{1cm}-\int_{0}^{z}\psi'(r) \,(v'(r+h)-v'(r),v(r+h)-v(r))_{H}\,\dr\\
    &\hspace{3cm}
      -\int_{0}^{z}\psi(r)\norm{v'(r+h)-v'(r)}_{H}^{2}\,\dr\\
  &+\int_{0}^{z}\xi(r)\,\Big((r+h)^{\frac{1-2s}{s}}-r^{\frac{1-2s}{s}}\Big)(v''(r),v(r+h)-v(r))_{H}\,\dr\ge0.
\end{align*}
Therefore, and since $(v'(z+h)-v'(z),v(z+h)-v(z))_{H}\le 0$,
\begin{align*}
& \int_{0}^{z}\psi(r)\norm{v'(r+h)-v'(r)}_{H}^{2}\,\dr
   +\int_{0}^{z}\psi'(r) \,\frac{\td}{\dr}\frac{1}{2}\norm{v(r+h)-v(r)}^{2}_{H}\,\dr\\
  &\hspace{1cm} \le  \psi(z) \,(v'(z+h)-v'(z),v(z+h)-v(z))_{H}\\
  &\hspace{1cm} \qquad
    +\int_{0}^{z}\xi(r)\,\Big((r+h)^{\frac{1-2s}{s}}-r^{\frac{1-2s}{s}}\Big)(v''(r),v(r+h)-v(r))_{H}\,\dr\\
  &\hspace{1cm} \le \int_{0}^{z}\xi(r)\,
    \Big((r+h)^{\frac{1-2s}{s}}-r^{\frac{1-2s}{s}}\Big)(v''(r),v(r+h)-v(r))_{H}\,\dr.
\end{align*}
Now, integrating by parts, yields
\begin{align*}
  & \int_{0}^{z}\psi(r)\norm{v'(r+h)-v'(r)}_{H}^{2}\,\dr
    +\psi'(z) \,\frac{1}{2}\norm{v(z+h)-v(z)}^{2}_{H}\\
  &\hspace{1cm} \le
    \int_{0}^{z}\xi(r)\,\Big((r+h)^{\frac{1-2s}{s}}-r^{\frac{1-2s}{s}}\Big)(v''(r),v(r+h)-v(r))_{H}\,\dr\\
  &\hspace{1cm} \quad +\int_{0}^{z}\psi''(r)
    \,\frac{1}{2}\norm{v(r+h)-v(r)}^{2}_{H}\,\dr
    + \psi'(0) \,\frac{1}{2}\norm{v(0+h)-v(0)}^{2}_{H}.
\end{align*}
Dividing this inequality by $h^{2}$ and sending $h\to 0+$, yields
\begin{align*}
  & \int_{0}^{z}\psi(r)\norm{v''(r)}_{H}^{2}\,\dr
    +\psi'(z) \,\frac{1}{2}\norm{v'(z)}^{2}_{H}\\
  &\hspace{1cm} \le \frac{1-2s}{s}
    \int_{0}^{z}\xi(r)\,r^{\frac{1-3s}{s}}\,(v''(r),v'(r))_{H}\,\dr\\
  &\hspace{1cm} \quad +\int_{0}^{z}\psi''(r)
    \,\frac{1}{2}\norm{v'(r)}^{2}_{H}\,\dr
    + \psi'(0) \,\frac{1}{2}\norm{v'(0)}^{2}_{H}.
\end{align*}
Inserting $\xi$ from~\eqref{eq:3},
then
\begin{align*}
  & \int_{0}^{z} r^{\beta}\norm{v''(r)}_{H}^{2}\,\dr
    +\beta\,z^{\beta-1} \,\frac{1}{2}\norm{v'(z)}^{2}_{H}\\
  &\hspace{1cm} \le \frac{1-2s}{s}
    \int_{0}^{z}r^{\beta-1}
    \,(v''(r),v'(r))_{H}\,\dr\\
  &\hspace{1cm} \quad +\beta (\beta-1)\int_{0}^{z} r^{\beta-2}
    \,\frac{1}{2}\norm{v'(r)}^{2}_{H}\,\dr
    + \beta\,z^{\beta-1}_{\vert z=0} \,\frac{1}{2}\norm{v'(\delta)}^{2}_{H}.
\end{align*}
Choosing $\beta=3$ in this estimate, then one finds,
\begin{align*}
  & \int_{0}^{z} r^{3}\norm{v''(r)}_{H}^{2}\,\dr
    +3\,z^{2} \,\frac{1}{2}\norm{v'(z)}^{2}_{H}\\
  &\hspace{1cm} \le \frac{1-2s}{s}
    \int_{0}^{z}r^{2}
    \,(v''(r),v'(r))_{H}\,\dr
    + 3\,\int_{0}^{z} r\,
    \norm{v'(r)}^{2}_{H}\,\dr.
\end{align*}
Thus, if $s\ge 1/2$, then by applying~\eqref{eq:66},
one sees that
\begin{displaymath}
  \int_{0}^{z} r^{3}\norm{v''(r)}_{H}^{2}\,\dr\le  3\,\frac{\norm{\varphi}_{H}^{2}}{2}
\end{displaymath}
and hence the first part of~\eqref{eq:77} holds by sending
$z\to +\infty$. If $0<s<1/2$, then $\frac{1-2s}{s}>0$
and so, by Young's inequality, we have for every $\varepsilon>0$ that
\begin{align*}
  & \int_{0}^{z} r^{3}\norm{v''(r)}_{H}^{2}\,\dr
    +3\,z^{2} \,\frac{1}{2}\norm{v'(z)}^{2}_{H}\\
  &\hspace{0.5cm} \le \varepsilon
    \int_{0}^{z}\norm{r^{\frac{3}{2}}v''(r)}_{H}^{2}\,\dr+\frac{s}{1-2s}\frac{1}{4\varepsilon}
    \int_{0}^{z}\norm{r^{\frac{1}{2}}v'(r)}^{2}_{H}\,\dr
    + 3\,\int_{0}^{z} r\,\norm{v'(r)}^{2}_{H}\,\dr.
\end{align*}
Choosing $\varepsilon=1/2$, one finds
\begin{displaymath}
  \frac{1}{2}\int_{0}^{z} r^{3}\norm{v''(r)}_{H}^{2}\,\dr\le \left(\frac{s}{1-2s}\frac{1}{2}+3\right)
    \int_{0}^{z}\norm{r^{\frac{1}{2}}v'(r)}^{2}_{H}\,\dr
\end{displaymath}
and hence by applying~\eqref{eq:66} and sending $z\to +\infty$, one
sees that the second part of~\eqref{eq:77} holds.\medskip

To see that for boundary data $\varphi\in
\overline{D(A)}^{\mbox{}_{H}}$, there is a solution $v$ of Dirichlet
problem~\eqref{eq:2bis} satisfying~\eqref{eq:87}-\eqref{eq:77},
and~\eqref{eq:88}, one proceeds as in the proof of Theorem~\ref{thm:1}
and Theorem~\ref{thm:2Robin}. 
This completes the proof of this theorem.
\end{proof}

Next, we outline the proof that the DtN operator
$\Theta_{1-2s}$ is monotone and establish the
characterization of the closure $\overline{\Theta}_{s}$.

\begin{proof}[Proof of Corollary~\ref{cor:1}]
For given $\varphi_{1}$ and $\varphi_{2}\in D(A)$, let $v_{1}$,
$v_{2}\in L^{\infty}(H)$ be two strong solutions of~\eqref{eq:2}
respectively to boundary data $\varphi_{1}$ and $\varphi_{2}$. Then
by the monotonicity of $A$,
\begin{equation}
  \label{eq:74}
  \begin{split}
    \frac{\td^2}{\dz^2}\frac{1}{2}\norm{v_{1}(z)-v_{2}(z)}_{H}^{2}
    &=(v''_{1}(z)-v''_{2}(z),v_{1}(z)-v_{2}(z))_{H}\\
    &\hspace{2.5cm} +\norm{v'_{1}(z)-v'_{2}(z)}_{H}^{2}\ge 0.
  \end{split}
\end{equation}
Thus, the map $z\mapsto \frac{1}{2}\norm{w(z)}_{H}^{2}$ is convex on
$[0,+\infty)$. Since $v_{1}-v_{2}\in L^{\infty}(H)$ and since every
bounded convex function is necessarily decreasing, we have that
$z\mapsto \frac{1}{2}\norm{w(z)}_{H}^{2}$ is decreasing on
$[0,+\infty)$. Thus,
\begin{displaymath}
  \frac{\td}{\dz}\frac{1}{2}\norm{v_{1}(z)-v_{2}(z)}_{H}^{2}=(v'_{1}(z)-v'_{2}(z),v_{1}(z)-v_{2}(z))_{H}\le
  0
\end{displaymath}
for every $z\ge \hat{z}\ge 0$, which shows that
$\Theta_{s}\varphi:=-(2s)^{1-2s} v'(0)$ is a monotone
operator. It follows from the definition of $\Theta_{s}$ and
by Theorem~\ref{thm:2} that $D(A)$ is a subset of
$D(\Theta_{s})$. On the other hand, for every
$\varphi\in D(\Theta_{s})$, there is a solution
$v\in C^{1}([0,+\infty);H)$ of Dirichlet problem~\eqref{eq:2bis} with
$v(0)=\varphi$. Since $v(t)\in D(A)$ for a.e. $t>0$, it follows that
$\varphi\in \overline{D(A)}^{\mbox{}_{H}}$, showing that
$D(\Theta_{s})\subseteq \overline{D(A)}^{\mbox{}_{H}}$. Further, by the regularity
$v\in C^{1}([0,+\infty);H)$ and the uniqueness of the solutions to Dirichlet
problem~\eqref{eq:2bis}, the operator  $\Theta_{s}$ is a
well-defined mapping from $D(\Theta_{s})$ to $H$.

Next, we show that the closure $\overline{\Theta}_{s}$ of
$\Theta_{s}$ in $H\times H_{w}$ coincides with the set
\begin{displaymath}
  B:=\Bigg\{(\varphi,w)\in H\times H\,\Bigg\vert
    \begin{array}[c]{l}
      \exists \text{ $(\varphi_{n},w_{n})\in \Theta_{s}$ s.t. }
      \displaystyle\lim_{n\to +\infty}(\varphi_{n},w_{n})=(\varphi,w)\\
      \text{ in $H\times H_{w}$ \&} \text{ a strong solution $v$ of~\eqref{eq:80}}\\
      \text{satisfying }~\eqref{eq:87}\;\&\; v(0)=\varphi\text{ in $H$.}
    \end{array}
    \Bigg\}
\end{displaymath}
Obviously, the set $B$ is cointained in
$\overline{\Theta}_{s}$. So, let
$(\varphi,w)\in \overline{\Theta}_{s}$. Then, there is a
sequence $((\varphi_{n},w_{n}))_{n\ge 1}\subseteq \Theta_{s}$
such that $(\varphi_{n},w_{n})\to (\varphi,w)$ in $H\times H_{w}$ as
$n\to +\infty$. By definition of $\Theta_{s}$, there are
solutions $v_{n}\in C^{1}([0,+\infty);H)$ of Dirichlet
problem~\eqref{eq:2bis} with $v_{n}(0)=\varphi_{n}$ and satisfying 
$\Theta_{s}v_{n}(0)=w_{n}$. Since every
$\varphi_{n}\in \overline{D(A)}^{\mbox{}_{H}}$, Theorem~\ref{thm:2}
yields that each $v_{n}$ satisfies~\eqref{eq:87}
and by~\eqref{eq:88}, $(v_{n})_{n\ge 0}$ is a Cauchy sequence in
$C^{b}([0,+\infty);H)$. Hence, there is a $v\in C^{b}([0,+\infty);H)$
such that $v_{n}\to v$ in $C^{b}([0,+\infty);H)$ as $n\to +\infty$ and
so, $v(0)=\varphi$. Since each
$\varphi_{n}\in \overline{D(A)}^{\mbox{}_{H}}$, we also have that
$\varphi \in \overline{D(A)}^{\mbox{}_{H}}$. Hence,
Theorem~\ref{thm:2} yields the existence of a unique strong solution $v$
of~\eqref{eq:2bis}, proving that $(\varphi,w)\in B$. 

Now, for given $\lambda>0$, the functional $j : H\to [0,+\infty)$
given by $j(v)=\frac{1}{2 \lambda}\norm{v}_{H}^{2}$, ($v\in H$), is
strictly coercive, continuously differentiable on $H$ and its
Fr\'echet derivative $j'$ coincides with the subdifferential
$\partial_{H}j$ given by
$\partial_{H}j(v)=\frac{\td}{\td v}j(v)=\frac{1}{\lambda} v$,
($v\in H$). Note, $\partial_{H}j$ is a bounded linear operator on $H$
and hence, in particular, weakly continuous, and $A$ is
$\partial_{H}j$ monotone. Thus, by Theorem~\ref{thm:2}, for every
$\varphi\in D(A)$, there is a unique solutions $v\in C^{1}([0,+\infty);H)$ of
\begin{equation}
  \label{eq:33}
  \begin{cases}
    \hspace{54pt}z^{-\frac{1-2s}{s}}v''(z)\in A(v(z)) & \text{for
      a.e. $z>0$,}\\
    -(2s)^{1-2s} v'(0)+\lambda v(0)=\varphi. &
  \end{cases}
\end{equation}
By definition of $\Theta_{s}$, we have thereby shown that
for every $\varphi\in D(A)$ and $\lambda>0$, there is a $v(0)\in D(\Theta_{s})$
satisfying
\begin{equation}
  \label{eq:95}
  v(0)+\lambda \Theta_{s}v(0)=\varphi.
\end{equation}
Thus, $D(A)\subseteq \textrm{Rg}(I_{H}+\lambda \Theta_{s})$
for every $\lambda>0$. To see that also
$\overline{D(A)}^{\mbox{}_{H}}\subseteq \textrm{Rg}(I_{H}+\lambda
\overline{\Theta}_{s})$, take
$\varphi\in \overline{D(A)}^{\mbox{}_{H}}$. Then, there is a sequence
$(\varphi_{n})_{n\ge 1}\subseteq D(A)$ such that $\varphi_{n}\to \varphi$ in
$H$ as $n\to +\infty$. Moreover, since $D(A)\subseteq
\textrm{Rg}(I_{H}+\lambda \Theta_{s})$, for each
$\varphi_{n}$, there is a unique solution $v_{n}\in C^{1}([0,+\infty);H)$ of Dirichlet
problem~\eqref{eq:2bis} satisfying $\Theta_{s}v_{n}(0)+\lambda
v_{n}(0)=\varphi_{n}$. Thus by the
monotonicity of $\Theta_{s}$, one has that
\begin{align*}
  \norm{v_{n}(0)-v_{m}(0)}_{H}^{2} & =
                                     \frac{1}{\lambda}(v_{n}(0)-v_{m}(0),
                                     \varphi_{n}-\varphi_{m})_{H}\\
  &\qquad -\frac{1}{\lambda} (v_{n}(0)-v_{m}(0),
                                  \Theta_{s}v_{n}(0)-\Theta_{s}v_{m}(0) )_{H}\\
  &\le \frac{1}{\lambda}\norm{v_{n}(0)-v_{m}(0)}_{H}\,
                                     \norm{\varphi_{n}-\varphi_{m}}_{H}
\end{align*}
and so,
\begin{equation}
  \label{eq:34}
  \norm{v_{n}(0)-v_{m}(0)}_{H}\le  \frac{1}{\lambda}\,
                                     \norm{\varphi_{n}-\varphi_{m}}_{H}.
\end{equation}
Since each $v_{n}(0)\in \overline{D(A)}^{\mbox{}_{H}}$, $v_{n}$
satisfies~\eqref{eq:87} and~\eqref{eq:88}holds. Combining this
with~\eqref{eq:34}, one sees that $(v_{n})_{n\ge 0}$ is a Cauchy
sequence in $C^{b}([0,+\infty);H)$ and so, there is a
$v\in C^{b}([0,+\infty);H)$ such that $v_{n}\to v$ in
$C^{b}([0,+\infty);H)$ as $n\to +\infty$. In particular,
$v_{n}(0)\to v(0)$ in $H$ as $n\to +\infty$. Thus, $v(0)\in \overline{D(A)}^{\mbox{}_{H}}$ and
so, Theorem~\ref{thm:2} yields that $v\in C^{1}([0,+\infty);H)$ and
$v$ is a solution of Dirichlet problem~\eqref{eq:2bis}. Moreover, we
have that
\begin{displaymath}
 w:=\lim_{n\to +\infty} \Theta_{s}v_{n}(0)=\varphi-\lambda v(0)
  \qquad\text{exists in $H$,}
\end{displaymath}
showing that $(v(0),w)\in \overline{\Theta}_{s}$ and
$\overline{\Theta}_{s}v(0)+\lambda v(0)=\varphi$. This proves
that $\overline{D(A)}^{\mbox{}_{H}}\subseteq \textrm{Rg}(I_{H}+\lambda
\overline{\Theta}_{s})$. 
Thus, if $\overline{D(A)}^{\mbox{}_{H}}=H$, then
$\overline{\Theta}_{s}$ is maximal monotone on $H$. 
%
%
This completes the proof of this
corollary. 
\end{proof}

%
%
%
%
%
%

\section{Interpolation properties}
\label{sec:interpol}

Our first theorem of this section, allows us to establish
interpolation properties and comparison principles of the semigroup
$\{\tilde{T}_{s}(t)\}_{t\ge 0}$ generated by $-\overline{\Theta}_{s}$.

\begin{theorem}
  \label{thm:4}
  Let $A$ be a maximal monotone operator on $H$
  with $0\in \textrm{Rg}(A)$. Suppose,
  $\phi : H\to \R\cup\{+\infty\}$ is convex, proper and lower
  semicontinuous such that $A$ is $\partial_{H}\phi$-monotone. Then
  for every for $0<s<1$, $\Theta_{s}$ is $\partial_{H}\phi$-monotone.
\end{theorem}

In the case $H=L^{2}(\Sigma,\mu)$ the Lebesgue space of $2$-integrable
functions defined on a
$\sigma$-finite measure space $(\Sigma,\mu)$, then
Theorem~\ref{thm:4} provides the following interpolation properties of
the semigroup $\{\tilde{T}_{r}\}_{r\ge 0}$ generated by $-\overline{\Theta}_{s}$.

\begin{corollary}
  \label{cor:5}
  Let $(\Sigma,\mu)$ be a $\sigma$-finite measure space and
  $A$ a maximal monotone operator on
  $L^{2}(\Sigma,\mu)$ with $0\in \textrm{Rg}(A)$. If $A$ is completely
  accretive, then for every  $0<s<1$, the DtN operator $\Theta_{s}$ is completely
  accretive on $L^{2}(\Sigma,\mu)$. In particular, if the closure
  $\overline{\Theta}_{s}$ of $\Theta_{s}$ is maximal
  monotone, then the semigroup $\{\tilde{T}_{s}(t)\}_{t\ge 0}$
  generated by $-\overline{\Theta}_{s}$ is $L^{\psi}$-contractive on
  $L^{2}(\Sigma,\mu)$ for any $N$-function $psi$, 
  $\{\tilde{T}_{s}(t)\}_{t\ge 0}$ is $L^{1}$- and $L^{\infty}$-contractive on
  $L^{2}(\Sigma,\mu)$, and order preserving.
\end{corollary}

The proof of  Corollary~\ref{cor:5} follows immediately from
Theorem~\ref{thm:4} since if $A$ is completely accretive on
$L^{2}(\Sigma,\mu)$, then for every $\lambda>0$, the resolvent
$J_{\lambda}^{A}$ of $A$ is a complete contraction, which by Proposition~\ref{prop:charact-cc}
means that for every convex, lower semicontinous function $j : \R\to \overline{\R}_{+}$  
satisfying $j(0)=0$, $A$ is $\partial_{H}\phi_{j}$-monotone for $\phi_{j} :
L^{2}(\Sigma,\mu)\to \overline{\R}_{+}$ given by
\begin{displaymath}
  \phi_{j}(u)=
  \begin{cases}
    \int_{\Sigma}j(u)\,\dmu & \text{if
      $j(u)\in L^{1}(\Sigma;\mu)$,}\\
    +\infty &\text{if otherwise,}
  \end{cases}
\end{displaymath}
for every $u\in L^{2}(\Sigma,\mu)$. In particular, for every
$N$-function $\psi$ and $\alpha>0$, $A$ is
$\partial_{H}\phi_{j}$-monotone for $j=\psi(\frac{\cdot}{\alpha})$,
and $A$ is $\partial_{H}\phi_{j}$-monotone for
$j=\psi(\frac{\max\{\cdot,0\}}{\alpha})$. Thus, by
Theorem~\ref{thm:4}, the semigroup $\{\tilde{T}_{s}(t)\}_{t\ge 0}$
generated by $-\overline{\Theta}_{s}$ is $L^{\psi}$-contractive on
  $L^{2}(\Sigma,\mu)$ for any $N$-function $psi$ and
  order preserving. Choosing $\psi=s^{q}$ for $q\in
  (1,+\infty)$. Letting $q\to 1$ or $q\to +\infty$ yields that
  $\{\tilde{T}_{s}(t)\}_{t\ge 0}$ is $L^{1}$- and $L^{\infty}$-contractive on
  $L^{2}(\Sigma,\mu)$.

\begin{remark}[Interpolation of $\{\tilde{T}_{s}(t)\}_{t\ge 0}$ on $L^{\psi}(\Sigma,\mu)$]
  By Corollary~\ref{cor:5}, if there is a $u_{0}\in L^{1}\cap
  L^{\infty}(\Sigma,\mu)$ such that the orbit
  $\tilde{T}_{s}(\cdot)u_{0}:=\{\tilde{T}_{s}(t)u_{0}\,\vert\,t\ge 0\}$ is locally
  bounded on $[0,+\infty)$ with values in $L^{1}\cap
  L^{\infty}(\Sigma,\mu)$, then for every $N$-function $\psi$, every
  $\tilde{T}_{s}(t)$ of the semigroup $\{\tilde{T}_{s}(t)\}_{t\ge 0}$ generated by
  $-\overline{\Theta}_{s}$ has unique contractive extension on $L^{\psi}(\Sigma,\mu)$,
  $L^{1}(\Sigma,\mu)$ and on
  $\overline{L^{2}\cap
    L^{\infty}(\Sigma,\mu)}^{\mbox{}_{L^{\infty}}}$, which we denote
  again by $\tilde{T}_{s}(t)$.
\end{remark}

We continue by giving the proof of the $\partial_{H}\phi$-monotonicity
of the DtN operator $\Theta_{s}$.

\begin{proof}[Proof of Theorem~\ref{thm:4}]
  For $\mu>0$, let $\phi_{\mu} : H\to \R$ be the regularization of
  $\phi$ defined by
  \begin{displaymath}
    \phi_{\mu}(v)=\min_{w\in H}\frac{1}{\mu
      2}\norm{w-v}_{H}^{2}+\phi(w)\qquad\text{for all $v\in H$.}
  \end{displaymath}
  By~\cite[Proposition~2.11]{MR0348562}, $\phi_{\mu}\in C^{1}(H;\R)$
  and the Fr\'echet derivative $\phi'_{\mu} : H\to H$ is Lipschitz
  continuous. Thus, for every $w\in W^{1,1}_{loc}((0,+\infty);H)$, $\phi'_{\mu}(w(z))$ is weakly
  differentiable at almost every $z>0$ and by the monotonicity of
  $\phi'_{\mu}$, one has that
  \begin{equation}
    \label{eq:75}
    \left(\tfrac{\td}{\dz}\phi'_{\mu}(w(z)),w'(z)\right)_{H}\ge
    0\qquad\text{for a.e. $z>0$.}
  \end{equation}
  Now, let $\varphi$, $\hat\varphi\in D(\Theta_{s})$ and
  $v$ and $\hat{v}$ two solutions of Dirichlet
  problem~\eqref{eq:2} with initial value $v(0)=\varphi$ and
  $\hat{v}(0)=\hat{\varphi}\in D(A)$. Since $v$ and $\hat{v}$ satisfy
  \begin{displaymath}
    z^{-\frac{1-2s}{s}}v''(z)\in A v(z)
    \qquad\text{and}\qquad
    z^{-\frac{1-2s}{s}}\hat{v}''(z)\in A \hat{v}(z)
  \end{displaymath}
  for almost every $z>0$ and since $A$ is
  $\partial_{H}\phi$-monotone (cf~\cite[Proposition~4.7]{MR0348562}),
  one has that
  \begin{displaymath}
    z^{-\frac{1-2s}{s}}(v''(z)-\hat{v}''(z),\phi'_{\mu}(v(z)-\hat{v}(z)))_{H}\ge 0
  \end{displaymath}
  for almost every $z>0$. Thus, if we set $w=v-\hat{v}$, then
  \begin{equation}
    \label{eq:76}
    (w''(z),\phi'_{\mu}(w(z)))_{H}\ge 0\qquad\text{ for
      almost every $z>0$.}
  \end{equation}
  Moreover, since $v(0)$ and $\hat{v}(0)\in D(A)$, one has
  $w'\in W^{1,2}_{\frac{1}{2},\frac{3s-1}{2s}}(H)$ and 
  $w \in
  W^{1,2}_{\frac{1-s}{2s},\frac{1}{2}}(H)$, and by the
  Lipschitz continuity of $\phi'_{\mu}$, 
  $\phi_{\mu}'\circ w\in
  W^{1,2}_{\frac{1-s}{2s},\frac{1}{2}}(H)$. Therefore, we
  can apply the integration by parts rule~\eqref{eq:32} of
  Lemma~\ref{lem:ibp}. This together with~\eqref{eq:75}
  and~\eqref{eq:76}, shows that
  \begin{align*}
    -(w'(0),\phi'_{\mu}(w(0)))_{H} &
      = \int_{0}^{+\infty}(w''(z),\phi'_{\mu}(w(z)))_{H}\,\dz\\
    &\qquad +
      \int_{0}^{+\infty}\left(w'(z),\frac{\td}{\dz}\phi'_{\mu}(w(z))\right)_{H}\,\dz\ge
    0.
  \end{align*}
  Since $w=v-\hat{v}$ and $\varphi$, $\hat\varphi\in
  D(\Theta_{s})$ were arbitrary, this implies by
  \cite[Proposition~4.7]{MR0348562}) that $\Theta_{s}$
  is $\partial_{H}\phi$-monotone.
\end{proof}


\section{Applications}
\label{sec:app}

In this section, we outline an application of Theorem~\ref{thm:1} -
Theorem~\ref{thm:2Robin}.\medskip

Let $\Sigma$ be an open subset of $\R^{d}$, $(d\ge 1)$, and $\mu=\mathcal{L}^{d}$ be
the $d$-dimensional Lebesgue measure. For
$1<p<+\infty$, suppose that $a : \Sigma \times \R^{d}\to \R^{d}$ is a
Carath\'eodory function satisfying the following
\emph{$p$-coercivity}, \emph{growth} and \emph{monotonicity}
conditions
 \begin{align}
   \label{eq:coerciveness}
   &a(x,\xi)\xi\ge \eta \abs{\xi}^{p}\\
   \label{eq:growth-cond}
   &\abs{a(x,\xi)}\le c_{1}\abs{\xi}^{p-1}+h(x)\\
   \label{eq:monotonicity-of-a}
   &(a(x,\xi_{1})-a(x,\xi_{2}))(\xi_{1}-\xi_{2})>0
 \end{align}
 for a.e. $x\in \Sigma$ and all $\xi$, $\xi_{1}$, $\xi_{2}\in \R^{d}$
 with $\xi_{1}\neq \xi_{2}$, where $h\in
 L^{p^{\mbox{}_{\prime}}}(\Sigma)$ and $c_{1}$, $\eta>0$ are constants
 independent of $x\in \Sigma$ and $\xi\in \R^{d}$. Under these assumptions, the
 second order quasi linear operator
 \begin{displaymath}
  \mathcal{B}u:= -\divergence (a(x,\nabla u))\qquad\text{in $\mathcal{D}'(\Sigma)$}
 \end{displaymath}
 for $u\in W^{1,p}_{loc}(\Omega)$ belongs to the class of
 \emph{Leray-Lions operators} (cf.~\cite{MR0194733}), of which the
 \emph{$p$-Laplace operator}
 $\Delta_{p}u=\textrm{div}(\abs{\nabla u}^{p-2}\nabla u)$ is a
 classical prototype. Here, we write $L^{q}(\Sigma)$ for
 $L^{q}(\Sigma,\mu)$ and if $L^{q}_{0}(\Sigma)$ is the closed subspace
 of $L^{q}(\Sigma)$ of all $u\in L^{q}(\Sigma)$ satisfying
 $\int_{\Sigma}u\,\dx=0$.

Then, the operator $\mathcal{B}$ either equipped
with homogeneous \emph{Dirichlet boundary
conditions}
\begin{equation}
  \label{eq:91}
    u=0\quad\text{on $\partial\Sigma\times (0,\infty)$}\qquad
    \text{if $\Sigma\subseteq \R^{d}$,}
\end{equation}
homogeneous \emph{Neumann boundary
conditions}
\begin{equation}
  \label{eq:92}
  a(x,\nabla u)\cdot\nu =0\quad\text{on
    $\partial\Sigma\times (0,\infty)$}\qquad
  \text{if $\abs{\Sigma}<\infty$,}
\end{equation}
or, homogeneous \emph{Robin boundary
conditions}
\begin{equation}
  \label{eq:93}
  a(x,\nabla u)\cdot\nu+b(x)
  \abs{u}^{p-1}u= 0 \quad\text{on
    $\partial\Sigma\times (0,\infty)$}\quad
  \text{if $\abs{\Sigma}<\infty$}
\end{equation}
is an $m$-completely accretive, single-valued, operator $A$ on
$L^{2}(\Sigma)$ (respectively, on $L^{2}_{0}(\Sigma)$ if $\mathcal{B}$
is equipped with~\eqref{eq:92}) with dense domain $D(A)$ in
$L^{2}(\Sigma)$ (respectively, in $L^{2}_{0}(\Sigma)$). In particular,
$A$ is a maximal monotone on $L^{2}(\Sigma)$ (respectively, in
$L^{2}_{0}(\Sigma)$) satisfying $A0=0$.

Thus by Theorem~\ref{thm:1}, the following Dirichlet problem is
well-posed. For every $\varphi\in L^{2}(\Sigma)$ (respectively, in $L^{2}_{0}(\Sigma)$) and $0<s<1$, there is a unique
solution $u : \Sigma\times [0,+\infty)\to \R$ of 
\begin{displaymath}
  \begin{cases}
    & -\frac{1-2s}{r} u_{r}-u_{rr}-\divergence_{x}
    (a(x,\nabla_{x}
    u))=0\qquad\hspace{1,6cm} \text{for $(x,r)\in \Sigma\times \R_{+}$,}\\
    & \text{$u$ satisfies~\eqref{eq:91}, (\eqref{eq:92} for $\varphi\in L^{2}_{0}(\Sigma)$), or
      ~\eqref{eq:93}}
    \qquad \text{for $(x,r)\in \partial\Sigma\times \R_{+}$,}\\
    & \hspace{4cm}u(\cdot,0)=\varphi(\cdot)\qquad\hspace{1,7cm} \text{on $\Sigma$.}
  \end{cases}
\end{displaymath}
Further by Theorem~\ref{thm:DtN}, for every $0<s<1$, there is a
strongly continuous semigroup $\{T_{s}(t)\}_{t\ge 0}$ generated
by $-A^{s}$ on $L^{2}(\Sigma)$ (respectively, on $L^{2}_{0}(\Sigma)$). The semigroup $\{T_{s}\}_{t\ge 0}$ is
order preserving, and each $T_{s}(t)$ has a unique contractive
extension on $L^{\psi}(\Sigma)$ (respectively, on
$L^{\psi}_{0}(\Sigma)$) for any $N$-function, on $L^{1}(\Sigma)$ and on
$\overline{L^{2}\cap L^{\infty}(\Sigma)}^{\mbox{}_{L^{\infty}}}$. In
particular, for every $\varphi\in L^{2}(\Sigma)$ (respectively,
$\varphi\in L^{2}_{0}(\Sigma)$) and $0<s<1$,
there is a unique solution
$U : \Sigma\times [0,+\infty)\times [0,+\infty)\to \R$ of
boundary-value problem
\begin{displaymath}
\begin{cases}
      & -\frac{1-2s}{r} U_{r}-U_{rr}-\divergence_{x} (a(x,\nabla_{x}
    U))=0\qquad \text{for $(x,r,t)\in \Sigma\times \R_{+}\times
      \R_{+}$,}\\
    & \hspace{5pt}\text{$U$ either satisfies~\eqref{eq:91},~\eqref{eq:92}, or
      ~\eqref{eq:93}}\qquad \text{for $(x,r,t)\in \partial\Sigma\times \R_{+}\times
      \R_{+}$,}\\
     &\mbox{}\hspace{3.7cm} U(0,t)=T_{s}(t)\varphi\qquad \text{on
       $\Sigma$ for every}t\ge 0,\\
    & \hspace{1.9cm} \displaystyle\lim_{t \to 0+}
    t^{1-2s}U_{r}(r,t)\in \tfrac{\td}{\dt} T_{s}(t)
    \qquad \text{on $\Sigma$, for every $t>0$,}\\
    & \hspace{3,75cm}U(\cdot,0,0)=\varphi(\cdot)\qquad\text{on $\Sigma$.}
    \end{cases}
  \end{displaymath}
  Finally, due to Theorem~\ref{thm:2Robin}, for every
  $\varphi\in D(A)$, $\lambda>0$, and $0<s<1$, there is
  a unique solution $u : \Sigma\times [0,+\infty)\to \R$ of the Robin problem
\begin{displaymath}
  \begin{cases}
    & -\frac{1-2s}{r} u_{r}-u_{rr}-\divergence_{x}
    (a(x,\nabla_{x}
    u))=0\qquad \text{for $(x,r)\in \Sigma\times \R_{+}$,}\\
    & \,\text{$u$ either satisfies~\eqref{eq:91},~\eqref{eq:92}, or
      ~\eqref{eq:93}}
    \qquad \text{for $(x,r)\in \partial\Sigma\times \R_{+}$,}\\
    & \hspace{0,85cm}\displaystyle
    \lim_{r \to 0+} r^{1-2s}u_{r}(r) +\lambda
    u(\cdot,0)=\varphi(\cdot)\qquad\text{on $\Sigma$.}
  \end{cases}
\end{displaymath}

%
%

\bibliographystyle{alpha}

\end{document}